\documentclass[11pt]{amsart}
\usepackage[utf8]{inputenc}
\usepackage{amsfonts, amsmath, amssymb, amsthm, color, bookmark,graphicx,subfig, mathabx,  tikz, tikz-cd,  cleveref, enumitem}

\usepackage{comment}
\usepackage{xy}
\xyoption{all}

\usetikzlibrary{calc}
\usetikzlibrary{calc, decorations.pathmorphing}

\allowdisplaybreaks

\usepackage[centertags]{amsmath}

\usepackage{mathrsfs}

\newtheorem{theorem}{Theorem}[section]                                          
\theoremstyle{plain}
            \newtheorem{lemma}[theorem]{Lemma}
            \newtheorem{corollary}[theorem]{Corollary}                   
            \newtheorem{proposition}[theorem]{Proposition}

\theoremstyle{remark}
            \newtheorem{remark}[theorem]{Remark}
\theoremstyle{definition}
            \newtheorem{definition}[theorem]{Definition}
            
            \newtheorem{example}[theorem]{Example}

\newcounter{theoremintro} 
\theoremstyle{plain}

\newtheorem{introtheorem}[theoremintro]{Theorem}
\newtheorem{introcorollary}[theoremintro]{Corollary}

\renewcommand{\ll }{\langle\hspace{-.7mm}\langle }
\newcommand{\rr }{\rangle\hspace{-.7mm}\rangle }

\newcommand{\bg}{\overline{G}}
\newcommand{\E}{\underline{E}}
\newcommand{\bh}{\overline{H}}

\newcommand{\bhc}{\overline{\mathcal H}}
\newcommand{\Z}{\mathbb{Z}}
\newcommand{\N}{\mathcal{N}}

\newcommand{\h}{\mathcal H}

\DeclareMathOperator{\cd}{cd}

\newcommand{\bsl}{\bigoplus_{\lambda\in\Lambda}}
\newcommand{\chg}{\mathrm{CoInd}^{\overline{G}}_{\overline{H}_{\lambda}}}
\newcommand{\ihg}{\mathrm{Ind}^{\overline{G}}_{\overline{H}_{\lambda}}}
\newcommand{\pl}{\prod_{\lambda\in\Lambda}}

\newcommand{\m}{\ll\mathcal{N}\rr}
\newcommand{\bhl}{\overline{H}_{\lambda}}
\newcommand{\bha}{\overline{H}_{\alpha}}
\newcommand{\bhb}{\overline{H}_{\beta}}

\newcommand{\ab}{\mathrm{ab}}
\newcommand{\bZ}{\mathbb{Z}}
\newcommand{\bC}{\mathbb{C}}
\newcommand{\bR}{\mathbb{R}}

\newcommand{\s}[1]{\langle #1 \rangle}

\newcommand{\Compacts}{\mathcal{K}}
\newcommand{\Linears}{\mathcal{L}}

\newcommand{\bX}{\overline{X}}
\newcommand{\bY}{\overline{Y}}

\newcommand{\tifrac}[2]{%
  \raisebox{0.5ex}{$#1$}\!/\!\raisebox{-0.5ex}{$#2$}%
} 

\newcommand{\llrr}[1]{\langle\!\langle #1 \rangle\!\rangle}

\usepackage[textsize=tiny]{todonotes}
\setuptodonotes{fancyline, color=blue!30}

\newcommand{\colim}{\mathrm{colim}}

\title[Dehn fillings, $K$-homology, and the Baum--Connes conjecture]{Dehn fillings, equivariant homology, \\and the Baum--Connes conjecture}
\author{Shintaro Nishikawa \& Nansen Petrosyan}

\address{S. Nishikawa: School of Mathematical Sciences, University of Southampton, University Road, Southampton, SO17 1BJ, UK}
\email{s.nishikawa@soton.ac.uk}

	\address{N. Petrosyan: School of Mathematical Sciences, University of Southampton, University Road, Southampton, SO17 1BJ, UK}
	\email{n.petrosyan@soton.ac.uk}

\date{\today}

\begin{document}

\begin{abstract}
We establish a connection between Cohen–Lyndon triples and equivariant homology theory, with a focus on the Baum–Connes conjecture.

In the first part of this work, we establish an excision sequence for the classifying spaces for proper actions in equivariant homology theories. This provides a direct link between Cohen–Lyndon triples and the left-hand side of the Baum–Connes conjecture.

Independently of these, we prove that the Baum–Connes conjecture with coefficients (BCC) with finite wreath products holds for all discrete hyperbolic groups, building on the monumental work of Lafforgue. Combining this with  permanence properties and the work of Dahmani–Guirardel–Osin on relatively hyperbolic groups, we identify a broad class of groups, including all lattices in simple Lie groups of real rank one that satisfy the BCC with finite wreath products. This significantly broadens the scope of our first result, as Cohen–Lyndon triples arise naturally in the context of relatively hyperbolic groups, thereby connecting both sides of the Baum–Connes conjecture.
\end{abstract}

\maketitle

\vspace{-1cm}

\tableofcontents

\section{Introduction}

In the 1970s, Thurston dramatically transformed the study of $3$-manifolds by introducing his \emph{Geometrization Conjecture}. As supporting evidence, he proved that many non-Haken $3$-manifolds satisfy the conjecture \cite{thurston1983three}, using the concept of \emph{Dehn surgery}, a two-step procedure that modifies a $3$-manifold by first removing a solid torus and then gluing it back in a different way.

The second step of this process, called \emph{Dehn filling}, begins with a $3$-manifold $M$ with toral boundary and produces a new manifold by gluing a solid torus to $M$ via a boundary identification. Topologically distinct ways of gluing a solid torus are parametrized by free homotopy classes of essential simple closed curves on $\partial M$ (the image of the meridian circle of the solid torus under the identification), called \emph{slopes}. A celebrated theorem of Thurston asserts that if the interior of $M$ admits a hyperbolic structure, then all but finitely many of its Dehn fillings also admit a hyperbolic structure.

\begin{theorem}[{Thurston \cite{thurston1983three}}]\label{thm. thurston hyperbolic Dehn filling}
Let $M$ be a compact orientable $3$-manifold with boundary a torus, and with interior admitting a complete finite-volume hyperbolic structure. Then for all but finitely many slopes $s$ on $\partial M$, $M_s$ admits a hyperbolic structure. 
\end{theorem}

There is an analogous construction in group theory, known as \emph{(group-theoretic) Dehn filling}. At an abstract level, this can be formalised as follows: given a group $G$, a subgroup $H \leq G$, and a normal subgroup $N \trianglelefteq H$, the Dehn filling associated to the triple $(G,H,N)$ is the quotient $G/\ll N \rr$, where $\ll N \rr$ denotes the normal closure of $N$ in $G$.

This algebraic framework mirrors the topological notion of Dehn filling, as seen through the Seifert--van Kampen theorem. Specifically, let \(G = \pi_1(M)\), \(H\) be the image of \(\pi_1(\partial M)\) in \(\pi_1(M)\), and \(N\) the image of \(\pi_1(\partial D^2)\). The gluing process in topology then corresponds to the group-theoretic construction.

However, in practice, group-theoretic Dehn fillings are studied not in this general form, but rather in settings that reflect the presence of negative curvature. Most notably, $G$ is either assumed to be \emph{hyperbolic relative} to $H$, or more generally $H$ is \emph{hyperbolically embedded} in $G$ (denoted $H\hookrightarrow_h G$) in the sense of Dahmani--Guirardel--Osin \cite{dahmani2017hyperbolically}. These frameworks provide the geometric control necessary to prove analogues of Thurston’s hyperbolic Dehn filling theorem in the algebraic setting.  

The first such results were established for relatively hyperbolic groups by Osin \cite{osin2007peripheral} and independently by Groves--Manning \cite{groves2008dehn}. These were later extended by Dahmani--Guirardel--Osin, who proved a generalisation of the main results of \cite{osin2007peripheral, groves2008dehn} in the setting of hyperbolically embedded subgroups \cite{dahmani2017hyperbolically}.

Dehn filling is a fundamental tool in group theory. It appears, for instance, in the solution of the Virtual Haken Conjecture \cite{agol2013virtual}, in the study of the Singer Conjecture \cite{PS_L2}, the Farrell--Jones Conjecture \cite{antolin2017farrell}, and the isomorphism problem \cite{DG18} for relatively hyperbolic groups, as well as in the construction of purely pseudo-Anosov normal subgroups of mapping class groups \cite{dahmani2017hyperbolically}. Further applications of Dehn fillings can be found in \cite{agol2016alternate, groves2016boundaries, sun2019cohomologyii}.

For simplicity of exposition, some of the results in this introduction are stated for a single peripheral subgroup, though they hold more generally for multiple peripherals.

\subsection{Cohen--Lyndon quotients}
A central focus of this article is the so-called \emph{Cohen--Lyndon property}, which originally arose in the study of asphericity conditions in group presentations. Roughly speaking, following the previous notation, the triple $(G,H,N)$ is said to satisfy the \emph{Cohen--Lyndon property} if the normal closure $\ll N \rr$ decomposes as a free product of certain conjugates of $N$. In this case, we say that $\bg := G / \ll N \rr$ is a \emph{Cohen--Lyndon quotient}. This property is closely related to the condition that $H$ is hyperbolically embedded in $G$. Indeed, a key result of Sun \cite{sun2018cohomologyi} states that for sufficiently deep normal subgroups $N \triangleleft H$, the triple $(G,H,N)$ satisfies the Cohen--Lyndon property.

The Cohen--Lyndon property has been established for many classes of groups, including one-relator groups \cite{cohen1963free}, certain groups with relative one-relator presentations \cite{EH87}, and more recently, classical and cubical small cancellation groups \cite{AD25, Arenas24}. 

Most recently, the second author and Bin Sun, using this property, obtained structural results on the cohomology and $L^2$-Betti numbers of Dehn filling quotients $\bg$ \cite{sun2019cohomologyii, PS_L2}. In \cite{PS_L2}, they constructed a contractible free $\bg$-CW-complex $E\bg$ associated to the Cohen--Lyndon quotient $\bg$, called the \emph{Dehn filling space}. This space proved extremely useful for applications, as it was tailored for topological excision. Here, we extend this construction and build a classifying space $\E \bg$ for proper actions (\Cref{thm: proper Dehn filling space}), which we refer to as a \emph{Dehn filling space for proper actions}. This new space retains similar excision properties. The key point is that the presence of torsion in $\bg$ rules out the existence of finite-dimensional models for $E\bg$, whereas $\E \bg$ is well-suited for constructing cocompact models in the cases of interest.

To construct $\E \bg$, we require a certain rigidity result on finite subgroups of Cohen--Lyndon quotients $\bg$, which we refer to as property {\rm \hyperlink{cond:fin}{(FIN)}}. We verify that this property holds under general conditions and for all our applications (\Cref{theorem FIN}).
\subsection{Equivariant homology sequence}
The general homological framework most suitable for our results is the equivariant homology theory defined axiomatically in \cite{BEL08} (see also \cite[Chapter 12]{Lueck25}). Fix a discrete group $G$ and a ring $S$. A $G$-homology theory $\h_\ast^G$ with values in $S$-modules is a collection of covariant functors $\h_n^G$ from the category of $G$-CW pairs to the category of $S$-modules, indexed by $n \in \mathbb{Z}$, and satisfying certain standard axioms. An equivariant homology theory $\h_\ast^{?}$ over a group $\Gamma$ (\cite[Definition 1.3]{BEL08}) assigns to every group $(G, \xi\colon G\to \Gamma)$ over $\Gamma$, a $G$-homology theory $\h_\ast^{G}$ (more precisely, $\h_\ast^{(G, \xi)}$), and comes with the induction structure for any homomorphism $\alpha\colon (G_1, \xi_1) \to (G_2, \xi_2)$ of groups over $\Gamma$ (see Section \ref{sec_homology}).

By applying the excision axiom to the Dehn filling space for proper actions $\E \bg$ constructed in \Cref{thm: proper Dehn filling space}, we obtain:

\begin{introtheorem}[Theorem \ref{thm_CLP_diagram}]\label{intro_thm_CLP_diagram}
Let $(G, H, N)$ be a Cohen--Lyndon triple satisfying property \hyperlink{cond:fin}{(FIN)}. Let $\h^{?}_{\ast}$ be any equivariant homology theory over $\bg$. Then we have the following commutative diagram with exact rows:
\begin{equation*}\label{eq_diag_equiv}
\xymatrix{
\cdots \ar[r] & \h_\ast^{H}(EH) \ar[r] \ar[d] & \h_\ast^{G}(EG) \ar[r] \ar[d] 
& \h_\ast^{\bg}(M_\phi, Z) \ar[d]^{\cong} \ar[r] & \cdots \\
\cdots \ar[r] & \h_\ast^{\bh}(\E \bh) \ar[r] & \h_\ast^{\bg}(\E \bg) \ar[r] 
& \h_\ast^{\bg}(Y, M_\theta) \ar[r] & \cdots.
}
\end{equation*}
\end{introtheorem}

Given a group $\Gamma$ and a ring $R$ (with involution) equipped with a $\Gamma$-action, \cite[Theorem 5.1]{BEL08} provides the following equivariant homology theories:
\[
H_\ast^{?}(-;\mathrm{K_R}),\quad H_\ast^{?}(-;\mathrm{KH_R}),\quad H_\ast^{?}(-;\mathrm{L_R^{\langle -\infty \rangle}}).
\]
These theories take values in $\mathbb{Z}$-modules over the group $\Gamma$ and satisfy the properties listed in \cite[Theorem 5.1]{BEL08}. Theorem \ref{intro_thm_CLP_diagram} applies to any of these equivariant homology theories.

Regarding the topological $K$-homology $H_\ast^{?}(-;\mathrm{K^{\mathrm{top}}_{A,r}})$ for a $C^*$-algebra $A$ with a $\Gamma$-action (provided by \cite[Theorem 5.1]{BEL08}), to the best of the authors' knowledge, the currently known construction (\cite{M04}, \cite{Kranz21}, \cite{BEL21}) yields an ``equivariant" homology theory whose induction homomorphisms are defined only for injective group homomorphisms. For this reason, we establish the analogue of Theorem \ref{intro_thm_CLP_diagram} for Kasparov's analytic $K$-homology separately in Section \ref{sec:ex_Khom} (Theorem \ref{thm_CLP_diagram_Khom}).

In \Cref{sec: CLA}, we apply Theorem \ref{thm_CLP_diagram_Khom} to compute the topological $K$-theory (the left-hand side of the BC) of all Cohen--Lyndon aspherical groups, which include one-relator groups and many small-cancellation groups (\Cref{thm: CLA}). As an application, we obtain a complete computation of the $K$-theory of the reduced group $C^*$-algebra for all groups satisfying  either $C(7)$ or $C'(1/4)-T(4)$ small-cancellation presentations (see Remark \ref{rmk:CLA_applications}). We expect that the flexibility of allowing arbitrary coefficients will enable further computations and applications in the future.

\subsection{Relatively hyperbolic groups and the Baum--Connes conjecture}

Bartels \cite{Bartels17} proved the Farrell--Jones conjecture for relatively hyperbolic groups whose peripheral subgroups satisfy the Farrell--Jones conjecture (see also \cite[Theorem 16.21]{Lueck25}). Later, an alternative proof was obtained by Antolín--Coulon--Gandini \cite{antolin2017farrell}, assuming in addition that the peripheral subgroups are residually finite; their argument extends the Dehn filling theorem following Dahmani--Guirardel--Osin \cite{dahmani2017hyperbolically} and combines it with permanence results and the Farrell--Jones conjecture for hyperbolic groups \cite{BLR, BL12}.

We adapt the approach of Antolín--Coulon--Gandini to obtain a result on the Baum--Connes conjecture for relatively hyperbolic groups. To maximise the benefit of this strategy, we introduce the \emph{Baum--Connes conjecture with coefficients with finite wreath products} (abbreviated the \emph{BCC with finite wreath products}), following the formulation of the Farrell--Jones conjecture with finite wreath products \cite[Conjecture 13.27]{Lueck25}:

\begin{definition}[The Baum–Connes Conjecture with Coefficients with finite wreath products] We say that a (countable) discrete group $G$ satisfies the Baum–Connes conjecture
with coefficients with finite wreath products if for any finite group $F$ the wreath product $G\wr F$ satisfies the Baum–Connes conjecture with coefficients.
\end{definition}

Note that the BCC with finite wreath products for $G$ implies the Baum--Connes conjecture with coefficients (BCC) for $G$ (the BCC with finite wreath products is, a priori, stronger than the BCC).

In \cite{Lueck25}, a systematic implementation of this finite wreath products version is suggested (see \cite[Remark 16.10]{Lueck25}). The main advantage is the stability under finite extensions; in fact, the BCC with finite wreath products for $G_1$ and for $G_2$ are equivalent whenever $G_1$ and $G_2$ are commensurable. Consequently, the stability under group extensions of the BCC with finite wreath products is better than that of the BCC: for any group extension $1\to N\to G\to G/N\to1$, the BCC with finite wreath products holds for $G$, if it holds for $N$ and $G/N$ (see Section \ref{sec: permanence} for all these). All a-T-menable groups satisfy the BCC with finite wreath products by the Higson--Kasparov theorem \cite{HK01} (see Example \ref{ex_HK}).

We verify the BCC with finite wreath products (and in particular, the BCC) for a broad class of relatively hyperbolic groups:

\begin{introtheorem}[Theorem \ref{thm_BCC_rel}]\label{intro_thm. BCC permanence}
Let $G$ be a countable discrete group hyperbolic relative to a family $\{H_{\lambda}\}_{\lambda \in \Lambda}$ of groups, with $|\Lambda|<\infty$. Suppose for any $\lambda\in \Lambda$, $H_\lambda$ is residually finite and satisfies the BCC with finite wreath products. Then $G$ satisfies the BCC with finite wreath products.
\end{introtheorem}

To establish Theorem \ref{intro_thm. BCC permanence}, we combine the Dehn filling theorem of \cite{dahmani2017hyperbolically} with permanence results (see \Cref{sec: permanence}) and the following extension of Lafforgue's theorem \cite{Lafforgue12} (the BCC for hyperbolic groups), proved in \Cref{sec_hyp_BCC}:

\begin{introtheorem}[Theorem \ref{thm_hyp_BCC_wr}]\label{intro_thm_hyp_BCC_wr}
The BCC with finite wreath products holds for any hyperbolic group.
\end{introtheorem}

We prove Theorem \ref{intro_thm_hyp_BCC_wr} by combining the work of Lafforgue on the BCC for hyperbolic groups \cite{Lafforgue12} with the following main technical theorem and its corollary. We refer the reader to Section \ref{sec_hyp_BCC} for all the terminologies that are involved:

\begin{introtheorem}[Theorem \ref{thm_gamma_wr}]\label{intro_thm_gamma_wr} Let $G$ be a second countable, locally compact group. Let $\ell$ be any length function on $G$. Suppose $G$ has a gamma element $\gamma_G \in KK_G(\bC, \bC)$ in the Meyer--Nest sense. Then, for any second countable, compact group $F$ and for any finite $F$-set $\Omega$, the wreath product $G\wr_\Omega F$ has a gamma element $\gamma_{G\wr_\Omega F}$ in the Meyer--Nest sense. Moreover, if $\gamma_G =1_G$ in $KK_{G, \ell}(\bC, \bC)$, then $\gamma_{G\wr_\Omega F} =1_{G\wr_\Omega F}$ in $KK_{G\wr_\Omega F, \tilde\ell}(\bC, \bC)$ where $\tilde \ell$ is the length function on $G\wr_\Omega F$ defined by $\tilde \ell(g) \coloneq \sum_{\omega \in \Omega}\ell(g_\omega)$ for $(g_\omega)_{\omega \in \Omega}$ in $\prod_\Omega G \subset G \wr_\Omega F$ and $\tilde \ell (f) \coloneq 0$ for $f \in F$.
\end{introtheorem}

\begin{introcorollary}[Corollary \ref{cor_gamma_to_wreath}]\label{intro_cor_gamma_to_wreath} Let $G$ be a second countable, locally compact group $G$ that acts continuously and isometrically on a metric space $(X, d)$ of finite asymptotic dimension with linear control. Let $x_0\in X$ and let $\ell$ be the length function defined by $\ell(g)=d(x_0, gx_0)$. Suppose that $G$ has a gamma element $\gamma_G$ in $KK_G(\bC, \bC)$ in the Meyer--Nest sense, and that $\gamma_G$ satisfies the following condition: for any $s>0$, there is $C\geq 0$ such that $\gamma_G=1_G$ in $KK_{G, s\ell+C}(\bC, \bC)$. Then, the BCC with finite wreath products holds for $G$.
\end{introcorollary}

To the best of our knowledge, the BCC for all groups that are commensurable to a product of hyperbolic groups, which is implied by Theorem \ref{intro_thm_hyp_BCC_wr} and the permanence properties (see Section \ref{sec: permanence}), had remained open in general, and it is not clear if it is implied by the previously-known permanence theorems and the BCC for hyperbolic groups (Lafforgue, \cite{Lafforgue12}). It would be desirable to prove the BCC (with finite wreath products) for all groups acting isometrically, continuously, and properly on a product of weakly geodesic, uniformly locally finite hyperbolic metric spaces, but this is beyond the scope of this article. 

As applications of Theorem \ref{intro_thm. BCC permanence}, we obtain the following corollaries. We believe these are new even when considering the BCC (without finite wreath products) but it is worth mentioning that the  versions without finite wreath products of the following corollaries can be derived without Theorem \ref{intro_thm_hyp_BCC_wr} from a combination of known permanence properties, the BCC for hyperbolic groups (Lafforgue, \cite{Lafforgue12}), and the Higson--Kasparov theorem \cite{HK01}, combined with the Dehn filling theorem of \cite{dahmani2017hyperbolically}. Establishing that a group $G$ satisfies the BCC with finite wreath products is particularly advantageous because this version of the conjecture is stable under group extensions.

\begin{introcorollary}[Corollary \ref{cor_negative_curvature}]\label{intro_cor_negative_curvature}
Let $G$ be the fundamental group of a complete, noncompact, finite-volume Riemannian manifold with pinched negative sectional curvatures. Then $G$ satisfies the BCC with finite wreath products.
\end{introcorollary}

\begin{introcorollary}[Corollary \ref{cor_lattice_BCC}]\label{intro_cor_lattice}
Any lattice in a simple Lie group of real rank one with finitely many connected components satisfies the BCC with finite wreath products.
\end{introcorollary}

Corollaries \ref{intro_cor_negative_curvature}, \ref{intro_cor_lattice} generalise the result of Chartterji--Ruane \cite[Corollary 0.3]{CR05} that   confirmed the Baum--Connes conjecture without coefficients (BC) for these groups.


\begin{introcorollary}[Corollary \ref{cor_Dehn_filling_mfld}]\label{intro_cor_Dehn_filling_mfld} Let $\overline{M}$ be a compact oriented $n$-manifold with nilmanifold boundary components  such that the centre of the fundamental group of each boundary component is of rank at least $2$. Suppose the interior of $\overline{M}$ admits a Riemannian metric with a complete pinched negative sectional curvature and finite volume.   If $M_T$  is a sufficiently deep Dehn filling manifold of $\overline{M}$, then $M_T$ is a closed oriented aspherical manifold and $\pi_1(M_T)$ satisfies the BCC with finite wreath products.
\end{introcorollary}

It is important to note that in Theorem \ref{intro_cor_Dehn_filling_mfld}, $M_T$ is constructed using sufficiently deep normal subgroups of the peripherals of $\pi_1(\overline{M})$ so that the resulting triple satisfies the Cohen--Lyndon property. This is crucial in showing that $M_T$ is aspherical. 

Our next application concerns aspherical Einstein manifolds. Recall that Thurston’s Hyperbolic Dehn Filling Theorem \ref{thm. thurston hyperbolic Dehn filling} applies only in dimension~$3$. The obstruction in higher dimensions is straightforward: if $M$ is a manifold with torus boundary and $\dim M \geqslant 4$, then for any Dehn filling $M_T$ of $M$, the fundamental group $\pi_1(M_T)$ necessarily contains a $\mathbb{Z}^2$ subgroup. Consequently, $M_T$ cannot support a hyperbolic Riemannian metric. Nevertheless, Anderson~\cite{And06} showed that all sufficiently large Dehn fillings of $M$ do admit an Einstein metric.

\begin{introcorollary}[Corollary \ref{cor_Einstein_mfld}] \label{intro_cor_Einstein_mfld}Let $M$ be a complete, finite-volume hyperbolic manifold of dimension at least three, with toral cusps. Then, any
  $M_T$ obtained by a sufficiently large Dehn filling of the cusps of $M$ is an Anderson-type Einstein manifold with fundamental group satisfying the BCC with finite wreath products.
\end{introcorollary}

Lastly, we present an application of Theorem \ref{intro_thm. BCC permanence} to quotients of the mapping class group of the closed genus 2 surface by sufficiently high powers of all Dehn twists. These types of quotients of mapping class groups of surfaces have recently attracted significant attention due to their hierarchically hyperbolic structure \cite{BHMS24, DHS21, BHS17}.

\begin{introcorollary}[Corollary \ref{cor: MCG quotients}] \label{intro_cor: MCG quotients}
There exists an integer $K_0 \geq 1$ such that for all non-zero multiples $K$ of $K_0$, the quotient $\mathrm{MCG}(\Sigma_2) / \mathrm{DT}_K$ satisfies the BCC with finite wreath products, where $\mathrm{DT}_K$ denotes the normal subgroup generated by the $K$-th powers of all Dehn twists.
\end{introcorollary}

Note that it is still open whether $\mathrm{MCG}(\Sigma_2)$ satisfies the BC: although the mapping class groups have the rapid decay property (RD) \cite{BJM11}, it is not known whether they lie in Lafforgue's class $\mathcal{C}'$ (see the paragraph on the BC in \cite[Section 1]{Cha17}). Combining Corollary \ref{intro_cor: MCG quotients} with known permanence results, we obtain:

\begin{introcorollary} [Corollary \ref{cor: MCG BCC equal}] \label{intro_cor: MCG BCC equal} There is an integer $K_0 \geq 1$ such that   $\mathrm{MCG}(\Sigma_2)$ satisfies the BCC with finite wreath products if and only if $\mathrm{DT}_K$ satisfies BCC with finite wreath products for some non-zero multiple $K$ of $K_0$.
\end{introcorollary}

\section{Dehn filling space for proper actions}

\subsection{Dehn filling}

Dehn filling is a process that takes as input a group~$G$, a family of subgroups~$\{H_\lambda\}_{\lambda\in\Lambda}$ of~$G$, and a family of normal subgroups~$\{N_\lambda \lhd H_\lambda\}_{\lambda\in\Lambda}$, and produces the quotient group
$$
G/\ll \bigcup_{\lambda\in\Lambda} N_\lambda \rr_G.
$$
Here and throughout, given a subset~$S$ of a group~$G$, we denote by~$\ll S \rr_G$ the normal closure of~$S$ in~$G$. When the ambient group is understood, we may simply write~$\ll S \rr$. We also let
$$\mathcal{N}=\bigcup_{\lambda\in\Lambda}N_{\lambda}, \; \h =\{H_\lambda\}_{\lambda\in\Lambda},  \; \bg=G/\m, \; \bhl=H_{\lambda}/N_{\lambda}, \; \bhc =\{\bhl\}_{\lambda \in \Lambda}.$$

The general setting of Dehn filling is too broad to yield meaningful conclusions. In practice, one often assumes that $G$ is hyperbolic relative to a family of subgroups $\{H_\lambda\}_{\lambda \in \Lambda}$ \cite[Definition~3.6]{dahmani2017hyperbolically}. More generally, one may assume that $\{H_\lambda\}_{\lambda \in \Lambda}$ is hyperbolically embedded in $G$ (denoted $\{H_\lambda\}_{\lambda \in \Lambda} \hookrightarrow_h G$)~\cite[Definition~4.25]{dahmani2017hyperbolically}. This notion extends relative hyperbolicity~\cite[Proposition~4.28]{dahmani2017hyperbolically} and provides a natural framework in which to study Dehn fillings. Under this assumption, it has been shown that for most choices of $\{N_\lambda \lhd H_\lambda\}_{\lambda\in\Lambda}$, the resulting quotient $\bg$ has desired properties.

We adopt the following terminology, which aligns with that used in $3$-manifold theory.

\begin{definition}\label{def. sufficiently deep}
Let $H\leq G$ be groups, and let $\mathcal{P} = \mathcal{P}(G,H,N)$  be properties defined for all normal subgroups~$N \lhd H$. We say that \emph{$\mathcal{P}$ holds for sufficiently deep normal subgroups~$N \lhd H$} if there exists a finite set~$\mathcal{F} \subset H \smallsetminus \{1\}$ such that $\mathcal{P}$ holds whenever~$N \cap \mathcal{F} = \emptyset$. 

If $\{N_{\lambda}\lhd H_{\lambda}\}_{\lambda\in\Lambda}$ is a family of sufficiently deep normal subgroups, then we say that $\bg$ is a \emph{sufficiently deep Dehn filling quotient of $G$}.
\end{definition}



\begin{theorem}[{Dahmani--Guirardel--Osin \cite[Theorem~7.19]{dahmani2017hyperbolically}, \cite[Theorem~1.1]{osin2007peripheral}}]\label{thm. simple Dehn filling}
Let $G$ be a group with a family of subgroups $\h\hookrightarrow_h G$. Then for all sufficiently deep normal subgroups $\{N_{\lambda}\lhd H_{\lambda}\}_{\lambda\in\Lambda}$, the natural homomorphism $\bhl\rightarrow\bg$ is injective for $\lambda\in\Lambda$ and  $\bhc \hookrightarrow_h \bg$.

Moreover, when $G$ is hyperbolic relative to $\h$, then all sufficiently deep Dehn filling quotients $\bg$ are hyperbolic relative to $\bhc$.
\end{theorem}

\subsection{Cohen--Lyndon property} The notion of a Cohen--Lyndon triple was introduced in~\cite{sun2018cohomologyi}, though the underlying idea dates back to a paper of Cohen and Lyndon~\cite{cohen1963free}, which motivates the terminology. This property plays a crucial role throughout the paper.

\begin{definition}
   A group triple $(G,\{H_{\lambda}\}_{\lambda\in\Lambda},\{N_{\lambda}\}_{\lambda\in\Lambda})$ is called a \textit{Cohen--Lyndon triple} if there exist left transversals $T_{\lambda}$ of $H_{\lambda}\m$ in $G$ such that
$$\m =\Asterisk_{\lambda\in\Lambda,t\in T_{\lambda}} tN_{\lambda}t^{-1}.$$
\end{definition}

\begin{proposition}\label{prop CL} Below is a list of examples of Cohen--Lyndon triples:
    \begin{enumerate}
        \item {\rm(Cohen--Lyndon, \cite{cohen1963free})} Let $G$ be a free group, $H$ a maximal cyclic subgroup, and $N=\langle x \rangle$, $\forall x\in H\smallsetminus \{1\}$. Then $(G, H, N)$ is a Cohen--Lyndon triple.\\
        \item {\rm(Arenas--Duda, \cite{AD25})} Let $P = \langle s_1, \dots, s_n \mid r_1, \dots, r_k \rangle $ be either $C(6)$, $C(4)-T(4)$, or $C(3)-T(6)$ small-cancellation presentation. Let $G= \langle s_1, \dots, s_n \rangle $ and $H_i$ be the normaliser of $N_i=\langle r_i \rangle $ in $G$. Then, $(G,\{H_i\}_{i=1}^k,\{N_i\}_{i=1}^k)$ is a Cohen--Lyndon triple.\\
         \item {\rm(Edjvet--Howie, \cite{EH87})} Let $G=A\ast B$ where $A, B$ are locally indicable groups, $N=\langle x \rangle$, where $x$ does not lie in either a conjugate of $A$ or $B$ in $G$ and $H$ is the centraliser of $N$ in $G$. Then $(G, H, N)$ is a Cohen--Lyndon triple.\\
        \item \label{prop CL_4} {\rm(Sun, \cite[Theorem 5.1]{sun2018cohomologyi})}\label{thm. cl}
Let $G$ be a group with a family of subgroups $\{H_{\lambda}\}_{\lambda\in\Lambda}\hookrightarrow_{h}G$. Then for all sufficiently deep normal subgroups $\{N_{\lambda}\lhd H_{\lambda}\}_{\lambda\in\Lambda}$, $(G,\{H_{\lambda}\}_{\lambda\in\Lambda},\{N_{\lambda}\}_{\lambda\in\Lambda})$ is a Cohen--Lyndon triple.
    \end{enumerate}
\end{proposition}

\subsection{Dehn filling space}

Let $(G,\h,\{N_{\lambda}\}_{\lambda\in\Lambda})$ be a Cohen--Lyndon triple. The goal of this section is to recall a specific $K(\bg,1)$-space from \cite{PS_L2}. 

First, we note that the natural map $\bhl\rightarrow \bg$ is injective (see e.g., \cite[Lemma 6.4]{sun2018cohomologyi}) for each $\lambda \in \Lambda$, and below we will view $\bhl$ as a subgroup of $\bg$. Let $BG$ (resp. $BH_{\lambda}, B\bhl$) be a $K(G,1)$ (resp. $K(H_{\lambda},1), K(\bhl,1)$) CW-complex.

\begin{definition}[\cite{PS_L2}]\label{Dehn_space}
For each $\lambda \in \Lambda$, there is a classifying map $\phi\colon \bigsqcup_{\lambda \in \Lambda} B H_{\lambda}\rightarrow BG$ (resp. $\psi\colon \bigsqcup_{\lambda \in \Lambda} B H_{\lambda}\rightarrow \bigsqcup_{\lambda \in \Lambda} B\bhl$) induced by the inclusions $H_{\lambda}\hookrightarrow G$ (resp. the quotient maps $q\colon H_{\lambda}\rightarrow \bhl$). Let $X$  be the complex obtained by gluing the mapping cylinders $M_{\phi},M_{\psi}$ along their common subcomplex $\bigsqcup_{\lambda \in \Lambda} B H_{\lambda}$. The complex $X$ is called the {\it Dehn filling space} associated to the triple  $(G,\h,\{N_{\lambda}\}_{\lambda\in\Lambda})$.  
\end{definition}

\begin{theorem}[\cite{PS_L2}]\label{thm. topology}
The Dehn filling space associated to a Cohen--Lyndon triple $(G,\h,\{N_{\lambda}\}_{\lambda\in\Lambda})$ is a $K(\bg,1)$-complex.
\end{theorem}

By combining the above theorem with \Cref{prop CL} (\ref{thm. cl}) we get the following:

\begin{corollary}[\cite{PS_L2}]\label{cor. topology}
Let $G$ be a group with a family of subgroups $\h\hookrightarrow_{h}G$. Then, for all sufficiently deep normal subgroups $\{N_{\lambda}\lhd H_{\lambda}\}_{\lambda\in\Lambda}$, the Dehn filling space is a $K(\bg,1)$-complex.
\end{corollary}

An excision argument on the Dehn filling space yields the following application.

\begin{corollary}[{\cite[Theorem A (ii)]{sun2019cohomologyii}}]\label{cor. alg excision} Let $(G,\h,\{N_{\lambda}\}_{\lambda\in\Lambda})$ be a Cohen--Lyndon triple. Then, there is natural isomorphism induced by the quotient maps $G\to \bg$ and  $H_{\lambda}\to \bhl$,
    $$H^n(G, \h ; A)\cong H^n(\bg, \bhc ; A),  \; \; \forall n\geq 0.$$
\end{corollary}

%
%

\subsection{Dehn filling space for proper actions} \label{subsec: Dehn filling proper}
Recall, for a discrete group $\Gamma$, a \emph{proper $\Gamma$-CW-complex} is a $\Gamma$-CW-complex with only finite cell stabilisers. A proper $\Gamma$-CW-complex $Y$ is said to be {\it a model for a classifying space for proper actions $\E\Gamma$}, if for any finite subgroup $F \leq \Gamma$, the fixed point set $Y^F$ is contractible. Such a complex $Y$ always exists and is unique up to equivariant homotopy equivalence.

In this section, we adapt the construction of the Dehn filling space to construct a cellular model of the classifying space $\E \bg$ for proper actions when  $(G,\h,\{N_{\lambda}\}_{\lambda\in\Lambda})$ is a Cohen--Lyndon triple. We say that {\it $\bg$ has property} (FIN) or equivalently, the triple $(G,\h,\{N_{\lambda}\}_{\lambda\in\Lambda})$ {\it satisfies property} (FIN) if 

\begin{itemize}
    \item[(\hypertarget{cond:fin}{FIN})]  for every finite subgroup $F \leq \bg$, there exist $\alpha \in \Lambda$ and $g \in \bg$ such that 
\[
F \leq g \bha g^{-1}
\quad \text{and} \quad
F \cap h\bhb h^{-1} = \langle 1\rangle
\]
for all $\beta \neq\alpha $, and for $\beta = \alpha$ if $h \notin g \bha$.
\end{itemize}

The next result shows that property {\rm \hyperlink{cond:fin}{(FIN)}} occurs naturally as a by-product of the Cohen--Lyndon property. First, we need a theorem of Serre  whose proof can be found in \cite{Hueb79}.

\begin{theorem}[{Serre, \cite{Hueb79}}]\label{Serre} Let $\Gamma$ be a group and $\{\Gamma_i\}_{i \in I}$ a family of subgroups such that for all
$q \geq q_0$ the canonical map
\[
H^q(\Gamma; M) \to \prod_{i \in I} H^q(\Gamma_i; M)
\]
is an isomorphism for every $\Z\Gamma$-module $M$. Then for every finite subgroup $F \leq \Gamma$, there exist $i \in I$ and $\gamma \in \Gamma$ such that 
\[
F \leq \gamma\Gamma_i \gamma^{-1}
\quad \text{and} \quad
F \cap h\Gamma_j h^{-1} = \langle 1\rangle
\]
for all $j \neq i$, and for $j = i$ if $h \notin \gamma \Gamma_i$.
\end{theorem}

\begin{theorem}\label{theorem FIN} Let $(G,\h,\{N_{\lambda}\}_{\lambda\in\Lambda})$ be a Cohen--Lyndon triple. If $\cd G<\infty$, then $\bg$ has property {\rm \hyperlink{cond:fin}{(FIN)}}.
\end{theorem}
\begin{proof} From Proposition 4.4 of \cite{sun2019cohomologyii} (see also \Cref{cor. alg excision}), it follows that the canonical map
\[
H^q(\bg; M) \to \prod_{\lambda \in \Lambda} H^q(\bhl; M)
\]
is an isomorphism for every $\bg$-module $M$ and all $q\geq \cd G +2$. The claim now follows by direct application of \Cref{Serre}.
\end{proof}

As an immediate corollary, we obtain the follow applications.

\begin{corollary}\label{cor FIN} 
\label{prop FIN} The following Cohen--Lyndon triples satisfy property {\rm \hyperlink{cond:fin}{(FIN)}}:
    \begin{enumerate}
        \item Let $G$ be a free group, $H$ a maximal cyclic subgroup, and $N=\langle x \rangle$, $\forall x\in H\smallsetminus \{1\}$. Then $\bg$ has property {\rm \hyperlink{cond:fin}{(FIN)}}.\\
        \item Let $P = \langle s_1, \dots, s_n \mid r_1, \dots, r_k \rangle $ be either $C(6)$, $C(4)-T(4)$, or $C(3)-T(6)$ small-cancellation presentation. Let $G= \langle s_1, \dots, s_n \rangle $ and $H_i$ be the normaliser of $N_i=\langle r_i \rangle $ in $ \langle s_1, \dots, s_n \rangle $.  Then $\bg$  has property {\rm \hyperlink{cond:fin}{(FIN)}}.\\
        \item Let $G=A\ast B$ where $A, B$ are locally indicable groups and  $N=\langle x \rangle$, where $x$ does not lie in either a conjugate of $A$ or $B$ in $G$ and $H$ is the centraliser of $N$ in $G$.  If $\cd G<\infty$, then $\bg$ has property {\rm \hyperlink{cond:fin}{(FIN)}}.\\
        \item Let $G$ be a group with a family of subgroups $\h\hookrightarrow_h G$. If $\cd G<\infty$, then any sufficiently deep Dehn filling quotient $\bg$ has property {\rm \hyperlink{cond:fin}{(FIN)}}.
    \end{enumerate}
    \begin{proof}All parts except (4) follow directly from \Cref{theorem FIN}. For part (4), for all sufficiently deep normal subgroups $\{N_{\lambda}\lhd H_{\lambda}\}_{\lambda\in\Lambda}$, the triple $(G,\h,\{N_{\lambda}\}_{\lambda\in\Lambda})$ satisfies the Cohen--Lyndon property. We can then apply \Cref{theorem FIN}.
    \end{proof}
\end{corollary}

\begin{remark} Parts (1) and (2) of \Cref{cor FIN} are classical results which can be found in \cite[\S II.5]{LS77} and \cite{Hueb79}, respectively. To the best of our knowledge, parts (3) and (4) appear to be new.
\end{remark}

\begin{definition}\label{proper_Dehn_space} Suppose $(G,\h,\{N_{\lambda}\}_{\lambda\in\Lambda})$ is a Cohen--Lyndon triple that satisfies property \hyperlink{cond:fin}{(FIN)}. The quotient maps $G\times_{H_{\lambda}} EH_{\lambda}\twoheadrightarrow \bg\times_{\bhl}N_{\lambda}\backslash EH_{\lambda}$ induce an isomorphism 
$$\ll \N \rr\backslash  \big(\bigsqcup_{\lambda \in \Lambda} G\times_{H_{\lambda}} EH_{\lambda}\big)\cong \bigsqcup_{\lambda \in \Lambda} \bg\times_{\bhl}(N_{\lambda}\backslash EH_{\lambda}).$$
Here, we use that $H_{\lambda}\cap \ll \N\rr = N_{\lambda}$, i.e.~ $\bhl \hookrightarrow \bg$, which is a consequence of the Cohen--Lyndon property \cite[Lemma 6.4]{sun2018cohomologyi}.
Let
\[\phi\colon  \bigsqcup_{\lambda \in \Lambda} \bg\times_{\bhl}(N_{\lambda}\backslash EH_{\lambda})\cong  \ll \N \rr\backslash  \big(\bigsqcup_{\lambda \in \Lambda} G\times_{H_{\lambda}} EH_{\lambda}\big)\xrightarrow{\psi} \ll \N \rr\backslash EG\]
and
\[\theta\colon\bigsqcup_{\lambda \in \Lambda}\bg\times_{\bhl}(N_{\lambda}\backslash EH_{\lambda})\rightarrow \bigsqcup_{\lambda \in \Lambda}\bg\times_{\bhl}\E\bhl\]
be the $\bg$-equivariant maps, where $\phi$ is the composition of the inverse of the above isomorphism with the map $\psi$ induced by the classifying maps  $G\times_{H_{\lambda}} EH_{\lambda} \to EG$ and $\theta$ is induced by the classifying maps $N_{\lambda}\backslash  EH_{\lambda} \to \E\bhl$. Denote by $M_{\phi}$ and $M_{\theta}$ the mapping cylinders of $\phi$ and $\theta$, respectively.
Let $Y$  be the complex obtained by gluing $M_{\phi}$ and $M_{\theta}$ along their common top (see \Cref{fig:proper_Dehn_space}). We call the complex $Y$ the {\it Dehn filling complex for proper actions} associated to the triple  $(G,\h,\{N_{\lambda}\}_{\lambda\in\Lambda})$.  
\end{definition}

\begin{figure}[ht!]
    \centering
\begin{tikzpicture}
  \coordinate (C1) at (0,0);
  \coordinate (C2) at (3.2,0);
  \coordinate (C3) at (6,0);
  
  \fill[cyan!10] 
  (0.98,0.2) -- (2.1,0.5) -- (2.1,-0.5) -- (0.98,-0.2) -- cycle;

\fill[cyan!10]
  (4.35,0.3) -- (5.2,0.1) -- (5.2,-0.1) -- (4.35,-0.3) -- cycle;

\filldraw[fill=gray!10] (C1) circle (1cm) node {\footnotesize{$\bg\times_{\bhl}\E\bhl$}};
\filldraw[fill=gray!10] (C2) circle (1.2cm) node {\footnotesize{$\bg\times_{\bhl}N_{\lambda}\backslash EH_{\lambda}$}};
\filldraw[fill=red!5] (C3) circle (.8cm) node {\footnotesize{$\ll \N \rr\backslash EG$}};

 \draw (.98,.2) -- (2.1,.5);
  \draw (.98,-.2) -- (2.1,-.5);

   \draw (4.35,.3) -- (5.2,0.1);
     \draw (4.35,-.3) -- (5.2,-0.1);

 \fill (0,1.1) circle (.7pt);
  \fill (0, -1.1) circle (.7pt);
   \fill (0.05,1.3) circle (.7pt);
  \fill (0.05, -1.3) circle (.7pt);
     \fill (0.1,1.5) circle (.7pt);
  \fill (0.1, -1.5) circle (.7pt);

   \fill (3.2,1.3) circle (.7pt);
  \fill (3.2, -1.3) circle (.7pt);
   \fill (3.25,1.5) circle (.7pt);
  \fill (3.25, -1.5) circle (.7pt);
     \fill (3.3,1.7) circle (.7pt);
  \fill (3.3, -1.7) circle (.7pt);

  \draw[<-] (-1, -2.5) -- (.8, -2.5);
  \draw[->] (1.7, -2.5) -- (3.2, -2.5);
  
  \node at (1.25, -2.5) {$M_{\theta}$};


  \draw[<-] (3.3, -2.5) -- (4.8, -2.5);
  \draw[->] (5.6, -2.5) -- (6.8, -2.5);
  
  \node at (5.2, -2.5) {$M_{\phi}$};

   \end{tikzpicture}
\caption{Dehn filling complex for proper actions associated to a Cohen--Lyndon triple $(G,\h,\{N_{\lambda}\}_{\lambda\in\Lambda})$. The left and middle circles represent the $\bg$-translates of spaces corresponding to a single subgroup $H_{\lambda} \in \h$, while the dots indicate the $\bg$-translates of spaces corresponding to all other subgroups in $\h$.}
    \label{fig:proper_Dehn_space}
\end{figure}
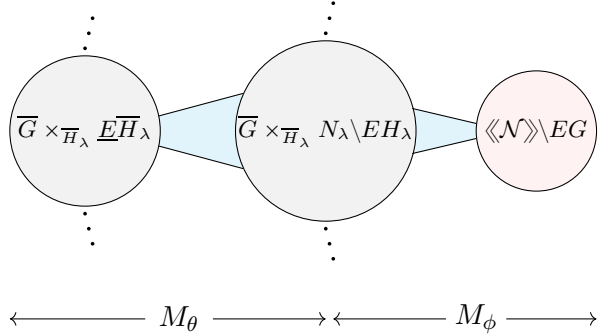

\begin{remark} Note that when $\bg$ is torsion-free,  $\E \bg$ is just $E\bg$ and the property \hyperlink{cond:fin}{(FIN)} holds vacuously. In this case, the complex $Y$ coincides with the universal cover of the Dehn filling space of \Cref{Dehn_space}.
\end{remark}

\begin{theorem}\label{thm: proper Dehn filling space}
Let $(G,\h,\{N_{\lambda}\}_{\lambda\in\Lambda})$ be  a Cohen--Lyndon triple that satisfies property \hyperlink{cond:fin}{(FIN)}. Then, every Dehn filling complex for proper actions is a CW-model for $\E \bg$.
\end{theorem}
\begin{proof} Suppose $Y$ is the Dehn filling complex for proper actions associated to the Cohen--Lyndon triple  $(G,\h,\{N_{\lambda}\}_{\lambda\in\Lambda})$. Let $F$ be a subgroup of $\bg$. If $Y^F$ is nonempty, then either there is a fixed point in $\ll \N \rr\backslash EG$, in which case $F$ is trivial or there is a fixed point in a $\bg$-translate of  $\E\bhl$ in $\bg\times_{\bhl}\E\bhl$ for some $\lambda \in \Lambda$, in which case $F$ is finite and is contained in a conjugate of $\bhl$.  It remains to show that for every finite $F< \bg$, the fixed-point set $Y^F$ is contractible.

Let $F$ be a finite subgroup of $\bg$. If $F=\langle 1\rangle$, then $Y^F=Y$ and the proof of the contractibility of $Y$ is completely analogous to the proof of the contractibility of the universal cover of the Dehn filling space of the triple  $(G,\h,\{N_{\lambda}\}_{\lambda\in\Lambda})$ in \cite[Theorem 4.2]{PS_L2}. If $F\ne \langle 1\rangle$, then by property {\rm \hyperlink{cond:fin}{(FIN)}},  $F$ is contained in a unique conjugate of $\bhl$ for some $\lambda \in \Lambda$. Without loss of generality, we can assume $F<\bhl$ and $F\cap h\bhl h^{-1}=\langle 1\rangle$ if $h\notin \bhl$. In this case, $Y^F=\E\bhl^F\simeq \ast$. 
\end{proof}

\section{Equivariant homology}\label{sec_homology}

\subsection{Excision diagram for equivariant homology}

We briefly recall from \cite{BEL08} (see also  \cite[Chapter 12]{Lueck25}) the notions of $G$-homology theories and equivariant homology theories. Fix a discrete group $G$ and a ring $S$. A $G$-homology theory $\h_\ast^G$ with values in $S$-modules is a collection of covariant functors $\h_n^G$ from the category of $G$-CW pairs to the category of $S$-modules indexed by $n\in \bZ$, together with the natural transformation 
$$\partial_n^G(X, X_0)\colon \h_n^G(X, X_0) \to \h_{n-1}^G(X_0)\coloneq \h^G_{n-1}(X_0, \emptyset)$$
satisfying the following axioms (see \cite[Definition 1.1]{BEL08}): $G$-homotopy invariance, long exact sequence of a pair, excision, and disjoint union axiom.

Let $(G,\{H_{\lambda}\}_{\lambda\in\Lambda},\{N_{\lambda}\}_{\lambda\in\Lambda})$ be a Cohen--Lyndon triple with $\bg=G/\ll \mathcal{N} \rr$ having property {\rm \hyperlink{cond:fin}{(FIN)}} and denote by $M_\phi, M_\theta$ the mapping cylinders of the $\bg$-equivariant maps $\phi, \theta$, respectively, as in Definition \ref{proper_Dehn_space}. Thus, we have canonical $\bg$-equivariant homotopy equivalences
\[
M_\phi \simeq \ll \mathcal{N}\rr \backslash EG \,\; \mbox{ and } \,\; M_\theta \simeq  \bigsqcup_{\lambda \in \Lambda} \bg \times_{\bhl} \E\bhl.
\]
Recall from \Cref{proper_Dehn_space} and Theorem \ref{thm: proper Dehn filling space} that the Dehn  filling complex $Y$ is a $\bg$-CW-model of $\E\bg$ where
\[
Y\coloneq M_\phi \sqcup_Z M_\theta  \,\; \mbox{ and } \,\;  Z\coloneq \bigsqcup_{\lambda \in \Lambda} \bg \times_{\bhl} (N_\lambda \backslash EH_\lambda).
\]

The following is immediate from the four axioms of a $\bg$-homology theory:

\begin{proposition}\label{prop_CL_diagram} Let $(G,\{H_{\lambda}\}_{\lambda\in\Lambda},\{N_{\lambda}\}_{\lambda\in\Lambda})$ be  a Cohen--Lyndon triple  satisfying property \hyperlink{cond:fin}{(FIN)}. Let $\h^{\bg}_{\ast}$ be any $\bg$-homology theory. Then, the map $(M_\phi, Z) \to (Y, M_\theta)$ of $G$-CW pairs  induces the following commutative diagram with exact rows
\begin{equation*}
\xymatrix{
\cdots  \ar[r] & \h_\ast^{\bg}(Z) \ar[r]^-{} \ar[d] &  \h_\ast^{\bg}(M_\phi) \ar[r] \ar[d]
& \h_\ast^{\bg}(M_\phi, Z) \ar[d]^{\cong}   \ar[r] & \cdots  \\
\cdots  \ar[r] & \h_\ast^{\bg}(M_\theta) \ar[r]^-{} & \h_\ast^{\bg}(Y)  \ar[r] 
& \h_\ast^{\bg}(Y, M_\theta) \ar[r] & \cdots,
}
\end{equation*}
which is isomorphic to
\begin{equation}\label{eq_diagram_Ghomlogy}
\xymatrix{
\cdots  \ar[r] & \bigoplus_{\lambda \in \Lambda}\h_\ast^{\bg}(\bg \times_{\bhl} N_\lambda \backslash EH_\lambda) \ar[r]^-{} \ar[d] &  \h_\ast^{\bg}(\m\backslash EG) \ar[r] \ar[d] 
& \h_\ast^{\bg}(M_\phi, Z) \ar[d]^{\cong}  \ar[r] & \cdots  \\
\cdots  \ar[r] & \bigoplus_{\lambda \in \Lambda}\h_\ast^{\bg}(\bg \times_{\bhl} \E \bhl)  \ar[r]^-{} & \h_\ast^{\bg}(\E \bg)  \ar[r] 
& \h_\ast^{\bg}(Y, M_\theta) \ar[r] & \cdots
}
\end{equation}
\end{proposition}

Now, recall from \cite[Definition 1.3]{BEL08} (see also \cite[Chapter 12]{Lueck25}) that an equivariant homology theory $\h_\ast^{?}$ over a group $\Gamma$  assigns to every group $(G, \xi\colon G\to \Gamma)$ over $\Gamma$, a $G$-homology theory $\h_\ast^{G}$ (more precisely,  $\h_\ast^{(G, \xi)}$), and comes with the induction structure: for any homomorphism $\alpha\colon (G_1, \xi_1) \to (G_2, \xi_2)$ of groups over $\Gamma$ and a $G_1$-CW pair $(X, X_0)$, there are natural homomorphisms
\begin{equation}\label{eq_induction}
\mathrm{ind}_\alpha\colon \h^{G_1}_n(X, X_0) \to \h^{G_2}_n(\alpha_\ast(X, X_0))\quad \alpha_\ast(X,X_0)\coloneq (G_2\times_{G_1}X, G_2\times_{G_1}X_0)
\end{equation}
satisfying (see \cite[Definition 1.3]{BEL08}) compatibility with the boundary homomorphisms, functoriality, compatibility with conjugation, and bijectivity axiom: 
$$\mathrm{ind}_\alpha\colon \h_n^{G_1}(\{\cdot\})\to \h_n^{G_2}(G_2/G_1)$$
is bijective for any homomorphism $\alpha\colon (G_1, \xi_1)\to (G_2, \xi_2)$ over $\Gamma$ such that the underlying map  $\alpha\colon G_1\to G_2$ is an inclusion of groups.

\begin{theorem}\label{thm_CLP_diagram} Let $(G,\{H_{\lambda}\}_{\lambda\in\Lambda},\{N_{\lambda}\}_{\lambda\in\Lambda})$ be a Cohen--Lyndon triple  satisfying property \hyperlink{cond:fin}{(FIN)}. Let $\h^{?}_{\ast}$ be any equivariant homology theory over $\bg$. We have the following commutative diagram with exact rows
\begin{equation}\label{eq_diag_equiv}
\xymatrix{
\cdots  \ar[r] & \bigoplus_{\lambda \in \Lambda}\h_\ast^{H_\lambda}( EH_\lambda) \ar[r]^-{} \ar[d] &  \h_\ast^{G}(EG) \ar[r] \ar[d] 
& \h_\ast^{\bg}(M_\phi, Z) \ar[d]^{\cong}  \ar[r] & \cdots  \\
\cdots  \ar[r] & \bigoplus_{\lambda \in \Lambda}\h_\ast^{\bhl}(\E \bhl)  \ar[r]^-{} & \h_\ast^{\bg}(\E \bg)  \ar[r] 
& \h_\ast^{\bg}(Y, M_\theta) \ar[r] & \cdots
}
\end{equation}
which is canonically isomorphic to the commutative diagram \eqref{eq_diagram_Ghomlogy}. Here, all the four maps in the first square are the obvious maps obtained by the induction homomorphisms followed by the classifying maps.
\end{theorem}
\begin{proof} By \cite[Lemma 1.5]{BEL08}, the homomorphisms $H_{\lambda}\twoheadrightarrow \bhl$, $\bhl \hookrightarrow \bg$, and $G\twoheadrightarrow \bg$ induce the induction isomorphisms
\[
\h_\ast^{H_\lambda}( EH_\lambda) \cong \h_\ast^{\bhl}(N_\lambda \backslash EH_\lambda) \cong  \h_\ast^{\bg}(\bg \times _{\bhl} N_\lambda \backslash EH_\lambda),
\]
\[
\h_\ast^{\bhl }( \E\bhl) \cong \h_\ast^{\bg}(\bg \times _{\bhl} \E \bhl),
\]
\[
\h_\ast^{G}(EG) \cong \h_\ast^{\bg}(\m \backslash EG).
\]
After replacing the corresponding terms in the diagram \eqref{eq_diagram_Ghomlogy} via these isomorphisms, the assertion on the corresponding maps in \eqref{eq_diag_equiv} in the first square follows by the naturality and the functoriality of the induction homomorphisms.
\end{proof}

\begin{remark}\label{rem_CLP_proper} Theorem \ref{thm_CLP_diagram} holds more generally for any equivariant homology over $\bg$ that is defined on the category of proper $\bg$-CW pairs. The proof, including \cite[Lemma 1.5]{BEL08}, works verbatim.
\end{remark}

Given a group $\Gamma$ and a ring $R$ (with involution) equipped with a $\Gamma$-action, by \cite[Theorem 5.1]{BEL08} we have the following equivariant homology theories 
\[
H_\ast^{?}(-;\mathrm{K_R}),\; H_\ast^{?}(-;\mathrm{KH_R}),\; H_\ast^{?}(-;\mathrm{L_R^{\langle -\infty \rangle}}), 
\]
with values in $\bZ$-modules over the group $\Gamma$, satisfying the properties listed in \cite[Theorem 5.1]{BEL08}. Theorem \ref{thm_CLP_diagram} is applicable to any of these equivariant homology theories. 

For any equivariant homology theory $\h^{?}_\ast$ (over a group $\Gamma$), by the naturality of the induction homomorphisms, we have the following commutative diagram:
\begin{equation}\label{eq_asembly_0}
\xymatrix{
\h_\ast^{G_1}(X) \ar[r]^{f_1\circ \mathrm{ind}_\alpha} \ar[d] &  \h^{G_2}_\ast(Y) \ar[d] \\
\h_\ast^{G_1}(\{\cdot\}) \ar[r]^{f_2 \circ \mathrm{ind}_\alpha} &  \h^{G_2}_\ast(\{\cdot\}), 
}
\end{equation}
where the vertical maps are induced by the projection to the point $\{\cdot\}$, $\alpha\colon G_1\to G_2$ is a homomorphism over $\Gamma$, $f_1\colon \alpha_*(X)\to Y$ is any $G_2$-equivariant map, $f_2$ is the projection $G_2/G_1 \to \{\cdot\}$. 

\begin{corollary}\label{cor_CLP_diagram} In the setting of Theorem \ref{thm_CLP_diagram}, suppose that the projection to $\{\cdot\}$ induces isomorphisms
\begin{align*}
\begin{array}{rcl@{\hspace{1.5em}}rcl}
    \h_\ast^{H_\lambda}( EH_\lambda) &\cong& \h_\ast^{H_\lambda}(\{\cdot\}), &
    \h_\ast^{\bhl}( \E \bhl) &\cong& \h_\ast^{\bhl}(\{\cdot\}), \\
    \h_\ast^{G}( EG) &\cong& \h_\ast^{G}(\{\cdot\}), &
    \h_\ast^{\bg}( \E \bg) &\cong& \h_\ast^{\bg}(\{\cdot\}).
\end{array}
\end{align*}
Then, there is a commutative diagram with exact rows:
\begin{equation}\label{eq_diagram_equiv}
\xymatrix{
\cdots  \ar[r] & \bigoplus_{\lambda \in \Lambda}\h_\ast^{H_\lambda}(\{\cdot\}) \ar[r]^-{} \ar[d] &  \h_\ast^{G}(\{\cdot\}) \ar[r] \ar[d] 
& \h_\ast^{\bg}(M_\phi, Z) \ar[d]^{\cong}  \ar[r] & \cdots  \\
\cdots  \ar[r] & \bigoplus_{\lambda \in \Lambda}\h_\ast^{\bhl}(\{\cdot\})  \ar[r]^-{} & \h_\ast^{\bg}(\{\cdot\})  \ar[r] 
& \h_\ast^{\bg}(Y, M_\theta) \ar[r] & \cdots.
}
\end{equation}
Here, all the four maps in the first square are the obvious maps obtained by the induction homomorphisms followed by the projection to the point $\{\cdot\}$.
\end{corollary}

\subsection{Excision diagram for the analytic $K$-homology}\label{sec:ex_Khom} For any locally compact group $G$, a $G$-$C^*$-algebra is a $C^*$-algebra equipped with a point-norm continuous action of $G$ by $\ast$-automorphisms.

In \cite[Theorem 5.1]{BEL08}, it is stated that for any discrete group $\Gamma$ and a $\Gamma$-$C^*$-algebra $A$, we have the following equivariant homology theories over $\Gamma$,
\[
H_\ast^{?}(-;\mathrm{K^{\mathrm{top}}_{A,r}}),  \; H_\ast^{?}(-;\mathrm{K^{\mathrm{top}}_{A,m}}),
\]
where $H_\ast^{?}(-;\mathrm{K^{\mathrm{top}}_{A,r}})$ has a slightly limited induction structure (the induction \eqref{eq_induction} exists iff the kernel of $\alpha$ acts on $X\smallsetminus X_0$ with amenable stabilizers). However, to the best of the authors' knowledge, with the currently known construction of these homology theories which uses $C^*$-categories (\cite{M04}, \cite{Kranz21}, \cite{BEL21}), the induction homomorphisms \eqref{eq_induction} are defined only when $\alpha$ is an injective group homomorphism. In particular, the proof of Theorem \ref{thm_CLP_diagram} does not work in this setting. For this reason, in this section, we prove the analogue of Theorem \ref{thm_CLP_diagram} for Kasparov's analytic $K$-homology separately.

Let us first recall the Baum--Connes conjecture with coefficients. For any locally compact group $G$ and for any $G$-$C^*$-algebra $A$, $A\rtimes_rG$ denotes the reduced crossed product. For any $C^*$-algebra $B$, $K_\ast(B)$ ($\ast=0,1$) denotes the topological $K$-theory of $B$.

\begin{definition}[The Baum--Connes conjecture with coefficients \cite{BCH}] We say that a second countable, locally compact group $G$ satisfies the Baum--Connes conjecture with coefficients (BCC) if for any separable $G$-$C^*$-algebra $A$, the Baum--Connes assembly map
\[
\mu^{G}_{A, r}\colon K^G_\ast(\E G; A)\to K_\ast(A\rtimes_rG)
\]
is an isomorphism of abelian groups for $\ast=0, 1$. 
\end{definition}

Here, for any separable $G$-$C^*$-algebra $A$, the analytic $K$-homology $K^G_\ast(-; A)$ for proper $G$-spaces is defined, using Kasparov's equivariant $KK$-theory $KK^G$ \cite{Kas88} as follows: for any pair $(X, X_0)$ of proper $G$-spaces with $X_0\subset X$ a closed $G$-subspace, we define
\[
K^G_\ast(X, X_0; A)\coloneq \colim_{(Y, Y_0)\subset (X, X_0)}KK^G_\ast(C_0(Y\smallsetminus Y_0), A) 
\]
where the colimit is taken over all $G$-compact $G$-sub-pairs $(Y, Y_0)\subset (X, X_0)$. Here, a proper $G$-space means it is either a metrizable proper $G$-space (as in \cite{BCH}) or a proper $G$-CW complex. For any such a proper $G$-space $X$, any $G$-compact $G$-subspace $Y$ is a second countable, locally compact proper $G$-space.

This is a $G$-homology theory defined for proper $G$-spaces; namely, it satisfies the four axioms: $G$-homotopy invariance, long exact sequence of a pair, excision and disjoint union. The long exact sequence of a pair exists since $C_0(Y_0), C_0(Y), C_0(Y\smallsetminus Y_0)$ are proper, nuclear separable $G$-$C^*$-algebras. Moreover, by \cite[Appendix]{KS03}, we have a canonical isomorphism $KK^G(B_1, B_2)\cong E^G(B_1, B_2)$ for any proper, nuclear separable $G$-$C^*$-algebra $B_1$ and for any separable $G$-$C^*$-algebra $B_2$, and all the arguments below in this section can be carried out using the equivariant $E$-theory \cite{GHT} instead.

For an introduction and comprehensive survey of the Baum--Connes conjecture, see \cite{BC2020}, \cite{BC02} and \cite{BCH}. For any countable discrete group $G$, it has been shown (see \cite{Kranz21}, \cite{BEL21}, \cite{HP04}) that the Baum--Connes assembly map $\mu^{G}_{A, r}$ is equivalent to the topological assembly map for the ``equivariant" (topological) $K$-homology theory:
\[
H_\ast^{G}(\E G; \mathrm{K^{\mathrm{top}}_{A,r}}) \to H_\ast^{G}(\{\cdot\}; \mathrm{K^{\mathrm{top}}_{A,r}}),
\]
induced by the projection $\E G\to \{\cdot\}$, but as we remarked above, the ``equivariant" homology theory $H_\ast^{?}(-; \mathrm{K^{\mathrm{top}}_{A,r}})$ lacks the induction structure with respect to  non-injective group homomorphisms. 

Below, throughout this section, we fix a separable $\bg$-$C^*$-algebra $A$, all groups and group homomorphisms are implicitly assumed to be over $\bg$, and for each such group $G_1$ over $\bg$, we consider $A$ as a $G_1$-$C^*$-algebra without writing $\mathrm{Res}_{\bg}^{G_1}(A)$ each time. Hence, $K_\ast^{G_1}(X; A)$ makes sense for any proper $G_1$-space $X$.

For any injective group homomorphism $\alpha\colon G_1\to G_2$, there are natural induction isomorphisms: 
\[
\mathrm{ind}_\alpha\colon K_\ast^{G_1}(X, X_0; A) \xrightarrow{\cong} K_\ast^{G_2}(\alpha_\ast(X, X_0); A)\quad \alpha_\ast(X,X_0)\coloneq (G_2\times_{G_1}X, G_2\times_{G_1}X_0).
\]
This is defined by the induction-restriction adjunction (the Frobenius reciprocity): a natural isomorphism (see \cite[Section 2]{CE} and \cite[Eq.(20)]{MN06}):
\[
KK^{G_2}(\mathrm{ind}_\alpha(B_1), B_2) \cong KK^{G_1}(B_1, B_2),
\]
for any separable $G_2$-$C^*$-algebra $B_2$ and for any separable $G_1$-$C^*$-algebra $B_1$. For any locally compact $G_1$-space $X$, we have $\mathrm{ind}_\alpha(C_0(X))\cong C_0(G_2\times_{G_1}X)$ canonically. 

\begin{lemma} Let $G$ be a countable discrete group and $(G,\{H_{\lambda}\}_{\lambda\in\Lambda},\{N_{\lambda}\}_{\lambda\in\Lambda})$ be  a Cohen--Lyndon triple satisfying property \hyperlink{cond:fin}{(FIN)}. Let $A$ be a separable $\bg$-$C^*$-algebra. Then, the map $(M_\phi, Z) \to (Y, M_\theta)$ of $G$-CW pairs induces the following commutative diagram with exact rows
\begin{equation*}
\xymatrix{
\cdots  \ar[r] & K_\ast^{\bg}(Z; A) \ar[r]^-{} \ar[d] &  K_\ast^{\bg}(M_\phi; A) \ar[r] \ar[d]
& K_\ast^{\bg}(M_\phi, Z; A) \ar[d]^{\cong}   \ar[r] & \cdots  \\
\cdots  \ar[r] & K_\ast^{\bg}(M_\theta; A) \ar[r]^-{} & K_\ast^{\bg}(Y; A)  \ar[r] 
& K_\ast^{\bg}(Y, M_\theta; A) \ar[r] & \cdots,
}
\end{equation*}
which is isomorphic to
\begin{equation}\label{eq_lem_diagram}
\xymatrix{
\cdots  \ar[r] & \bigoplus_{\lambda \in \Lambda}K_\ast^{\bhl}(N_\lambda \backslash EH_\lambda; A) \ar[r]^-{} \ar[d] &  K_\ast^{\bg}(\m\backslash EG; A) \ar[r] \ar[d] 
& K_\ast^{\bg}(M_\phi, Z; A) \ar[d]^{\cong}  \ar[r] & \cdots  \\
\cdots  \ar[r] & \bigoplus_{\lambda \in \Lambda}K_\ast^{\bhl}(\E \bhl; A)  \ar[r]^-{} & K_\ast^{\bg}(\E \bg; A)  \ar[r] 
& K_\ast^{\bg}(Y, M_\theta; A) \ar[r] & \cdots
}
\end{equation}
where, in the first square, the two vertical maps are induced by the classifying maps and  the two horizontal maps are induced by the induction maps followed by the maps induced by the classifying maps $EH_\lambda\to EG$  and $\E \bhl \to \E\bg$, respectively.
\end{lemma}
\begin{proof} We apply (the proof of) Proposition \ref{prop_CL_diagram} for the $\bg$-homology theory $\h^{\bg}_\ast(-)=K_\ast^{\bg}(-;A)$ for proper $\bg$-spaces. The diagram \eqref{eq_diagram_Ghomlogy} is isomorphic to \eqref{eq_lem_diagram} via the induction isomorphisms for the inclusions $\bhl\to \bg$. The assertion on the maps in the first square follows from the naturality of the induction isomorphisms.
\end{proof}

For any normal subgroup $N_1$ of a countable discrete group $G_1$, let $\bg_1\coloneq G_1/N_1$. We continue to suppose that all groups and homomorphisms are over $\bg$, and $A$ is a $\bg$-$C^*$-algebra. Hence, $A$ is canonically a $G_1$-$C^*$-algebra with the trivial $N_1$-action. 

Let $\alpha_1 \colon G_1\to \bg_1$ be the quotient map. For any proper $G_1$-space $X$ on which $N_1$ acts freely, let $\bX$ be the proper $\bg_1$-space $N_1\backslash X$. We define an induction isomorphism 
\begin{equation}\label{eq_induction_quotient}
\mathrm{ind}_{\alpha_1}\colon K_\ast^{G_1}(X; A) \xrightarrow{\cong} K_\ast^{\bg_1}(\bX; A),
\end{equation}
by the composition
\begin{equation}\label{eq_induction_quotient_def}
K_\ast^{G_1}(X; A) \xrightarrow{\cong} K_\ast^{G_1\times \bg_1}(\Delta_\ast{X}; A) \xrightarrow{\cong}  K_\ast^{\bg_1}(\bX; A),
\end{equation}
where the first map is the induction isomorphism $\mathrm{ind}_{\Delta}$ for the diagonal inclusion $\Delta\colon G_1\to G_1\times \bg_1$, regarding $A$ as a $ G_1\times \bg_1$-$C^*$-algebra with the trivial $G_1$-action, and the second map is defined as follows: for a $G_1$-compact, locally compact second countable, proper $G_1$-space $X$ on which $N_1$ acts freely,
\begin{enumerate}
\item[(i)] we have $KK^{G_1\times \bg_1}(C_0(\Delta_\ast X), A) \cong KK^{\bg_1}(C_0(\Delta_\ast X)\rtimes G_1, A)$ canonically, since $G_1$-acts on $A$ trivially (see \cite[Eq (9)]{Meyer08}); 
\item[(ii)] the $\bg_1$-equivariant $C_0(\bX)$-$C_0(\Delta_\ast X)\rtimes G_1$-imprimitivity bimodule $\Lambda_{\Delta_\ast X,G_1}^{G_1\times \bg_1, G_1}$ (see \cite[Section 1]{CE} or \cite[Definition 5.2]{CE01}) induces an isomorphism $KK^{\bg_1}(C_0(\Delta_\ast X)\rtimes G_1, A)\cong KK^{\bg_1}(C_0(\bX), A)$. Here, we use the canonical $\bg_1$-equivariant isomorphism $G_1\backslash (\Delta_\ast X) = G_1\backslash((G_1\times \bg_1)\times_{\Delta G_1}  X) \cong \bX $, which sends $[(g_1,g_2,x)]\to [g_2\cdot x]$, to identify the two spaces.
\end{enumerate}

Regarding the imprimitivity bimodule $\Lambda_{\Delta_\ast X,G_1}^{G_1\times \bg_1, G_1}$ that we used, it is easier to describe its dual (the inverse), which is $C_0(\Delta_\ast X)\rtimes G_1$-$C_0(\bX)$-imprimitivity bimodule $E_{\Delta_\ast X,G_1}^{G_1\times \bg_1, G_1}$. This is defined by completing $C_c(\Delta_\ast X)$ with respect to the following $C_0(\bX)$-valued pre-inner product: for any $\xi_1,\xi_2\in C_c(\Delta_\ast X)$
\[
\langle \xi_1,\xi_2\rangle_{C_0(\bX)}(\overline{x})\coloneq \sum_{p(x)=\overline{x}} \overline{\xi_1(x)}\xi_2(x)
\]
where $p\colon \Delta_\ast X\to G_1\backslash\Delta_\ast X =\bX$ is the quotient map. With the canonical left-action of $C_0(\Delta_\ast X)\rtimes G_1$, this is a $\bg_1$-equivariant  $C_0(\Delta_\ast X)\rtimes G_1$-$C_0(\bX)$-imprimitivity bimodule $E_{\Delta_\ast X,G_1}^{G_1\times \bg_1, G_1}$. See \cite[Proposition 1.1]{KP22} or \cite[Section 5]{KS03} (for the dual $\Lambda_{\Delta_\ast X,G_1}^{G_1\times \bg_1, G_1}$) for example. 

For a general proper $G_1$-space, the induction \eqref{eq_induction_quotient} is defined by taking the colimit over $G_1$-compact proper $G_1$-subspaces, and it is well-defined by the naturality of the induction:

\begin{lemma}[{c.f.\cite[Lemma 1.5]{CE}}]\label{lem_quotient_ind_natural} For any $G_1$-equivariant continuous map $f\colon X\to Y$ between $G_1$-compact, second countable locally compact, proper $G_1$-spaces on which $N_1$ act freely, let $\overline{f}\colon \bX\to \bY$ be the induced map between the quotients $\bX=N_1\backslash X$ and $\bY=N_1\backslash Y$, which is a $\bg_1$-equivariant continuous map between proper $\bg_1$-spaces. The following square commutes
\[
\xymatrix{
\mathrm{ind}_{\alpha_1}\colon   &  K_\ast^{G_1}(X; A) \ar[d]^-{f_\ast} \ar[r]^-{\cong} &  K_\ast^{\bg_1}(\bX; A) \ar[d]^-{\overline{f}_\ast} \\
\mathrm{ind}_{\alpha_1}\colon   &  K_\ast^{G_1}(Y; A)  \ar[r]^-{\cong} &  K_\ast^{\bg_1}(\bY; A).
}
\]
It follows that the induction \eqref{eq_induction_quotient} is defined canonically for any proper $G_1$-space on which $N_1$-acts freely, and the above square commutes for any such proper $G_1$-spaces $X, Y$ and $f\colon X\to Y$.
\end{lemma}
\begin{proof} Note that the first map of the induction \eqref{eq_induction_quotient} is the induction isomorphism $\mathrm{ind}_\Delta$ for the injective homomorphism $\Delta\colon G_1\to G_1\times \bg_1$, which is natural with respect to maps between spaces. Thus, it suffices to show that the following square commutes:
\[
\xymatrix{
 K_\ast^{G_1\times \bg_1}(\Delta_\ast{X}; A) \ar[d]^-{\Delta_\ast(f)_\ast} \ar[r]^-{\cong} &  K_\ast^{\bg_1}(\bX; A) \ar[d]^-{\overline{f}_\ast} \\
 K_\ast^{G_1\times \bg_1}(\Delta_\ast{Y}; A)  \ar[r]^-{\cong} &  K_\ast^{\bg_1}(\bY; A).
}
\]
Here, $\Delta_\ast(f)\colon (G_1\times \bg_1)\times_{\Delta G_1}X \to  (G_1\times \bg_1)\times_{\Delta G_1}Y$ is the $G_1\times \bg_1$-equivariant map that $f$ induces.

Let $(E, T)$ be a cycle for $KK_\ast^{G_1\times \bg_1}(C_0(\Delta_\ast{X}); A) \cong KK_\ast^{\bg_1}(C_0(\Delta_\ast{X})\rtimes G_1, A)$ where $E$ is a graded $\bg_1$-equivariant Hilbert $C_0(\Delta_\ast{X})\rtimes G_1$-$A$-module and $T\in \Linears(E)$ is an odd bounded adjointable operator on $E$ satisfying the usual axioms (which do not matter in the argument below).

On one hand, the composition of the upper and right sides sends the element $[E, T]$ to 
\[
[\overline{f}^\ast]\otimes_{C_0(\bX)} [\Lambda_{\Delta_\ast X,G_1}^{G_1\times \bg_1, G_1} ]\otimes_{C_0(\Delta_\ast X)\rtimes G_1}[E,T]
\]
where $[\Lambda_{\Delta_\ast X,G_1}^{G_1\times \bg_1, G_1}] \in KK^{\bg_1}(C_0(\bX),C_0(\Delta_\ast X)\rtimes G_1)$ is the element defined by the imprimitivity bimodule $\Lambda_{\Delta_\ast X,G_1}^{G_1\times \bg_1, G_1}$ and $[\overline{f}^\ast]$ is the class of the $\ast$-homomorphism $\overline{f}^*\colon C_0(\bY)\to C_0(\bX)$.

On the other hand, the composition of left and lower sides sends the element $[E, T]$ to 
\[
[\Lambda_{\Delta_\ast Y,G_1}^{G_1\times \bg_1, G_1} ]\otimes_{C_0(\Delta_\ast Y)\rtimes G_1} [\Delta_\ast (f)^\ast\rtimes G_1]\otimes_{C_0(\Delta_\ast X)\rtimes G_1}[E,T]
\]
where $[\Delta_\ast (f)^\ast\rtimes G_1]$ is the class of the $\ast$-homomorphism $C_0(\Delta_\ast Y)\rtimes G_1\to C_0(\Delta_\ast X)\rtimes G_1$ induced from $f$.

By the proof of \cite[Lemma 1.5]{CE}, we have
\[
[\Lambda_{\Delta_\ast Y,G_1}^{G_1\times \bg_1, G_1} ]\otimes_{C_0(\Delta_\ast Y)\rtimes G_1} [\Delta_\ast (f)^\ast\rtimes G_1] = [\overline{f}^\ast]\otimes_{C_0(\bX)} [\Lambda_{\Delta_\ast X,G_1}^{G_1\times \bg_1, G_1} ] 
\]
in $KK^{\bg_1}(C_0(\bY), C_0(\Delta_\ast X) \rtimes G_1)$ on the level of bimodules. It follows that the above square commutes, and we are done.
\end{proof}

\begin{lemma}[c.f. {\cite[Lemma 5.13]{CE01}}]\label{lem_natural_ind_ind} Let $A$ be a separable $\bg$-$C^*$-algebra. For $i=1,2$, let $G_i$ be a countable discrete group and $N_i$ be a normal subgroup of $G_i$, and let $\bg_i=G_i/N_i$. Suppose these are all defined over the group $\bg$, and so, in particular $N_i$ acts trivially on $A$. Suppose $\beta\colon G_1\to G_2$ is an inclusion and $N_2\cap G_1=N_1$. Let $\alpha_i\colon G_i\to \bg_i$ be the quotient maps. Let $\overline{\beta}\colon \bg_1\to \bg_2$ be the inclusion induced by $\beta$.

For any proper $G_1$-space $X$ on which $N_1$-acts freely, the following square commutes
\[
\xymatrix{
\mathrm{ind}_{\alpha_1}\colon   &  K_\ast^{G_1}(X; A) \ar[d]_-{\mathrm{ind}_\beta}^-{\cong} \ar[r]^-{\cong} &  K_\ast^{\bg_1}(\bX; A) \ar[d]_-{\mathrm{ind}_{\overline{\beta}}}^-{\cong} \\
I_\ast \circ \mathrm{ind}_{\alpha_2}\colon   &  K_\ast^{G_2}(G_2\times_{G_1} X; A)  \ar[r]^-{\cong} &  K_\ast^{\bg_2}(\bg_2\times_{\bg_1} \bX; A),
}
\]
where we use the canonical isomorphism
\[
I\colon N_2\backslash(G_2\times_{G_1} X )\cong  \bg_2\times_{\bg_1}  (N_1\backslash X), \quad [g,x]\mapsto [g, x].
\]
\end{lemma}
\begin{proof} 
Note that the first map of the induction \eqref{eq_induction_quotient} is the induction isomorphism $\mathrm{ind}_\Delta$ for the injective homomorphism $\Delta\colon G_1\to G_1\times \bg_1$, which is functorial with respect to the composition of injective group homomorphisms. Thus, it suffices to show that the following square commutes:
\[
\xymatrix{
 K_\ast^{G_1\times \bg_1}(\Delta_\ast{X}; A) \ar[d]^-{\cong}_-{\mathrm{ind}_{\beta\times \overline{\beta}}} \ar[r]^-{\cong} &  K_\ast^{\bg_1}(\bX; A) \ar[d]^-{\cong}_-{\mathrm{ind}_{ \overline{\beta}}} \\
 K_\ast^{G_2\times \bg_2}((G_2\times \bg_2)\times_{G_1\times \bg_1} \Delta_\ast{X}; A)  \ar[r]^-{\cong} &  K_\ast^{\bg_2}(\bg_2\times _{\bg_1}\bX; A).
}
\]
where we use the canonical isomorphism of $\bg_2$-spaces
\[
G_2\backslash \left((G_2\times \bg_2) \times_{G_1\times \bg_1} \Delta_\ast{X}\right) \cong \bg_2\times_{\bg_1}  (G_1\backslash \Delta_\ast X) \cong  \bg_2\times_{\bg_1}  (N_1\backslash X),
\]
sending
\[
[(g_1,g_2),(g_3, g_4, x)] \mapsto [g_2,[g_3,g_4,x]] \mapsto [g_2, g_4\cdot x], \quad g_1,g_2\in G_2, g_3,g_4 \in G_1, x\in X.
\]

It is enough to show this for $G_1$-compact $X$. We show the inverses of the two compositions are equal. We use the compression isomorphism (see \cite[Proposition 5.14]{CE}) which is the inverse of the induction isomorphism for an (open) inclusion of discrete groups.

Let $[E, T]$ be an element in $KK^{\bg_2}(C_0(\bg_2\times_{\bg_1}\bX), A)$ where $E$ has a representation of $C_0(\bg_2\times_{\bg_1}\bX)$. The restriction via the $\bg_1$-equivariant inclusion $C_0(\bX)= C_0(\bg_1\times_{\bg_1}\bX)\to C_0(\bg_2\times_{\bg_1}\bX)$ defines the so-called the compression/projection $(E^0, T^0)$ of $(E, T)$, which is a cycle for $KK^{\bg_1}(C_0(\bX), A)$. This is what the compression isomorphism $\mathrm{ind}_{\overline{\beta}}^{-1}$ maps $[E, T]$ to (see \cite[Proposition 5.14]{CE}).

On one hand, the inverse of the composition of the upper and right sides sends the element $[E, T]$ in $KK^{\bg_2}(C_0(\bg_2\times_{\bg_1}\bX), A)$ to 
\[
 [E_{\Delta_\ast X,G_1}^{G_1\times \bg_1, G_1} ]\otimes_{C_0(\bX)}[E^0,T^0]
\in KK^{\bg_1}(C_0(\Delta_\ast X)\rtimes G_1, A) \cong KK^{G_1\times \bg_1}(C_0(\Delta_\ast X), A),
\]
where $[E_{\Delta_\ast X,G_1}^{G_1\times \bg_1, G_1}] \in KK^{\bg_1}(C_0(\Delta_\ast X)\rtimes G_1, C_0(\bX))$ is the dual (the inverse) of $[\Lambda_{\Delta_\ast X,G_1}^{G_1\times \bg_1, G_1}]$ (see the paragraph before Lemma \ref{lem_quotient_ind_natural}). 

On the other hand, the inverse of the composition of left and lower sides sends the element $[E, T]$ to the compression of 
\begin{equation}\label{eq_compression_ET}
[E_{(G_2\times \bg_2) \times_{G_1\times \bg_1} \Delta_\ast X,G_2}^{G_2\times \bg_2, G_2} ]\otimes_{C_0(\bg_2\times_{\bg_1} \bX)}[E,T],
\end{equation}
and the compression is defined as its $G_1\times \bg_1$-equivariant restriction via the inclusion $C_0(\Delta_\ast\bX)\to C_0((G_2\times\bg_2)\times_{G_1\times\bg_1}\Delta_\ast\bX)$.

Analogously to the proof of \cite[Lemma 5.13]{CE01}, the $G_1\times \bg_1$-equivariant compression/restriction of $E_{(G_2\times \bg_2) \times_{G_1\times \bg_1} \Delta_\ast X,G_2}^{G_2\times \bg_2, G_2}$ via the inclusion $C_0(\Delta_\ast\bX)\to C_0((G_2\times\bg_2)\times_{G_1\times\bg_1}\Delta_\ast\bX)$ is canonically isomorphic to the completion of the subspace $C_c(\Delta_\ast X)$ whose inner product takes values in $C_0(\bX)$ inside $C_0(\bg_2\times_{\bg1}\bX)$, and that is isomorphic to $E_{\Delta_\ast X,G_1}^{G_1\times \bg_1, G_1}$. It follows, that the compression of \eqref{eq_compression_ET} is equal to  $[E_{\Delta_\ast X,G_1}^{G_1\times \bg_1, G_1} ]\otimes_{C_0(\bX)}[E^0,T^0]$ in $KK^{G_1\times \bg_1}(C_0(\Delta_\ast X), A)$, and we are done.
\end{proof}

\begin{theorem}\label{thm_CLP_diagram_Khom} Let $G$ be a countable discrete group, and let $(G,\{H_{\lambda}\}_{\lambda\in\Lambda},\{N_{\lambda}\}_{\lambda\in\Lambda})$ be a Cohen--Lyndon triple satisfying property \hyperlink{cond:fin}{(FIN)}. Let $A$ be a separable $\bg$-$C^*$-algebra. Then, we have the following commutative diagram with exact rows:
\begin{equation}\label{eq_diag_equiv_Khom}
\xymatrix{
\cdots  \ar[r] & \bigoplus_{\lambda \in \Lambda}K_\ast^{H_\lambda}( EH_\lambda; A) \ar[r]^-{} \ar[d] &  K_\ast^{G}(EG; A) \ar[r] \ar[d] 
& K_\ast^{\bg}(M_\phi, Z; A) \ar[d]^{\cong}  \ar[r] & \cdots  \\
\cdots  \ar[r] & \bigoplus_{\lambda \in \Lambda}K_\ast^{\bhl}(\E \bhl; A)  \ar[r]^-{} & K_\ast^{\bg}(\E \bg; A)  \ar[r] 
& K_\ast^{\bg}(Y, M_\theta; A) \ar[r] & \cdots
}
\end{equation}
which is canonically isomorphic to the commutative diagram \eqref{eq_lem_diagram}. Here, all the four maps in the first square are the obvious maps obtained by the induction homomorphisms followed by the classifying maps.
\end{theorem}
\begin{proof} 

Using the induction isomorphisms for the quotient maps $G\to \bg$ and $H_\lambda \to \bhl$, the diagram \eqref{eq_lem_diagram} is isomorphic to the diagram \eqref{eq_diag_equiv_Khom}. The assertion on the corresponding maps in \eqref{eq_diag_equiv_Khom} in the first square follows by the naturality of the induction homomorphisms (for quotient homomorphisms), proven by Lemma \ref{lem_quotient_ind_natural}, and by the compatibility with the induction for injective group homomorphisms, proven by Lemma \ref{lem_natural_ind_ind}.
\end{proof}

\begin{corollary}\label{cor_CLP_BCC} Let $G$ be a countable discrete group, and let $(G,\{H_{\lambda}\}_{\lambda\in\Lambda},\{N_{\lambda}\}_{\lambda\in\Lambda})$ be a Cohen--Lyndon triple  satisfying property \hyperlink{cond:fin}{(FIN)}. Suppose that $G$ is torsion-free. Suppose that $G$ and $\bg$ satisfy the Baum--Connes conjecture with coefficients. Then, for any separable $\bg$-$C^*$-algebra $A$, we have the following commutative diagram with exact rows:
\begin{equation}\label{eq_diagram_RHS}
\xymatrix{
\cdots  \ar[r] & \bigoplus_{\lambda \in \Lambda}K_\ast(A\rtimes_r H_\lambda) \ar[r]^-{} \ar[d] & K_\ast(A\rtimes_r G) \ar[r] \ar[d] 
& K_\ast^{\bg}(M_\phi, Z; A)  \ar[d]^{\cong}  \ar[r] & \cdots  \\
\cdots  \ar[r] & \bigoplus_{\lambda \in \Lambda}K_\ast(A\rtimes_r \bhl)  \ar[r]^-{} & K_\ast(A\rtimes_r \bg)  \ar[r] 
& K_\ast^{\bg}(Y, M_\theta; A) \ar[r] & \cdots.
}
\end{equation}

Here, in the first square, the horizontal maps are induced by the canonical inclusions $A\rtimes_r H_\lambda \to A\rtimes_r G$, $A\rtimes_r \bhl \to A\rtimes_r \bg$. The second vertical map is the unique map that makes the following diagram commutes
\begin{equation}\label{eq_CLP_diagram_Khom_A}
\xymatrix{
K_\ast^{G}(EG; A) \ar[r]^-{\cong}_{\mu^G_{A,r}} \ar[d] &  K_\ast(A\rtimes_r G) \ar[d] \\
K_\ast^{\bg}(\E \bg; A) \ar[r]^-{\cong}_{\mu^{\bg}_{A, r}} &  K_\ast(A\rtimes_r \bg), 
}
\end{equation}
where the left vertical map is the induction for the quotient map $\alpha\colon G\to \bg$ followed by the classifying map, and similarly for the first vertical map.
\end{corollary}
\begin{proof} The only nontrivial assertion is the one on the horizontal maps in the first square, and this follows from the known fact (see the proof of \cite[Theorem 6.3]{BMP02}) that the following diagram commutes for any inclusion $G_1\to G_2$ of countable groups and for any separable $G_2$-$C^*$-algebra $B$,
\begin{equation*}
\xymatrix{
K_\ast^{G_1}(\E G_1; B) \ar[r]^-{\cong}_{\mu^{G_1}_{B,r}} \ar[d] &  K_\ast(B\rtimes_r G_1) \ar[d] \\
K_\ast^{G_2}(\E G_2; B) \ar[r]^-{\cong}_{\mu^{G_2}_{B, r}} &  K_\ast(B\rtimes_r G_2), 
}
\end{equation*}
where the first vertical map is induction map followed by the classifying map and the second vertical map is induced by the inclusion $B\rtimes_r G_1 \to B\rtimes_rG_2$.
\end{proof}

\begin{remark} In \Cref{cor_CLP_BCC}, one can alternatively assume that $\cd G < \infty$. Indeed, in this case, $G$ is torsion-free, and the Cohen–Lyndon triple satisfies \hyperlink{cond:fin}{(FIN)} by \Cref{theorem FIN}.
\end{remark}

\begin{remark}\label{rem_CLP_coeff} In \Cref{cor_CLP_BCC}, if we fix a coefficient $\bg$-$C^*$-algebra $A$, we only need that all groups $H_\lambda$, $\bhl$, $G$ and $\bg$ satisfy the Baum--Connes conjecture for the specific coefficient $A$. In many cases of interest, $H_\lambda$, $\bhl$, $G$ all satisfy the BCC with coefficients, and in that case, we only need the BC (without coefficient) for $\bg$ to have the diagram \eqref{eq_CLP_diagram_Khom_A} for $A=\bC$.
\end{remark}

\subsection{K-theory of Cohen--Lyndon aspherical groups}\label{sec: CLA} 


A group presentation 
$$P = \langle s_1, \dots, s_n \mid r_1, \dots, r_k \rangle$$ 
is said to be {\it Cohen--Lyndon aspherical (CLA)} if $(G,\{H_i\}_{i=1}^k,\{N_i\}_{i=1}^k)$ is a Cohen--Lyndon triple, where $G= \langle s_1, \dots, s_n \rangle $, $N_i=\langle r_i \rangle$ and $H_i$ is the normaliser of $N_i$ in $G$. Since $G$ is a free group, for each $i$, we have $H_i=\langle r_i'\rangle$ and $r_i=(r_i')^{d_i}$ for some $d_i\geq1$.

Cohen--Lyndon asphericity was introduced in \cite[\S III.10]{LS77}, and it implies standard notions of combinatorial asphericity conditions for group presentations, as discussed in \cite{CCH81}. Observe that, by \Cref{prop CL}, one-relator groups as well as groups satisfying small cancellation conditions  $C(6)$, $C(4)-T(4)$, or $C(3)-T(6)$ are (CLA). Applying \Cref{theorem FIN}, we obtain:

\begin{proposition}\label{prop: CLA} Let $P = \langle s_1, \dots, s_n \mid r_1, \dots, r_k \rangle $ be (CLA), then $P$ satisfies {\rm \hyperlink{cond:fin}{(FIN)}}. 
\end{proposition}

\begin{corollary}[See Corollary \ref{cor_CLP_BCC}] Let $P = \langle s_1, \dots, s_n \mid r_1, \dots, r_k \rangle $ be (CLA). Let $G= \langle s_1, \dots, s_n \rangle $, $N_i=\langle r_i  \rangle$ and $H_i$ be the normaliser of $N_i$ in $G$. Suppose $\bg=P$ satisfies the Baum--Connes conjecture with coefficients. Then, for any separable $\bg$-$C^*$-algebra $A$, we have the following commutative diagram with exact rows:
\begin{equation}\label{eq_CLA_A}
\xymatrix{
\cdots  \ar[r] & \bigoplus_{1\leq i \leq k}K_\ast(A\rtimes_r H_i)  \ar[d] \ar[r] & K_\ast(A\rtimes_r G) \ar[d] \ar[r] & K_\ast^{\bg}(M_\phi, Z; A) \ar[d]^{\cong} \ar[r] & \cdots \\
\cdots  \ar[r] &  \bigoplus_{1\leq i \leq k} K_\ast(A\rtimes_r \bh_i) \ar[r] &  K_\ast(A\rtimes_r \bg)  \ar[r] & 
K_\ast^{\bg}(Y, M_\theta; A) \ar[r] & \cdots.
}
\end{equation}
\end{corollary}
 

For brevity, we focus on the simplest case $A = \bC$. In this case, we only need to assume the BC without coefficient for $\bg=P$ to have \eqref{eq_CLA_A} (Remark \ref{rem_CLP_coeff}). Let us simplify notations and let $K^?_\ast(X, X_0)\coloneq K^?_\ast(X, X_0; \bC)$. In this case, the diagram \eqref{eq_CLA_A} is isomorphic to the following commutative diagram

\begin{equation}\label{eq_sc0}
\xymatrix{
\cdots  \ar[r] &\bigoplus_{1\leq i \leq k} K_\ast(C^*_r(H_i))  \ar[d] \ar[r] & K_\ast(C^*_r(G)) \ar[d] \ar[r] & K_\ast(BG, B\h) \ar[d]^{=} \ar[r] & \cdots \\
\cdots  \ar[r] & \bigoplus_{1\leq i \leq k}K_\ast(C^*_r(\bh_i)) \ar[r] & K_\ast(C^*_r(\bg))  \ar[r] & K_\ast(BG, B\h) \ar[r] & \cdots 
}
\end{equation}
which is isomorphic to
\begin{equation}\label{eq_sc1}
\xymatrix{
\cdots  \ar[r] & \bigoplus_{1\leq i \leq k} K_\ast(BH_i)  \ar[d] \ar[r] & K_\ast(BG) \ar[d] \ar[r] & K_\ast(BG, B\h) \ar[d]^{=} \ar[r] & \cdots \\
\cdots  \ar[r] &  \bigoplus_{1\leq i \leq k}K^{\bh_i}_\ast(\E\bh_i) \ar[r] &  K^{\bg}_\ast(\E\bg)  \ar[r] & K_\ast(BG, B\h)\ar[r] & \cdots 
}
\end{equation}

Here, we use the induction isomorphisms (using the freeness of the actions)
\[
K_\ast^{\bg}(M_\phi, Z) \cong  K_\ast^G(EG, \bigsqcup_{\lambda \in\Lambda}G\times_{H_\lambda}EH_\lambda ) \
\cong K_\ast(BG, B\h)
\]
where $B\h \coloneq \bigsqcup_{\lambda \in \Lambda} BH_\lambda$.

To proceed, we use the following explicit identifications (c.f. \cite[Chapter 7]{BC02})
\begin{equation*}
\begin{array}{cc}
K_0(BH_i)\cong K_0(\{\cdot\})\cong \mathbb{Z}, & K_1(BH_i)\cong H_i (\cong \mathbb{Z}), \\
K_0(BG)\cong K_0(\{\cdot\}) \cong \mathbb{Z}, & K_1(BG) \cong G^{\ab} (\cong \mathbb{Z}^n),\\
K^{\bh_i}_0(\E\overline{H_i}) =  K^{\bh_i}_0(\{\cdot\}) \cong R(\bh_i) \cong K_0(C^*(\overline{H}_i)), & K^{\bh_i}_1(\E\overline{H}_i) = K^{\bh_i}_1(\{\cdot\}) \cong 0,
\end{array}
\end{equation*}
where $R(\overline{H_i})$ is the (complex) representation ring of $\overline{H_i}$. Here, under the isomorphism $R(\bh_i) \cong K_0(C^*(\overline{H}_i))$, the unit class $[1_{C^*(\overline{H}_i)}]$ in $K_0(C^*(\overline{H}_i))$ corresponds to the class $[\lambda_{\overline{H}_i}]$ of the regular representation. The image of the vertical map $\bZ\cong K_0(\{\cdot\})\cong K_0(BH_i)\cong K_0^{H_i}(EH_i)\to K^{\bh_i}_0(\E\overline{H_i}) =  K^{\bh_i}_0(\{\cdot\}) \cong R(\bh_i)$ consists of $\bZ[\lambda_{\overline{H}_i}]$. This is because the map factors through $K_0(\{\cdot\}) \cong K^{H_i}_0(H_i)\cong K_0^{\bh_i}(\bh_i)\to K^{\bh_i}_0(\{\cdot\})$.

With these, the diagram \eqref{eq_sc1} is identified as the following: 

\begin{equation}\label{eq_relator1}
\xymatrix{
\cdots  \ar[r]^-{\partial_1} & \bigoplus_{1\leq i \leq k}\bZ  \ar[d]^{\bigoplus_{1\leq i\leq k}[\lambda_{\bh_i}]\cdot} \ar[r]^-{+} & \bZ \ar[d] \ar[r] & K_0(BG, B\h) \ar[d]^{=} \ar[r]^-{\partial_0} & \cdots \\
\cdots  \ar[r]^-{d_1} &  \bigoplus_{1\leq i \leq k} R(\bh_i) \ar[r] & K^{\bg}_0(\E\overline{G})  \ar[r] & K_0(BG, B\h) \ar[r]^-{d_0} & \cdots 
}
\end{equation}

\begin{equation}\label{eq_relator2}
\xymatrix{
\cdots  \ar[r]^-{\partial_0} &  \bigoplus_{1\leq i \leq k} H_i  \ar[d] \ar[r] & G^\ab \ar[d] \ar[r] & K_1(BG, B\h) \ar[d]^{=} \ar[r]^-{\partial_1} & \cdots \\
\cdots  \ar[r]^-{d_0} & 0 \ar[r] & K^{\bg}_1(\E\overline{G})  \ar[r] & K_1(BG, B\h) \ar[r]^-{d_1} & \cdots 
}
\end{equation}

We see that $\partial_0$ is injective, and it induces an isomorphism
\[
K_0(BG, B\h)\cong \ker( \bigoplus_{1\leq i \leq k}H_i \to G^\ab).
\]
We see that the image of $\partial_1$ consists of the elements in $\bigoplus_{1\leq i \leq k}\bZ$ whose sum is $0$, and we have the following exact sequence
\[
0\to K_0(BG, B\h) \to \bigoplus_{1\leq i \leq k} H_i \to G^{\ab} \to K_1(BG, B\h)  \to \bZ^{k-1}\to 0
\]
We have $d_0=0$ trivially. The image of $d_1$ in $\bigoplus_{1\leq i \leq k} R(\bh_i)$ is generated by the differences of the classes $[\lambda_{\overline{H}_i}]$ in $R(H_i)$, which we denote by $\langle [\lambda_{\overline{H}_i}] - [\lambda_{\overline{H}_j}]\rangle$; it is isomorphic to $\bZ^{k-1}$. This is because $d_1$ factors through $\partial_1$. From these, we obtain the following exact sequences
\[
0\to \langle [\lambda_{\overline{H}_i}] - [\lambda_{\overline{H}_j}]\rangle\to  \bigoplus_{1\leq i \leq k} R(\bh_i) \to K_0^{\bg}(\E\bg) \to K_0(BG, B\h)\to 0,
\]
\[
0\to  K_1^{\bg}(\E\bg) \to K_1(BG, B\h)\to \bZ^{k-1} \to 0.
\]
Denote by $\pi:G\to G^\ab$ and  $\bar\pi:\bg\to \bg^\ab$ the abelianisation homomorphisms. Since $\ll \mathcal N \rr \leq \ll \mathcal H \rr \lhd G$, we have the natural isomorphism:
\begin{align*}
\tifrac{G^\mathrm{ab}}{\sum_{i=1}^k\pi(H_i)}
&\cong \tifrac{G}{[G,G]\cdot \llrr{\mathcal H}}\\
&\cong \tifrac{{\tifrac{G}{\llrr{\mathcal N}}}}{{\tifrac{([G,G]\cdot \llrr{\mathcal H})}{\llrr{\mathcal N}}}}\\
&\cong \tifrac{\bg}{[\bg, \bg]\cdot \llrr{\overline{\mathcal H}}}\\
&\cong\tifrac{\bg^\mathrm{ab}}{\sum_{i=1}^k\bar\pi(\bh_i)}.
\end{align*}

Combining these we see that
\[
K_1^{\bg}(\E\bg) \cong \tifrac{G^\mathrm{ab}}{\sum_{i=1}^k\pi(H_i)}\cong \tifrac{\bg^\mathrm{ab}}{\sum_{i=1}^k\bar\pi(\bh_i)},
 \]
 and
 \[
K_0^{\bg}(\E\bg)  \cong R \oplus \ker( \bigoplus_{1\leq i \leq k}H_i \to G^\ab),
\]
where $R$ is the quotient of $\bigoplus_{1\leq i \leq k} R(\bh_i)$ by $\langle [\lambda_{\overline{H}_i}] - [\lambda_{\overline{H}_j}]\rangle$ (identifying the classes of the regular representations). Here, we used the projectivity of the free abelian group $\ker( \bigoplus_{1\leq i \leq k}H_i \to G^\ab)$.

Thus, we have obtained the following:

\begin{theorem}\label{thm: CLA} Suppose $P = \langle s_1, \dots, s_n \mid r_1, \dots, r_k \rangle$ is (CLA) and satisfies the Baum--Connes conjecture (without coefficients). Let $G= \langle s_1, \dots, s_n \rangle $, $N_i=\langle r_i \rangle$ and $H_i$ is the normaliser of $N_i$ in $G$. Let $\bh_i=H_i/N_i$ and $\bg=G/\ll \mathcal{N} \rr=P$. Denote by $\pi:G\to G^\ab$ and  $\bar\pi:\bg\to \bg^\ab$ the abelianisation homomorphisms. 
\begin{enumerate}
\item In general, we have
 \[
K_0(C^*_r(\bg))  \cong K_0^{\bg}(\E\bg)  \cong R \oplus \ker( \bigoplus_{1\leq i \leq k}H_i \to G^\ab),
\]
\[
K_1(C^*_r(\bg)) \cong K_1^{\bg}(\E\bg) \cong \tifrac{G^\mathrm{ab}}{\sum_{i=1}^k\pi(H_i)}\cong \tifrac{\bg^\mathrm{ab}}{\sum_{i=1}^k\bar\pi(\bh_i)},
 \]
\noindent where $R$ is the quotient of $\bigoplus_{1\leq i \leq k} R(\bh_i)$ by $\langle [\lambda_{\overline{H}_i}] - [\lambda_{\overline{H}_j}]\rangle$ (identifying the classes of the regular representations). In particular, we have $K^{\bg}_1(\E\bg)\cong 0$ if and only if $\bigoplus_{1\leq i \leq k} H_i$ surjects onto $G^\ab$.\\

\item Suppose $\pi(r_i)=0$ for all $1\leq i\leq k$.  Then,  we have
\[
 K_0(C^*_r(\bg)) \cong  K^{\bg}_0(\E\overline{G})\cong R \oplus \bigoplus_{1\leq i \leq k}H_i,
\]
\[
K_1(C^*_r(\bg)) \cong K^{\bg}_1(\E\bg) \cong   G^\ab \cong \bg^\ab.
\]
\end{enumerate}
\end{theorem}

This streamlines and generalises the $K$-theory/$K$-homology computations for one relator groups done in \cite{BBV99}, \cite{MV03}. 

\begin{remark}\label{rmk:CLA_applications} We point out that besides one relator groups (\cite{Tu99Tree}, \cite{BBV99}, \cite{Oyo01}), there are other known classes of groups with (CLA) presentations that satisfy the BCC. For instance, groups with $C(7)$ small-cancellation presentations are hyperbolic (see e.g.\,\cite{AJW24}) and hence satisfy the BCC by \cite[Theorem 0.4]{Lafforgue12}. Also, groups with $C'(1/4)-T(4)$ small-cancellation presentations act properly and cocompactly on a CAT(0) cube complex by \cite[Theorem 1.2]{wise06} and hence satisfy the BCC by \cite{NR97} and \cite{HK01}. Therefore,  \Cref{thm: CLA} applies to these classes of groups.
\end{remark}

\section{Baum--Connes conjecture with finite wreath products}\label{sec: permanence}


Following the formulation of the Farrell--Jones Conjecture with finite wreath products \cite[Conjecture 13.27]{Lueck25}, we introduce \emph{the Baum--Connes conjecture with coefficients with finite wreath products}, the \emph{BCC with finite wreath products} in short. The content of this section is essentially a replicate of the treatment in \cite[Section 13.7]{Lueck25} in the context of the Baum--Connes conjecture. In the book, a systematic implementation of the finite wreath products version has already been suggested (see \cite[Remark 16.10]{Lueck25}). Our primary reason for introducing this version is because we believe that our results on the Baum--Connes conjecture for relatively hyperbolic groups, which will be treated in Section \ref{sec_CL_BCC}, is best formulated for the BCC with finite wreath products.

For any discrete groups $G, F$ and for any $F$-set $\Omega$, we denote by $G\wr_\Omega F$, the (unrestricted) wreath product $(\prod_{\Omega}G)\rtimes F$. The restricted wreath product $(\oplus_{\Omega}G)\rtimes F$ is denoted by $G\mathrm{wr}_\Omega F$, where $\bigoplus_{\Omega}G$ denotes the subgroup of $\prod_{\Omega} G$ consisting of tuples with finite support (note this is not the coproduct in $\mathbf{Grp}$). When $\Omega = F$, equipped with the left-translation action of $F$, we denote it by $G\wr F$, instead of $G\wr_F F$.

If $N$ is a normal subgroup of $G$, then by the Krasner--Kaloujnine universal embedding theorem \cite{Kras_Kal_51}, $G$ can be embedded as a subgroup of $N\wr (G/N)$. It follows that, if $H$ is a subgroup of $G$ of finite index, $G$ can be embedded as a subgroup of $N\wr (G/N)$ for some finite index normal subgroup $N\subset H$ of $G$. 

Suppose $1 \to K  \to G \to Q \to 1$ is an extension of groups. Then, 
\[
1 \to \prod_\Omega K   \to G \wr_\Omega  F \to Q \wr_\Omega F \to 1
\]
is an extension of groups.

\begin{definition}[The Baum--Connes Conjecture with Coefficients with finite wreath products] We say that a (countable) discrete group $G$ satisfies the Baum--Connes conjecture with coefficients with finite wreath products if for any finite group $F$ the group $G\wr F$ satisfies the Baum--Connes conjecture with coefficients.
\end{definition}

\begin{example}\label{ex_HK} All a-T-menable groups (see \cite{CCJJV}) satisfy the BCC with finite wreath products by the Higson--Kasparov theorem \cite{HK01}.  In fact, if $G$ is a-T-menable, $G\wr F$ is a-T-menable for any finite group $F$: more generally, any extension of an amenable group by an a-T-menable group is a-T-menable \cite[Example 6.1.6]{CCJJV} (see also paragraphs before \cite[Theorem 1.1]{dCSV}). Hence,  $G\wr F$ satisfies the BCC for any finite group $F$.
\end{example}

The following is a formal counterpart of \cite[Theorem 13.32]{Lueck25} in the context of the Baum--Connes conjecture:

\begin{theorem}[Inheritance properties of the BCC with finite wreath products] \label{thm_inh_wr} Let $\mathfrak{BCCwr}$ be the class of countable discrete groups satisfying the Baum--Connes conjecture with coefficients with finite wreath products. Then, the following holds:
\begin{enumerate}
\item\label{item_wr_sub} If $H$ is a subgroup of $G$ and if $G \in \mathfrak{BCCwr}$, then $H \in \mathfrak{BCCwr}$; 
\item\label{item_wr_wr} If $G \in \mathfrak{BCCwr}$, then for any finite group $F$, $G \wr F \in \mathfrak{BCCwr}$; 
\item\label{item_wr_over} Suppose $H$ is a subgroup of $G$ of finite index. Then, $G \in \mathfrak{BCCwr}$ if and only if $H \in \mathfrak{BCCwr}$; 
\item\label{item_wr_commen} Suppose that two groups $G_0$ and $G_1$ are commensurable. Then $G_0 \in \mathfrak{BCCwr}$ if and only if $G_1 \in \mathfrak{BCCwr}$; 
\item\label{item_wr_prod} If $G_0$ and $G_1$ belong to $\mathfrak{BCCwr}$, then $G_0\times G_1 \in \mathfrak{BCCwr}$; 
\item\label{item_wr_ext} Let $1 \to K  \to G \to Q \to 1$ be an extension of groups. Suppose $K$ and $Q$ belong to $\mathfrak{BCCwr}$. Then $G \in \mathfrak{BCCwr}$; 
\item\label{item_wr_uni} Suppose that $G$ is the union of an increasing sequence $(G_n)_{n\geq1}$ of subgroups of $G$ where $G_n \in  \mathfrak{BCCwr}$. Then $G  \in \mathfrak{BCCwr}$; 
\item\label{item_wr_gen} If $G, H \in \mathfrak{BCCwr}$, then for any countable $H$-set $\Omega$, the restricted wreath product $G \mathrm{wr}_\Omega H \in \mathfrak{BCCwr}$; 
\item\label{item_wr_free} Let $\{G_i \,\mid \,i\in I\}$ be a countable collection of groups such that $G_i$ belongs to $\mathfrak{BCCwr}$ for all $i \in I$. Then, the free product $\Asterisk_{i \in I} G_i \in \mathfrak{BCCwr}$.
\end{enumerate}
\end{theorem}
\begin{proof} 
\eqref{item_wr_sub} For any group $F$, $H\wr F$ is a subgroup of $G\wr F$. The BCC for $G\wr F$ implies the BCC for its (closed) subgroup $H\wr F$ \cite[Theorem 2.5]{CE}. 

\eqref{item_wr_wr} For any finite groups $F$ and $F'$, we have a canonical embedding $(G \wr F)\wr F' \to G \wr (F \wr F')$ (see \cite[Lemma 13.26]{Lueck25}).  The BCC for $G\wr (F\wr F')$ implies the BCC for its (closed) subgroup $(G\wr F)\wr F'$.

\eqref{item_wr_over} For any finite index subgroup $H$ of $G$, $G$ embeds into $N\wr F$ for some normal subgroup $N\subset H$ of $G$ of finite index and $F=G/N$. The assertion follows from this and \eqref{item_wr_sub}, \eqref{item_wr_wr}.

\eqref{item_wr_commen} By definition, there are subgroups $H_0\leq  G_0$ and $H_1\leq G_1$ of finite index such that $H_0$ is isomorphic to $H_1$. The assertion thus follows from \eqref{item_wr_over}.

\eqref{item_wr_prod} For any finite group $F$, $(G_0\times G_1)\wr F$ embeds into $(G_0\wr F)\times (G_1 \wr F)$.  The BCC for $G_0 \wr F$ and $G_1\wr F$ implies the BCC for their product (\cite[Theorem 3.17]{CE}, \cite[Corollary 7.12]{Oyono01}). Thus, the subgroup $(G_0\times G_1)\wr F$ satisfies the BCC.

\eqref{item_wr_ext} For any finite group $F$, we consider an extension
\[
1 \to \prod_FK  \to G \wr F \to Q \wr F \to 1.
\]
By assumption and \eqref{item_wr_over}, \eqref{item_wr_prod}, $Q\wr F$ and any subgroup of $G\wr F$ containing $\prod_FK$ with finite index satisfy the BCC. Hence, $G\wr F$ satisfies the BCC by the permanence properties of the BCC for the extensions of groups (see \cite[Theorem 2.1]{CEO04}, \cite[Theorem 3.1.]{Oyono01}).

\eqref{item_wr_uni} For any finite group $F$, $G\wr F$ is the increasing union of subgroups $G_n\wr F$ which satisfy the BCC. Hence, the union $G\wr F$ satisfies the BCC by \cite[Theorem 6.3]{BMP02}, \cite[Theorem 10.4]{MN06}, or \cite[Theorem 0.8]{BEL08}.

\eqref{item_wr_gen} We consider an extension 
\[
1 \to \bigoplus_\Omega G  \to G \mathrm{wr}_\Omega  H \to H\to 1. 
\]
The assertion follows by assumption and by \eqref{item_wr_prod}, \eqref{item_wr_ext}, \eqref{item_wr_uni}.

\eqref{item_wr_free} The only-if part follows by \eqref{item_wr_sub}. For the other direction, We first consider the free product $G_0 \Asterisk G_1$ of two groups $G_0$ and $G_1$ in $\mathfrak{BCCwr}$. Consider the following extension:
\[
1 \to N\to G_0 \Asterisk G_1 \to G_0 \times G_1 \to 1.
\]
Then, $N$ is a free group (see proof of \cite[Lemma 3.21.]{Kuhl09}). The free groups belong to  $\mathfrak{BCCwr}$ since they are a-T-menable groups. By \eqref{item_wr_prod} and \eqref{item_wr_ext}, the free product $G_0 \Asterisk G_1 \in \mathfrak{BCCwr}$. The general case follows from this by induction and by \eqref{item_wr_uni}. 
\end{proof}

\begin{remark} It follows that a discrete group $G$ satisfies the BCC with finite wreath products if and only if all finite products $G^n$ and their finite extensions satisfy the BCC. This is, a priori, stronger than saying all finite extensions of $G$ satisfy the BCC. 
\end{remark}

\begin{remark} Consider the following statement: Suppose $G$ acts on a countable (not necessarily locally finite) oriented tree. Then $G \in \mathfrak{BCCwr}$ if and only if, for every vertex $v$ of the tree, the stabilizer $G_v = \mathrm{Stab}_G(v)$ belongs to $\mathfrak{BCCwr}$. This can be shown by combining the equivariant power functor (Definition \ref{def_wreath_func}), introduced in Section \ref{sec_hyp_BCC}, with the powerful results established by Pimsner in \cite{Pimsner}. However, the detailed argument is somewhat technical, and we do not provide the proof in this article.
\end{remark}

\section{Hyperbolic groups satisfy the BCC with finite wreath products}\label{sec_hyp_BCC}

The following is a deep result of Vincent Lafforgue \cite{Lafforgue12} which says in particular that all discrete hyperbolic groups satisfy the BCC (see also the survey paper \cite{Puschnigg14}):

\begin{theorem}[{\cite[Theorem 0.4]{Lafforgue12}}]\label{thm_Laff_BCC} The Baum--Connes conjecture with coefficients holds for any second countable, locally compact group $G$ that acts isometrically, continuously and properly on a weakly geodesic, uniformly locally finite hyperbolic metric space. 
\end{theorem}

The goal of this section is to provide our proof of the following extension of the theorem above. For any locally compact groups $G$, $F$, and for any finite $F$-set $\Omega$, we consider the wreath product $G \wr_\Omega F=(\prod_\Omega G) \rtimes F$. A finite $F$-set will always mean a finite discrete set equipped with a continuous $F$-action. 

\begin{theorem}\label{thm_hyp_BCC_wr} Let $G$ be any second countable, locally compact group that acts isometrically, continuously and properly on a weakly geodesic, uniformly locally finite hyperbolic metric space. Then, for any second countable, compact group $F$, and for any finite $F$-set $\Omega$, the wreath product $G\wr_\Omega F=(\prod_\Omega G) \rtimes F$ satisfies the Baum--Connes conjecture with coefficients.
\end{theorem}

For the definition of the equivariant $KK$-theory, we refer the reader to \cite{Kas88} and \cite{Bla98}. We recall from \cite[Definition 4.10]{BC2020} that for a second countable, locally compact group $G$, we say that a gamma element exists for $G$ in the classical sense if there is an element $\gamma_G$ in $KK_G(\bC, \bC)$ satisfying the following two conditions:
\begin{enumerate}
\item $\gamma_G$ factors through a separable, graded, proper $G$-$C^*$-algebra $A$: $\gamma_G=\beta \otimes_A \alpha$ for $\alpha\in KK_G(A, \bC)$ and $\beta\in KK_G(\bC, A)$;
\item $\mathrm{Res}_G^{L}(\gamma_G) = 1_L$ in $KK_L(\bC, \bC)$ for any compact subgroup $L$ of $G$. 
\end{enumerate}
We also briefly recall from \cite[Definition 8.1]{MN06} the following (a priori) slightly more general definition of the gamma element. First of all, we denote by $KK_G$, the (triangulated) category of separable, \emph{ungraded} $G$-$C^*$-algebras with morphism sets $KK_G(A, B)$. A morphism in $KK_G$ is called a weak equivalence if its restriction to $KK_L$ is a $KK_L$-equivalence for all compact subgroups $L$ of $G$. We say that \emph{a gamma element exists for $G$ in the Meyer--Nest sense} if there is $\eta\in KK_G(\bC, P)$ (a dual Dirac morphism) such that $D\otimes_\bC \eta =1_P$ for a Dirac morphism $D \in KK_G(P,\bC)$. Here, as in \cite[Definition 4.5]{MN06}, a Dirac morphism $D$ is a weak equivalence $P\to \bC$ with $P$ in $\langle \mathcal{CI}_G \rangle$, the localizing subcategory (see \cite[Section 2.5]{MN06}) of $KK_G$ generated by the full subcategory $\mathcal{CI}_G$ consisting of all separable ungraded $G$-$C^*$-algebras that are $KK_G$-equivalent to the induced algebra $\mathrm{Ind}_L^G(B)$ for some compact subgroup $L$ of $G$ and for some separable ungraded $L$-$C^*$-algebra $B$. In this case, $\gamma_G\coloneq \eta\otimes_P D$ in $KK_G(\bC, \bC)$ is called a gamma element in the Mayer--Nest sense.

If a gamma element $\gamma_G$ exists for $G$ in the classical sense, it is a gamma element in the Meyer--Nest sense \cite[Theorem 8.2]{MN06}. The gamma element (in either sense), if exists, is a unique element in $KK_G(\bC, \bC)$ characterized by the definition, and satisfies $\gamma_G^2=\gamma_G$.

The existence of a gamma element for $G$ implies the split-injectivity of the Baum--Connes assembly map $\mu^G_{A, r}$ for all coefficients $A$. If $\gamma_G=1_G$ in $KK_G(\bC, \bC)$, the BCC holds for $G$. The famous obstruction to $\gamma_G=1_G$ in $KK_G(\bC, \bC)$ is Kazhdans's property $(T)$: the gamma element $\gamma_G$ for $G$, if it exists, is not equal to $1_G$ for any non-compact group $G$ with property $(T)$.

\begin{remark} It is worth remarking that any second countable, locally compact group $G$ that admits a uniform embedding into a Hilbert space has a gamma element in the classical sense. This is \cite{Tu04} in the discrete group case. The general case follows by \cite[Theorem 3.4]{DL15} (see also \cite[Theorem 6.1]{STY02}) and \cite{Tu99}. Note that \cite[Theorem 3.5]{DL15} (\cite{STY02} in the discrete case) proves the split-injectivity of the assembly map in this setting. We explain this briefly. As in the proof of \cite[Theorem 3.5]{DL15}, for some second countable, compact convex $G$-space $X$ such that $X\rtimes G$ is a-T-menable, the unit inclusion $\iota\colon \bC\to C(X)$ is a weak equivalence in $KK_G$. The point is that $\iota$ is a dual-Dirac morphism for $G$. By \cite{Tu99}, the identity on $C(X)$ factors thorough a separable, graded proper $G$-$C^*$-algebra $A$. As in the proof of \cite[Theorem 8.2]{MN06}, this implies $D\otimes 1_{C(X)} \colon P\otimes C(X)\to C(X)$ is a $KK_G$-equivalence for a (any) Dirac morphism $D\colon P\to \bC$ for $G$.  Also, $1_P\otimes \iota \colon P\to P\otimes C(X)$ is a $KK_G$-equivalence since it is a weak equivalence between objects in $\langle \mathcal{CI}_G \rangle$. It follows $\iota\colon \bC \to C(X)$ is a dual-Dirac morphism, a weak inverse of some Dirac morphism $D'\colon C(X) \to \bC$. Moreover, since the identity on $C(X)$ factors through $A$, the gamma element exists in the classical sense as well. 
\end{remark}

We now recall from \cite[Section 1]{Lafforgue12} (see also \cite{La02}), the definition of the group $KK_{G, \ell}(A, B)$ for any separable graded $G$-$C^*$-algebras $A, B$ and for any length function $\ell$ for $G$. The set $E_{G, \ell}(A, B)$ is defined as the set of unitary isomorphism classes of triples $(E, \pi, T)$ satisfying the following: $E$ is a countably generated, graded Hilbert $A$-$B$-module, $\pi$ is a graded, (not necessarily isometric, not necessarily adjointable) continuous $G$-action on $E$ as the $B$-pair $(\overline{E}, E)$ (see Remark \ref{rem_El} below),  with operator-norm bound $\| \pi(g) \| \leq e^{l(g)}$ for $g\in G$, such that the representation of $A$ on $E$ is $G$-equivariant, and $T\in \Linears(E)$ (the algebra of adjointable operators on $E$) is an odd, (not necessarily self-adjoint), bounded, adjointable operator on $E$ such that for any $a\in A$, the operators $[a, T]$, $a(1-T^2)$ belong to $\mathcal{K}(E)$ (the algebra of $B$-compact operators on $E$) and $g \mapsto a(T-g(T))$ is a norm-continuous map from $G$ to $\mathcal{K}(E)$. The homotopy equivalence relation on $E_{G, \ell}(A, B)$ is defined by the cycles of $E_{G, \ell}(A, B[0,1])$ where $B[0,1]\coloneq C([0,1], B)$: two cycles $(E_i, \pi_i, T_i)$, $i=0,1$ in $E_{G, \ell}(A, B)$  are homotopic if and only if there is a cycle $(E, \pi, T)$ in $E_{G, \ell}(A, B[0,1])$ whose evaluation at $t=i$ is isomorphic to $(E_i, \pi_i, T_i)$ for $i=0,1$. The abelian group $KK_{G, \ell}(A, B)$ is the set of the homotopy equivalence classes of the cycles in $E_{G, \ell}(A, B)$, equipped with the direct sum operation.

In particular, $E_{G, \ell}(\bC, \bC)$ consists of the isomorphism classes of cycles $(H, \pi, T)$ where $H$ is a separable, graded Hilbert space equipped with a continuous $G$-action with $\| \pi(g) \| \leq e^{l(g)}$, $T$ is an odd, bounded operator on $H$ such that $1-T^2$ is compact and $g\mapsto g(T)-T$ is a norm-continuous map from $G$ to $\mathcal{K}(H)$. 

\begin{remark}\label{rem_El} Even though $A, B$ are $G$-$C^*$-algebras, we need to consider a $G$-action on the $B$-pair $(\overline{E}, E)$ in the sense of \cite{La02} (see [Proposition 1.1.4 and Section 1.2]\cite{La02}). In our setting, this is equivalent to saying that we have two continuous (left) $G$-actions $e \mapsto g^{\rangle}(e)$ and $f \mapsto  g^{\langle}(f)$ on $E$ satisfying
\[
\s{g^{\langle}f, g^{\rangle}e}_B = g(\s{f, e}_B)
\]
for $g\in G$, $e, f\in E$.
This is redundant for $A=B=\bC$, but it is a necessary replacement of the usual axiom $\s{g(f), g(e)}_B =g(\s{f, e}_B)$ for $e, f\in E$, and $g \in G$ when the $G$-actions are no longer unitary. This guarantees that the canonical $G$-actions on $\Compacts(E)$ and on $\Linears(E)$ are well-defined: we have two $G$-actions on $\Linears(E)$ given by $g^{\rangle}(T)\coloneq g^{\rangle}T (g^{-1})^{\rangle}$, $g^{\langle}(T)\coloneq g^{\langle}T (g^{-1})^{\langle}$. We have
\[
\s{x, g^{\rangle}(T)y}_B = \s{g^{\langle}(T^*)x, y}_B
\]
for any $x, y\in E$. That is, we have $(g^{\rangle}(T))^* = g^{\langle}(T^*)$. For any rank-one operator $\theta_{x, y}$, we have $g^{\rangle}(\theta_{x, y})=\theta_{g^{\rangle}x, g^{\langle}y}$. When the $G$-action on $B$ is trivial, for example if $B=\bC$ or $C[0,1]$, such a $G$-action on $E$ as the $B$-pair $(\overline{E}, E)$ is same as a $G$-action on $E$ by adjointable operators: if $g^{\rangle}$ is the left multiplication by $u_g\in \Linears(E)$, $g^{\langle}$ is the left multiplication by $u_{g^{-1}}^*$. When the $G$-action on $B$ is non-trivial, $g^{\rangle}$ is not adjointable as a map on $E$ but it is adjointable as a map from $E$ to $E'$ where $E'$ is a Hilbert $B$-module which is equal to $E$ as a Banach space and equipped with the following Hilbert $B$-module structure: $e{\cdot}b\coloneq eg(b)$, $\s{f, e}_B'\coloneq g^{-1}(\s{f, e}_B)$. Then, $g^{\rangle}\colon E\to E', e\mapsto g^{\rangle}(e)$ is adjointable with adjoint $(g^{-1})^{\langle}$: we have
\[
\langle f, g^{\rangle}e\rangle_B'= g^{-1}(\langle f, g^{\rangle}e  \rangle_B) = \s{(g^{-1})^{\langle}f, e}_B.
\]
In particular, the operator norm of $g^{\rangle}$ on $E$ is equal to the square-root of the operator norm of the positive operator $(g^{-1})^{\langle} \cdot g^{\rangle } \in \Linears(E)$. We have an analogous remark for $g^{\langle}$. From these, we see that $g^{\rangle}\otimes1$ and $g^{\langle}\otimes1$ define a continuous $G$-action on the Hilbert $B\otimes C$-module $E\otimes X$ as the $B\otimes C$-pair for any $C^*$-algebra $C$ with the trivial $G$-action and for any Hilbert $C$-module $X$, and the operator norm of $g^{\rangle}\otimes1$ is equal to that of $g^{\rangle}$, and similarly for the graded tensor product.

In \cite{Laff10}, \cite{Lafforgue12}, only the coefficients $A=B=\bC$, and $B=[0,1]$ (for defining the homotopy equivalence relation) are relevant to the main results. Hence, these nuances were not so essential and not emphasized. The same remark applies to us: for our main results, we only need $KK_{G, \ell}(\bC, \bC)$ but we work in a more general setting to facilitate future applications.
\end{remark}

\begin{remark}
Another remark is that the definition on $E_{G,\ell}(A, B)$ only involves conditions on $T$, and no extra conditions on $T^*$ since these are redundant: for example if $g^{\rangle}(T)-T$ is compact, then $g^{\langle}(T^*)-T^*$ is compact. Hence, we simply write $g$ to mean  $g^{\rangle}$ or $\pi(g)$, and use the notation $g(T)$ for $g^{\rangle}(T)$. We certainly do not require $T$ to be essentially self-adjoint, and that would be unnatural for the non-unitary $G$-actions setting.
\end{remark}

\begin{theorem}[{\cite[Theorem 2.3]{Laff10}, see also \cite[Theorem 1.2]{Lafforgue12}}] \label{thm_gamma_Laff} Let $G$ be a second countable, locally compact group $G$ that acts continuously and isometrically on a metric space $(X, d)$ of finite asymptotic dimension with linear control. Let $x_0\in X$ and let $\ell$ be the length function defined by $\ell(g)=d(x_0, gx_0)$. Suppose that an element $y\in KK_G(\bC, \bC)$ satisfies the following condition: for any $s>0$, there is $C\geq0$ such that $y=1_G$ in $KK_{G, s\ell+C}(\bC, \bC)$. Then, 
\[
(y\otimes_\bC 1_A) \rtimes_rG \in KK(A\rtimes_rG, A \rtimes_rG)
\]
acts as the identity on $K_\ast( A \rtimes_rG)$ for any separable ungraded $G$-$C^*$-algebra $A$. In particular, if in addition $y$ is a gamma element $\gamma_G$ of $G$, the BCC holds for $G$.
\end{theorem}

See \cite[Definition 2.1]{Laff10} for the definition of finite asymptotic dimension with linear control. Any weakly geodesic, uniformly locally finite, hyperbolic metric space $(X, d)$ has finite-asymptotic dimension with linear control (see \cite[Proposition 2.10]{Laff10}). If metric spaces $X_1, \ldots, X_m$ have finite asymptotic dimension with linear control, then so does their product, equipped with any of the $l^p$-product metric for $1\leq p \leq \infty$ (see Remarque after \cite[Corollary 2.12]{Laff10}). In particular, any finite product of weakly geodesic, uniformly locally finite, hyperbolic metric spaces (with $\ell^1$-product metric) has finite asymptotic dimension with linear control. 

The following is our main technical tool to prove Theorem \ref{thm_hyp_BCC_wr}:

\begin{theorem}\label{thm_gamma_wr} Let $G$ be a second countable, locally compact group. Let $\ell$ be any length function on $G$. Suppose $G$ has a gamma element $\gamma_G \in KK_G(\bC, \bC)$ in the Meyer--Nest sense. Then, for any second countable, compact group $F$ and for any finite $F$-set $\Omega$, $G\wr_\Omega F$ has a gamma element $\gamma_{G\wr_\Omega F}$ in the Meyer--Nest sense. Moreover, if $\gamma_G =1_G$ in $KK_{G, \ell}(\bC, \bC)$, then $\gamma_{G\wr_\Omega F} =1_{G\wr_\Omega F}$ in $KK_{G\wr_\Omega F, \tilde\ell}(\bC, \bC)$ where $\tilde \ell$ is the length function on $G\wr_\Omega F$ defined by $\tilde \ell(g) \coloneq \sum_{\omega \in \Omega}\ell(g_\omega)$ for $(g_\omega)_{\omega \in \Omega}$ in $\prod_\Omega G \subset G \wr_\Omega F$ and $\tilde \ell (f) \coloneq 0$ for $f \in F$.
\end{theorem}

The proof is given at the end of this section after several preparations. We first give immediate consequences of Theorem \ref{thm_gamma_wr}:

\begin{corollary}\label{cor_gamma_to_wreath} Let $G$ be a second countable, locally compact group $G$ that acts continuously and isometrically on a metric space $(X, d)$ of finite asymptotic dimension with linear control. Let $x_0\in X$ and let $\ell$ be the length function defined by $\ell(g)=d(x_0, gx_0)$. Suppose that $G$ has a gamma element $\gamma_G$ in $KK_G(\bC, \bC)$ in the Meyer--Nest sense, and that $\gamma_G$ satisfies the following condition: for any $s>0$, there is $C\geq0$ such that $\gamma_G=1_G$ in $KK_{G, s\ell+C}(\bC, \bC)$. Then, the BCC with finite wreath products holds for $G$.
\end{corollary}
\begin{proof} This follows directly from Theorem \ref{thm_gamma_Laff} and Theorem \ref{thm_gamma_wr}. Note that the length function $\tilde \ell$ on $G\wr_\Omega F$ of Theorem \ref{thm_gamma_wr} is equal to $g\mapsto \tilde{d}((x_0)_\omega, g\cdot (x_0)_\omega)$ for the product metric space $(X^\Omega, \tilde{d})$ of $(X, d)$, equipped with the canonical $G\wr_\Omega F$-action, where $\tilde d$ is the $\ell^1$-product metric. If $\gamma_G=1_G$ in $KK_{G, s\ell+C}(\bC, \bC)$, then $\gamma_{G\wr_\Omega F}=1_{G\wr_\Omega F}$ in $KK_{G\wr_\Omega F, s\tilde \ell + |\Omega|C}(\bC, \bC)$ by Theorem \ref{thm_gamma_wr} so Theorem \ref{thm_gamma_Laff} is applicable to the gamma element  $\gamma_{G\wr_\Omega F}$. 
\end{proof}

\begin{proof}[Proof of Theorem \ref{thm_hyp_BCC_wr}]
By \cite[Theorem 1.3]{Lafforgue12}, the assumptions of Corollary \ref{cor_gamma_to_wreath} are satisfied for any second countable, locally compact group $G$ that acts isometrically, continuously and properly on a weakly geodesic, uniformly locally finite hyperbolic metric space. 
\end{proof}

\begin{remark} See also \cite{CP22} for a deep discussion on the gamma element for hyperbolic groups. 
\end{remark}

The idea for proving Theorem \ref{thm_gamma_wr} is quite simple. The theorem follows immediately once we construct the analogue for $KK_{G, \ell}(A, B)$ of the equivariant power map/functor:
\[
\hat\otimes_\Omega \colon KK_G(A, B) \to KK_{G\wr_\Omega F}(A^{\hat\otimes \Omega}, B^{\hat\otimes \Omega}).
\]
It is worth remarking that, for the proof of Theorem \ref{thm_gamma_wr}, we only need this map for $KK_{G, \ell}(\bC, \bC)$ and for $KK_G(A, B)$ where $A, B$ are ungraded. We warn the reader that the map $\hat\otimes_\Omega$ is not additive unless $|\Omega|=1$.

For this purpose, we first define the equivariant power map $\hat\otimes_\Omega$ for $KK_G(A, B)$.

As usual, the commutator $[-,-]$ is graded and the tensor product $\hat\otimes$ is graded and minimal/spatial. When the algebras involved are ungraded, we also use the notation $\otimes$ in place of $\hat\otimes$. We refer the reader to \cite[Section 2]{Kas81} and \cite[Chapter 14]{Bla98} for the definitions of graded tensor products of graded $C^*$-algebras and graded Hilbert modules.

Let $\Omega$ be a finite $F$-set. It is important to clarify the canonical actions of the wreath product $G \wr_\Omega F$ on the graded tensor products $A^{\hat\otimes \Omega}$ and $E^{\hat\otimes \Omega}$, where $A$ is a graded $G$-$C^*$-algebra and $E$ is a $G$-equivariant graded Hilbert $A$-module. We emphasize that $A^{\hat\otimes \Omega}$ and $E^{\hat\otimes \Omega}$ are defined without requiring a fixed total order on the index set $\Omega$.

Let $(A_\omega)_{\omega \in \Omega}$ be a family of unital graded $\mathbb{C}$-algebras. Define the algebraic graded tensor product over $\Omega$ as the universal unital graded $\mathbb{C}$-algebra $\hat\bigotimes^{\mathrm{alg}}_{\omega \in \Omega} A_\omega$ generated by all $A_\omega$ ($\omega \in \Omega$), subject to the following relations: the inclusion maps $A_\omega \hookrightarrow \hat\bigotimes^{\mathrm{alg}}_{\omega \in \Omega} A_\omega$ are unital graded algebra homomorphisms; for any $a \in A_\omega$, $b \in A_{\omega'}$ with $\omega \ne \omega'$, the graded commutator vanishes: $[a, b]=0$. We note that, for any fixed total order on $\Omega =\{\omega_i\}_{1\leq i \leq |\Omega|}$, $ \omega_1 < \omega_2 < \dots < \omega_{|\Omega|}$, there is a canonical isomorphism $\hat\bigotimes^{\mathrm{alg}}_{\omega \in \Omega} A_\omega \cong A_{\omega_1} \hat\otimes^{\mathrm{alg}} \cdots \hat\otimes^{\mathrm{alg}} A_{\omega_{|\Omega|}}$.

If each $A_\omega$ is a unital graded $C^*$-algebra, we define $\hat\bigotimes_{\omega \in \Omega} A_\omega$ as the completion of $\hat\bigotimes^{\mathrm{alg}}_{\omega \in \Omega} A_\omega$ with respect to the spatial norm. In the non-unital case, we use the unitization $A_\omega^+$ and pass to the appropriate subalgebra. Again, with a fixed total order on $\Omega$, we obtain a canonical isomorphism:
$\hat\bigotimes_{\omega \in \Omega} A_\omega \cong A_{\omega_1} \hat\otimes \cdots \hat\otimes A_{\omega_{|\Omega|}}$. If all $A_\omega = A$, we write $A^{\hat\otimes \Omega} := \hat\bigotimes_{\omega \in \Omega} A$.

For a family $(E_\omega)_{\omega \in \Omega}$ where $E_\omega$ is a graded Hilbert $A_\omega$-module, we define the graded tensor product $\hat\bigotimes_{\omega \in \Omega} E_\omega $ as the closed subspace of $\hat\bigotimes_{\omega \in \Omega} \Linears (A_\omega \oplus E_\omega )$ (topologically) spanned by the tensor product of $E_\omega$ regarded as the subspace $\Compacts(A_\omega, E_\omega)$ of $\Linears (A_\omega \oplus E_\omega )$ by 
\[
E_\omega \cong \begin{bmatrix}0 & 0 \\ \Compacts(A_\omega, E_\omega)& 0 \end{bmatrix} \subset \Linears (A_\omega \oplus E_\omega),
\]
with the inner product $\s{e_1, e_2}\coloneq e_1^\ast e_2 \in \hat\bigotimes_{\omega \in \Omega} A_\omega \subset \hat\bigotimes_{\omega \in \Omega} \Linears (A_\omega \oplus E_\omega)$. Again, after fixing a total order on $\Omega$, this agrees canonically with the standard graded tensor product.  If all $E_\omega = E$, we write $E^{\hat\otimes \Omega} := \hat\bigotimes_{\omega \in \Omega} E$. 

With these definitions, the actions of the wreath product $G \wr_\Omega F$ on $A^{\hat\otimes \Omega}$ and $E^{\hat\otimes \Omega}$ are canonical and well-defined for any graded $G$-$C^*$-algebra $A$ and for any graded $G$-equivariant Hilbert $A$-module $E$. This makes $E^{\hat\otimes \Omega}=\hat\bigotimes_\Omega E$ a $G\wr_\Omega F$-equivariant graded Hilbert $A^{\hat\otimes \Omega}$-module. 

It is sometimes convenient to fix a total order on $\Omega$ and work with the standard expressions for the graded tensor product. We will occasionally do this implicitly with care. For instance, if $\Omega = C_2$ with nontrivial action by the cyclic group $C_2 = \{1, \tau\}$ with order $1<\tau$ (or the other), then for homogeneous elements $a, b \in A$, we have:
\[
\tau(a \hat\otimes b) = (-1)^{|a||b|} (b \hat\otimes a),
\]
where we identify $A^{\hat\otimes \Omega} \cong A \hat\otimes A$ (using the order). 

We caution the reader that unless a total order on $\Omega$ is fixed, expressions such as $\hat\otimes_{\omega \in \Omega} a_\omega$ are not well-defined as elements of $\hat\bigotimes_{\omega \in \Omega} A_\omega$, due to the nontrivial braiding and resulting sign ambiguity (if more than one  $a_\omega$ has degree $1$).

\begin{theorem}\label{thm_KK_wr_functor} For any second countable, locally compact groups $G, F$, for any finite $F$-set $\Omega$, and for any separable graded $G$-$C^*$-algebras $A, B$, there is a map
\[
\hat\otimes_\Omega \colon KK_G(A, B) \to KK_{G\wr_\Omega F}(A^{\hat\otimes \Omega}, B^{\hat\otimes \Omega})
\]
 where $A^{\hat\otimes \Omega}$, $B^{\hat\otimes \Omega}$ are equipped with the canonical actions of $G\wr_\Omega F$, satisfying the following:
\begin{enumerate}
\item $\hat\otimes_\Omega$ sends the class of a $G$-equivariant, graded $\ast$-homomorphism $\phi\colon A \to B$ to the class of the $G\wr_\Omega F$-equivariant, graded $\ast$-homomorphism $\phi^{\hat\otimes \Omega}\colon A^{\hat\otimes \Omega} \to B^{\hat\otimes \Omega}$;
\item for any $G$-equivariant, graded $\ast$-homomorphisms $\phi\colon A\to B$, $\psi\colon C \to D$, and for any $x\in KK_G(B, C)$, we have $\hat\otimes_\Omega(\phi \otimes_B x  \otimes_C \psi ) = (\hat\otimes_\Omega(\phi)) \otimes_{B^{\hat\otimes \Omega}}  (\hat\otimes_\Omega(x))   \otimes_{C^{\hat\otimes \Omega}} (\hat\otimes_\Omega(\psi))$ in $KK_{G\wr_\Omega F}(A^{\hat\otimes\Omega}, D^{\hat\otimes \Omega})$;
\item for any $G$-equivariant, graded $\ast$-homomorphism $\phi\colon A \to B$ that is a $KK_G$-equivalence, $\hat\otimes_\Omega (\phi)$ is a $KK_{G\wr_\Omega F}$-equivalence.
\end{enumerate}
The restriction of $\hat\otimes_\Omega$ to the subcategory of the separable ungraded $C^*$-algebras is a functor from $KK_G$ to $KK_{G\wr_ \Omega F}$, and it is a unique functor satisfying the first condition (1) above.
\end{theorem}
\begin{proof} We give a proof using Kasparov cycles since we need to generalise the construction later to the group $KK_{G, \ell}(A, B)$, in which case both the categorical/homotopical approach (c.f. \cite[Section 2.2]{Bunke2023}) and the asymptotic homomorphisms approach (c.f. \cite[Lemma 2.4]{CEKN}) are subtle. We also note that, if one prefers, $F$ can be assumed to be a finite group without loss of generality for the obvious reason. 

First, we construct a candidate map
\[
F_{A, B} \colon KK_G(A, B) \to KK_{G\wr_\Omega F}(A^{\hat\otimes \Omega}, B^{\hat\otimes \Omega})
\]
for any separable graded $G$-$C^*$-algebras $A, B$. Let $(E, \pi, T)$ be a Kasparov cycle for $KK_G(A, B)$. Consider the following triple
\[
(E^{\hat\otimes \Omega}, \pi\wr_\Omega F, T') 
\]
where $\pi\wr_\Omega F$ is the canonical action of $G\wr_\Omega F$ on the graded tensor product $E^{\hat\otimes \Omega}$, making it a $G \wr_\Omega F$-equivariant graded Hilbert $A^{\hat \otimes \Omega}$-$B^{\hat \otimes \Omega}$-module, $T' \in \Linears(E^{\hat\otimes \Omega})$ is of the form
\[
T'= \sum_{\omega \in \Omega} N_\omega^{1/2} T_\omega,
\]
where $T_\omega\coloneq (1\hat\otimes  \ldots \hat\otimes1\hat \otimes T \hat\otimes 1 \hat\otimes \ldots \hat\otimes1 )$ with $T$ in the $\omega$-th factor, and the family $(N_\omega)_{\omega\in \Omega}$ of operators in $\Linears(E^{\hat\otimes \Omega})$ satisfies the following seven conditions:
\begin{enumerate}
\item $N_\omega$ are $\prod_\Omega G$-continuous, even degree, positive elements in $\mathcal{L}(E^{\hat\otimes \Omega})$; 
\item $1=\sum_{w\in \Omega} N_\omega$;
\item $[N_\omega, N_{\omega'}] \in \Compacts(E^{\hat\otimes \Omega})$ for any $\omega, \omega'$ in $\Omega$;\item $g(N_\omega)-N_\omega \in \Compacts(E^{\hat\otimes \Omega})$ for any $g\in \prod_{\Omega}G$;
\item $N_\omega \Compacts(E)_\omega \subset \Compacts(E^{\hat\otimes \Omega})$ where $\Compacts(E)_\omega=\bC \hat\otimes \ldots \hat\otimes  \Compacts(E) \hat\otimes \bC \hat\otimes\ldots \hat\otimes \bC$ with $\Compacts(E)$ in the $\omega$-th factor;
\item $[N_\omega, \Delta] \subset \Compacts(E^{\hat\otimes \Omega})$ for a separable set $\Delta$ containing $A^{\hat\otimes \Omega}$, $\{T_{\omega'}, \omega' \in \Omega\}$ and $\Compacts(E)_{\omega'}$ for all $\omega'\in \Omega$;
\item $w(N_{w^{-1}\omega})=N_{\omega}$ for any $w\in F$, $\omega\in \Omega$.
\end{enumerate}

The existence of such a family $(N_\omega)_{\omega\in \Omega }$ is proven in Theorem \ref{thm_tech_thm} below.

Then, it is easy to see that the triple $(E^{\hat\otimes\Omega}, \pi\wr_\Omega F, T')$ is a Kasparov cycle for $KK_{G\wr_\Omega F}(A^{\hat\otimes \Omega}, B^{\hat\otimes \Omega})$. The homotopy class $[E^{\hat\otimes\Omega}, \pi\wr_\Omega F, T']$ does not depend on the choice of $(N_\omega)_{\omega \in \Omega}$ satisfying the conditions above. To see this, given two families $(N_\omega)_{\omega \in \Omega}$, $(N'_\omega)_{\omega \in \Omega}$ of operators satisfying the condition above, by Theorem \ref{thm_tech_thm}, we can take another family $(N''_\omega)_{\omega \in \Omega}$ satisfying, in addition to the conditions above, that $[N''_\omega, N_{\omega'}], [N''_\omega, N'_{\omega'}] \in \Compacts(E^{\hat\otimes \Omega})$ for all $\omega, \omega' \in \Omega$. Then, the straight-line homotopies from $N_\omega$ to $N''_\omega$, and from $N''_\omega$ to $N'_\omega$ provide a homotopy  $(N^t_\omega)_{\omega \in\Omega, t\in [0,1]}$ connecting $(N_\omega)_\omega$ to $(N'_\omega)_\omega$ satisfying the conditions above. This shows the independence of $[E^{\hat\otimes\Omega}, \pi\wr_\Omega F, T']$ on the choice of $(N_\omega)_{\omega \in \Omega}$.

The construction applied to the cycles for $KK_G(A, B[0,1])$ shows that the map 
\[
(E, \pi, T) \mapsto  [E^{\hat\otimes \Omega}, \pi \wr_\Omega F, T'] 
\]
descends to 
\[
F_{A, B}\colon KK_G(A, B) \to KK_{G\wr_\Omega F}(A^{\hat\otimes \Omega}, B^{\hat\otimes \Omega}).
\]
From the construction, $F_{A, B}$ maps the class $[\phi]$ of a $G$-equivariant graded $\ast$-homomorphism $\phi\colon A \to \mathcal{K}(E)$ to the class of the $G\wr_\Omega F$-equivariant graded $\ast$-homomorphism $\phi^{\hat\otimes \Omega}\colon A^{\hat\otimes \Omega} \to \mathcal{K}(E)^{\hat\otimes \Omega} \cong \mathcal{K}(E^{\hat\otimes \Omega})$. In particular, it sends the unit in $KK_G(A, A)$ to the unit in $KK_{G\wr_\Omega F}(A^{\hat\otimes \Omega}, A^{\hat\otimes \Omega})$. 

From the construction, we see that $F_{A, C}$ maps the Kasparov product $x\otimes_B y$ to $F_{A, B}(x)\otimes_{B^{\hat\otimes \Omega}} F_{B, C}(y)$ provided either $x \in KK_G(A, B)$ is a $G$-equivariant graded $\ast$-homomorphism from $A$ to $B$, or $y \in KK_G(B, C)$ is a surjective $G$-equivariant graded $\ast$-homomorphism from $B$ to $C$. The first case is easy to see. To see the case when $y$ is a surjective $\ast$-homomorphism $\phi\colon B\to C$, we note that in this case, the canonical map $\Compacts(E)\hat\otimes \bC \to \Compacts(E\hat\otimes_B C)$ is surjective  for any Hilbert $B$-module $E$, since we have $\theta_{e_1\hat\otimes \phi(b_1), e_2\hat\otimes \phi(b_2)} =\theta_{e_1b_1, e_2b_2}\hat\otimes 1$ on $E\hat\otimes_BC$ for any $e_1, e_2 \in E$ and $b_1, b_2\in B$. For any element $x=[E, \pi, T]$ in $KK_G(A, B)$, $F_{A, B}(x)\otimes_{B^{\hat\otimes \Omega}} \phi^{\hat \otimes \Omega}$ is represented by a cycle of the form
\[
((E\hat\otimes_B C)^{\hat\otimes \Omega}, \pi \wr_\Omega F, T')
\]
where $T'=\sum_{\omega \in \Omega}N_\omega^{1/2} (T\hat\otimes 1)_\omega$ and operators $N_\omega \in \Linears((E\hat\otimes_B C)^{\hat\otimes \Omega})$ satisfy the seven conditions as above with $E\hat\otimes_B C$ in place of $E$. Here, the point is that condition (6) of $N_\omega$ is satisfied thanks to the surjectivity of $\Compacts(E)\hat\otimes \bC \to \Compacts(E\hat\otimes_B C)$ that we mentioned earlier (the surjectivity of $\phi\colon B\to C$ is not needed for any of the other conditions for $N_\omega$ to hold). It follows that $F_{A, B}(x)\otimes_{B^{\hat\otimes \Omega}} \phi^{\hat \otimes \Omega}= F_{A, C}(x \otimes_{B} \phi)$.

For the functoriality with respect to a not necessarily surjective $\ast$-homomorphism $\phi \colon B\to C$, we take the following roundabout approach. Any map $\phi\colon B\to C$ is the composition of the canonical split-inclusion $\iota_B\colon B\to M_\phi$ to the mapping cylinder $M_\phi$ of $\phi$ and a surjective map $\psi\colon M_\phi\to C$. The map $\iota_B\colon B\to M_\phi$ is a homotopy equivalence with a surjective inverse $\pi_B$. Combining this observation with the previous part, we have for any $x\in KK_G(A, B)$, $F_{A, B}(x\otimes_B \phi) = F_{A, B}(x)\otimes_{B^{\hat \otimes \Omega}} F_{B, C}(\phi)$ for any $\ast$-homomorphism $\phi \colon B\to C$. 

 It follows that if $\phi\colon A\to B$ is a graded $G$-$\ast$-homomorphism which is a $KK_G$-equivalence with the inverse $x$, then $F_{A, B}(\phi)$ is a $KK_{G\wr _\Omega F}$-equivalence with inverse $F_{B, A}(x)$. 
 
 For the last part of the statement, recall from \cite[Theorem 6.5]{Meyer2000} that any map in $KK_G(A, B)$ is a composition of $\ast$-homomorphisms and the inverses of $\ast$-homomorphisms in $KK_G$, if $A, B$ are ungraded. By the previous part, it follows $F_{A, C}(x\otimes_B y )= F_{A, B}(x)\otimes_{B^{\hat\otimes \Omega}} F_{B, C}(y)$ for any $x\in KK_G(A, B), y\in KK_G(B, C)$ if $A, B, C$ are ungraded. Therefore $F_{A, B}$ defines a functor $\hat\otimes_\Omega$ from $KK_G$ to $KK_{G\wr_\Omega F}$ for the ungraded $C^*$-algebras. The uniqueness follows again by \cite[Theorem 6.5]{Meyer2000}.
\end{proof}

Here is some variant of the Kasparov technical theorem. The latter  assertion (the symmetric case) is what is used in the proof above.

\begin{theorem}[c.f. {\cite[Theorem 1.4]{Kas88}}]\label{thm_tech_thm} For $1\leq i\leq N$, let $G_i$ be a second countable, locally compact group, let $A_i, B_i$ be separable graded $G_i$-$C^*$-algebras, and let $(E_i, \pi_i, T_i)$ be a Kasparov cycle for $KK_{G_i}(A_i, B_i)$. Let $E=\hat\bigotimes_{1\leq i \leq N}E_i$ be the exterior tensor product (over $\bC$), which is a $G$-equivariant graded $\hat\bigotimes_{1\leq i \leq N}A_i$-$\hat\bigotimes_{1\leq i \leq N}B_i$-module with the $G$-action $\hat\bigotimes_{1\leq i \leq N} \pi_i$ where $G=\prod_{1\leq i\leq N}G_i$. We regard $T_i\in \Linears(E)$ via $\Linears(E_i) \subset \Linears(E)$. Let $\Delta$ be any separable subspace of $\Linears(E)$ consisting of elements $x\in \Linears(E)$ such that $[x, \Compacts(E_i)]\subset \Compacts(E_i) + \Compacts(E)$ for all $i$ where we regard $\Compacts(E_i)\subset \Linears(E)$ (in practice, we include $\Compacts(E_j)$, $A_j$ and $T_j$ in $\Delta$ for all $j$).

Then, there is a family $(N_i)_{1\leq i\leq N}$ of elements in $\Linears(E)$ satisfying the following conditions:
\begin{enumerate}
\item $N_i$ are $G$-continuous, even degree, positive elements in $\mathcal{L}(E)$; 
\item $1 =\sum_{1\leq i \leq N}N_i$;
\item $[N_i, N_{j}] \in \Compacts(E)$ for any $1\leq i, j\leq N$;
\item $g(N_i)-N_i \in \Compacts(E)$ for any $g\in G$;
\item $N_i \Compacts(E_i) \subset \Compacts(E)$;
\item $[N_i, \Delta] \subset \Compacts(E)$.
\end{enumerate}

If in addition, $G_i$, $A_i$, $B_i$, $E_i$, $\pi_i$, $T_i$ are independent of $i$, then $(N_i)_{1\leq i\leq N}$ can be taken to be $S_N$-equivariant, for the symmetric group $S_N$, in the sense that 
\begin{enumerate}
\item[(7)] $w(N_{w^{-1}i})=N_{i}$ for any $w\in S_N$ and for any $1\leq i \leq N$,
\end{enumerate}
for the canonical $S_N$-action on $E$.
\end{theorem}
\begin{proof} 

For each subset $J\subset \{1, \ldots, N\}$, we define elements $M_J\in \Linears(E)$ of the form
\begin{equation}\label{eq_MJ}
M_J \coloneq \sum_{k\geq 0  }d_k\left( \prod_{j\in J} (1-u_k^{(j)})  \prod_{i \in \{1, \ldots, N\} \smallsetminus J} u_k^{(i)}  \right) d_k,
\end{equation}
\begin{itemize}
\item $d_k=(w_{k+1}-w_{k})^{1/2}$ for a suitable increasing sequence $(w_k)_{k\geq 0}$ of approximately $G$-equivariant, quasi-central approximate unit in $\Compacts(E)$ of even degree, with $w_0=0$;
\item $(u^{(i)}_k)_{k\geq0}$ is a suitable increasing sequence of approximately $G$-equivariant, quasi-central approximate unit in $\Compacts(E_i)$ of even degree.
\end{itemize}
Note that each summand is positive since $[u_k^{(i)}, u_k^{(j)}]=0$. The sum that defines $M_J$ is in the strict topology, and the convergence is guaranteed by taking a suitable $(w_k)_k$ so that $\|d_k h \| \leq 2^{-k}$ for a fixed, contractive, strictly positive element $h$ in $\Compacts(E)$ (see \cite{Hig87} for example).

The key condition (aside from the standard ones) we arrange for $(M_J)_J$ to satisfy is $[M_J, M_{J'}]\in \Compacts(E)$ for any $J, J'$. The desired positive elements $N_i$ are defined by
\begin{equation}\label{eq_Ni}
N_i \coloneq  \frac{1}{N}M_{\emptyset} + \sum_{i \in J\subset \{1, \ldots, N\}}\frac{1}{|J|}M_J.
\end{equation}
It is clear that we have
\[
1 = \sum_{J\subset\{1, \ldots, N \}}M_J =\sum_{1\leq i \leq N}N_i.
\]

We define $u^{(i)}_k, w_k$ inductively in the following order: $u^{(i)}_k$, $w_k$, $u^{(i)}_{k+1}$, $w_{k+1}, \ldots $, starting with $u^{(i)}_0=w_0=0$. We fix contractive, strictly positive elements $h_i\in \Compacts(E_i)$ and $h \in \Compacts(E)$. Write $G$ as the increasing union $\bigcup_{n\geq 1} Y_n$ of relatively compact open subsets $Y_n$. Write (the norm closure of) $\Delta$ as the closed span of a compact subset $Z$ contained in the unit ball. Let $(\delta_n)_{n\geq 1}$ be a fixed, decreasing sequence of strictly positive universal constants that are  explained later below. Having constructed $u^{(i)}_k$, $w_k$ for $k<n$, we construct $u^{(i)}_n\geq u^{(i)}_{n-1}$, and then $w_n\geq w_{n-1}$, so that the following conditions hold: 
 \begin{enumerate}
\item[(a1)] $\|[u^{(i)}_n, w_m]\| <  \delta_n$ for $m<n$;
\item[(a2)] $\|[u^{(i)}_n, u^{(i)}_m]\| < 2^{-2n}$ for $m\leq n$;
\item[(a3)] $\|(1-u^{(i)}_n)h_i\| \leq 2^{-n}$;
\item[(a4)] $\|{g(u^{(i)}_n)-u^{(i)}_n}\|< 2^{-n}$ for any $g \in Y_n \subset G$;
\item[(a5)] $\|[u^{(i)}_n, x]\| < 2^{-n}$ for $x\in Z$;
\item[(a6)] $\|[w_n, u_m^{(i)}]\| <  \delta_n$ for $m\leq n$;
\item[(a7)] $\|[w_n, w_m]\| <  \delta_{n}$ for $m\leq n$;
\item[(a8)] $\|(1-w_n) u_m^{(1)}u_m^{(2)}\ldots u^{(N)}_m  \|< \delta_n$ for $m\in \{n, n-1\}$; 
\item[(a9)] $\|{g(w_n)-w_n}\|< \delta_n$ for any $g \in Y_n \subset G$;
\item[(a10)] $\|[w_n, x]\| < \delta_n$ for $x\in Z\cup \{h_j\mid 1\leq j \leq N\}$;  
\item[(a11)] $\|(1-w_n)h\| < \delta_n$.
\end{enumerate}

By the proof of \cite[Lemma 1.4]{Kas88} (see also \cite[Section 1]{Arveson}, \cite{Hig87} and \cite[Section 12.4]{Bla98}), for any approximate units (increasing, contractive, even degree) $\underline{u_n^{(i)}}$ of $\Compacts(E_i)$ and $\underline{w_n}$ of $\Compacts(E)$, we can inductively obtain such $u_n^{(i)}$ in the convex hull $\mathrm{conv}\{\underline{u_m^{(i)}}\mid m\geq n\}$ and $w_n \in  \mathrm{conv}\{\underline{w_m} \mid m\geq n\}$. Note that any approximate unit of $\Compacts(E_i)$ is an approximate unit of $\Compacts(E_i)+ \Compacts(E)$. 

The constants $(\delta_n)_n$ are chosen, using the approximation of $x^{1/2}$ by polynomials in $x$ in $C[0, 1]$ (see \cite{Hig87} for example) so that, with these conditions satisfied, we have for $d_n=(w_{n+1}-w_n)^{1/2}$,
 \begin{enumerate}
\item[(b1)] $\|[u^{(i)}_n, d_m]\| <  2^{-2n}$ for $m< n$;
\item[(b2)] $\|[d_n, u_m^{(i)}]\| <  2^{-2n}$ for $m\leq n$;
\item[(b3)] $\|[d_n, d_m]\| <  2^{-2n}$ for $m\leq n$;
\item[(b4)] $\|{g(d_n)-d_n}\|< 2^{-n}$ for any $g \in Y_n \subset G$;
\item[(b5)] $\|[d_n, x]\| < 2^{-n}$ for $x\in Z\cup \{h_j\mid 1\leq j \leq N\}$;
\item[(b6)] $\|d_n h\| \leq 2^{-n}$.
\item[(b7)] $\|d_n u_n^{(1)}u_n^{(2)}\ldots u^{(N)}_n  \|\leq 2^{-n}$; 
\end{enumerate}
To see that item (b1) for $m=n-1$ holds, use $\|[u^{(i)}_n, w_{n-1}]\| <\delta_n$ from item (a1) and  $\| [u^{(i)}_n, w_{n} ]\| <\delta_n$ from item (a6). Combining items (b1), (b2), we have 
 \begin{enumerate}
\item[(b8)] $\|[u^{(i)}_n, d_m]\| <  2^{-2\max\{n, m\}}$ for all $m, n$.
\end{enumerate}

Define $M_J$ as in \eqref{eq_MJ}. Then these sums are well-defined, strictly convergent sums whose partial sums all have norm  bounded by one, and the following conditions hold:
\begin{enumerate}
\item[(c1)] $M_J$ are $G$-continuous, even degree, positive elements in $\mathcal{L}(E)$; 
\item[(c2)] $g(M_J)-M_J \in \Compacts(E)$ for any $g\in G$;
\item[(c3)] $M_J\Compacts(E_i) \subset \Compacts(E)$ if $i\in J$;
\item[(c4)] $M_\emptyset \in \Compacts(E)$;
\item[(c5)] $[M_J, x] \in \Compacts(E)$ for $x \in \Delta$.
\end{enumerate}
We leave this routine verification to the reader. We now check the critical condition $[M_J, M_{J'}]\in \Compacts(E)$. Let us simply write $M_J=\sum_{k\geq0}d_k U^{(J)}_k d_k$ where $U^{(J)}_k=\left( \prod_{j\in J} (1-u_k^{(j)})  \prod_{i \in \{1, \ldots, N\} \smallsetminus J} u_k^{(i)}  \right)$.

First of all, we have
\[
M_J = \sum_{k\geq0}( [d_k, U^{(J)}_k]d_k+ U^{(J)}_kd_k^2 )
\]
where the sum is strictly convergent with uniformly bounded partial sums. Since $\sum_{k\geq0}( [d_k, U^{(J)}_k] d_k)$ is absolutely convergent to a compact operator in $\Compacts(E)$ (by item (b8)), we have
\[
M_J - \sum_{k\geq0}(U_k^{(J)}d_k^2) \in \Compacts(E)
\]
where the sum is strictly convergent with uniformly bounded partial sums. Thus, to show $[M_J, M_{J'}] \in \Compacts(E)$, it suffices to show 
\[
\sum_{k\geq0}(U_k^{(J)}d_k^2 )  \sum_{l\geq0}(U_l^{(J')}d_l^2 ) -  \sum_{l\geq0}(U_l^{(J')}d_l^2 )  \sum_{k\geq0}(U_k^{(J)}d_k^2 )  \in \Compacts(E).
\]
Since each sum is strictly convergent with uniformly bounded partial sums, the left hand side is equal to 
\[
\lim_{T\to \infty}\sum_{0\leq k, l\leq T}(U_k^{(J)}d_k^2 )(U_l^{(J')}d_l^2 ) - (U_l^{(J')}d_l^2 )(U_k^{(J)}d_k^2 ) 
\]
where the limit is in the strict topology, with uniformly bounded terms. We have
\[
\sum_{0\leq k, l\leq T}(U_k^{(J)}d_k^2 )(U_l^{(J')}d_l^2 ) - (U_l^{(J')}d_l^2 )(U_k^{(J)}d_k^2 )  = \sum_{0\leq k, l \leq T}(A_{k,l} +  B_{k,l} + C_{k, l} + D_{k, l})
\]
where
\[
A_{k, l}= U_k^{(J)}[d_k^2, U_l^{(J')}]d_l^2; 
\]
\[
B_{k, l}= U_k^{(J)}U_l^{(J')}[d_k^2, d_l^2];  
\]
\[
C_{k, l}= [U_k^{(J)},U_l^{(J')}]d_l^2d_k^2;
\]
\[
D_{k, l}= U_l^{(J')}[U_k^{(J)}, d_l^2]d_k^2. 
\]
By (a2), (b3), (b8), we have 
\[
\max\{\|A_{k, l}\|, \|B_{k, l}\|, \|C_{k, l}\|, \|D_{k, l}\|\} \leq  4N\cdot 2^{-k-l}
\]
for all $k, l$. It follows that
\[
\sum_{k, l\geq0} (A_{k,l} +  B_{k,l} + C_{k, l} + D_{k, l})
\]
is absolutely convergent to a compact operator in $\Compacts(E)$. We have shown that $[M_J, M_{J'}] \in \Compacts(E)$. 

Now, define $N_i$ as in \eqref{eq_Ni}. It is easy to check that the all the conditions from (1) to (6) in the statement are satisfied for $(N_i)_{1\leq i\leq N}$.

We now consider the symmetric setting of the latter assertion. In this case, we can take $w_k$ to be $S_N$-equivariant, and $u_k^{(i)}$ to be $S_N$-equivariant in the sense that $w(u^{(w^{-1}i)}_k)=u_k^{(i)}$ for $w\in S_N$. It is not hard to see that the required conditions (a1)-(a11) can be arranged with these extra conditions on $w_k$ and $u_k^{(i)}$. Then, the corresponding elements $N_i$ satisfy $w(N_{w^{-1}i})=N_i$ for $w \in S_N$ as desired.
\end{proof}

\begin{remark} A naive approach to prove the latter assertion (the symmetric case) of Theorem \ref{thm_tech_thm} would be to take a standard partition of unity in the formula of the external Kasparov products, and take their $S_N$-averages. We note that, without any extra adjustment, this will not work unless $N=2$. For example, for $N=3$, suppose we take the $S_N$-average of the standard partition $(N_i')_{1\leq i\leq 3}$ of unity satisfying $N'_1+N'_2+N'_3=1$. Even if $N'_i$ are essentially commuting to each other, in order for their $S_3$-averages to essentially commute to each other, we want a condition like $[N'_1, w(N'_1)]\in \Compacts(E)$ for $w\in S_3$, which is not automatic from the standard construction of $N'_1$.
\end{remark}

\begin{definition}\label{def_wreath_func} For any second countable, locally compact groups $G, F$, for any finite $F$-set $\Omega$, and for any separable graded $G$-$C^*$-algebras $A, B$, we define the equivariant power map
\[
\hat\otimes_\Omega\colon KK_G(A, B) \to KK_{G\wr_\Omega F}(A^{\hat\otimes \Omega}, B^{\hat \otimes \Omega})
\]
as the map constructed in the proof of Theorem \ref{thm_KK_wr_functor}. For the separable ungraded $G$-$C^*$-algebras, it is a unique functor from $KK_G$ to $KK_{G\wr_\Omega F}$ sending $[\phi]$ to $[\phi^{\otimes \Omega}]$ for any $G$-$\ast$-homomorphism $\phi\colon A\to B$, as shown in Theorem \ref{thm_KK_wr_functor}.
\end{definition}

\begin{remark} It is expected but the authors do not know whether $\hat\otimes_\Omega$ is a functor for the graded $G$-$C^*$-algebras as well. We leave this to the interested reader, as it will not be needed in this article. By Theorem \ref{thm_KK_wr_functor}, this would follow if we could show that every morphism in $KK_G(A, B)$ is a composition of $\ast$-homomorphisms and their inverses in $KK_G$ for any separable graded $G$-$C^*$-algebras $A, B$. By \cite[Theorem 7.4]{Meyer2000}, this boils down to showing this for the single element $e\in KK(\bC, \hat {S})$, but the authors do not know if this is true. 
\end{remark}

We remind our convention that the category $KK_G$ consists of the separable \emph{ungraded} $G$-$C^*$-algebras. 

\begin{lemma}  For any second countable, locally compact groups $G, F, L$, for any finite $F$-set $\Omega$, and for any continuous homomorphism $L\to G\wr_\Omega F$, the composition 
\[
\mathrm{Res}_{G\wr _\Omega F}^{L}\circ (\hat\otimes_\Omega) \colon KK_G\to KK_{G\wr_\Omega F} \to KK_{L}
\]
is a unique functor from $KK_G$ to $KK_L$ sending any separable ungraded $G$-$C^*$ algebra $A$ to $A^{\otimes \Omega}$ (equipped with the $L$-action restricted from the $G\wr_\Omega F$-action) and the class of a $G$-equivariant $\ast$-homomorphism $\phi\colon A\to B$ to the class of the $L$-equivariant $\ast$-homomorphism $\phi^{\otimes \Omega}\colon A^{\otimes \Omega }\to B^{\otimes \Omega}$.
\end{lemma}
\begin{proof} The uniqueness follows because any map in $KK_G(A, B)$ is the composition of $\ast$-homomorphisms and the inverses of $\ast$-homomorphisms in $KK_G$, if $A, B$ are ungraded (\cite[Theorem 6.5]{Meyer2000}).
\end{proof}


\begin{theorem}\label{thm_gamma_wr_KK} Let $G$ be a second countable, locally compact group. If $G$ has a gamma element $\gamma_G$ in $KK_G(\bC, \bC)$ in the Meyer--Nest sense, then for any second countable, compact group $F$ and for any finite $F$-set $\Omega$, $G\wr_\Omega F$ has a gamma element $\gamma_{G\wr_\Omega F}$ in $KK_{G\wr_\Omega F}(\bC, \bC)$ in the Meyer--Nest sense, and $\gamma_{G\wr_\Omega F}=\hat\otimes_\Omega(\gamma_G)$. If $\gamma_G=1_G$ in $KK_G(\bC, \bC)$, then $\gamma_{G\wr_\Omega F}=1_{G\wr_\Omega F}$ in $KK_{G\wr_\Omega F}(\bC, \bC)$.
In particular, the BCC with finite wreath products holds for $G$ if $\gamma_G=1_G$ in $KK_G(\bC, \bC)$.
\end{theorem}
The proof is given after some preparations. 

Let $G, F$ be any locally compact groups and $\Omega$ be a finite $F$-set. We recall from \cite[Definition 3.10]{KN25}, the \emph{Izumi filtration} of $A^{\hat\otimes \Omega}$ associated with any short exact sequence 
\[
0 \to I \to A\to A/I \to 0
\]
of graded $G$-$C^*$-algebras. This is the filtration of $A^{\hat\otimes \Omega}$ by $G\wr_\Omega F$-invariant graded ideals:
\[
0 = I_{|\Omega|+1} \subset  I_{|\Omega|} \subset  \ldots \subset  I_{j+1} \subset I_j \subset  \ldots \subset I_{0} = A^{\hat\otimes \Omega},
\]
where $I_j$ is the closure of $\sum_{S\subset \Omega, |S|=j}I^{\hat\otimes S}\hat\otimes A^{\hat\otimes \Omega \smallsetminus S}$ in $A^{\hat\otimes \Omega}$. In particular, we have $ I_{|\Omega|}= I^{\hat\otimes \Omega}$. By convention, $I_{|\Omega|+1} \coloneq0$. By the proof of \cite[Proposition 3.9]{KN25}, assuming that the quotient map $A\to A/I$ admits a graded (not necessarily $G$-equivariant, not necessarily contractive) completely positive splitting (c.p.), we have a canonical $G\wr_\Omega F$-equivariant isomorphism 
\[
I_{j}/I_{j+1} \cong \bigoplus_{S\subset \Omega, |S|=j} I^{\hat\otimes S} \hat\otimes (A/I)^{\hat\otimes \Omega \smallsetminus S},
\]
where the right-hand side is endowed with the canonically induced $G\wr_\Omega F$-action. Moreover, the short exact sequence
\[
0 \to I_{j+1} \to I_j \to I_{j}/I_{j+1} \to 0
\]
admits a graded (not necessarily $G$-equivariant, not necessarily contractive) c.p. splitting. 

Let $A, B$ be any graded $G$-$C^*$-algebras. Let $\phi\colon A\to B$ be any (not necessarily surjective) graded $G$-equivariant $\ast$-homomorphism. Let $M_\phi$ be the mapping cylinder:
\[
M_\phi \coloneq \{ (a, b) \in A\oplus C([0,1], B) \mid \phi(a)=b(0)   \},
\]
and let $C_\phi$ be the mapping cone:
\[
C_\phi \coloneq \{ (a, b) \in A\oplus C([0,1], B) \mid \phi(a)=b(0), b(1)=0   \}.
\]
For the discussion below, $M_\phi$ is best viewed as the algebra of sections of an upper-semicontinuous field of $C^*$-algebras on $[0, 1]$ whose fibers are $A$ at $0$ and $B$ over $(0, 1]$. The  sections are generated (spanned) by $(a, \phi(a))$ for $a\in A$ and $(0, b)$ for $b \in C_0((0, 1], B)$, and $M_\phi$ is a $C[0,1]$-algebra. The ideal $C_\phi$ of $M_\phi$ consists of the sections vanishing at $1$.

Let $(I_j)_{0\leq j \leq |\Omega|+1}$ be the Izumi filtration of $M_\phi^{\hat\otimes \Omega}$ associated with the exact sequence
\[
0\to C_\phi \to M_\phi \xrightarrow{{\mathrm{ev}_{t=1}}}  B\to 0.
\] 
Since this exact sequence admits a graded ($G$-equivariant) contractive c.p. (c.c.p.) section $\psi\colon b\mapsto (0, tb)$, the sequence
\begin{equation}\label{eq_Izumi_ref}
0\to I_{j+1} \to I_j \to I_{j}/I_{j+1} \to 0
\end{equation}
admits a graded ($G\wr_\Omega F$-equivariant, not necessarily contractive) c.p. section, and we have
\begin{equation}\label{eq_Izumi_decomp}
I_{j}/I_{j+1} \cong \bigoplus_{S\subset \Omega, |S|=j} C_\phi^{\hat \otimes S} \hat\otimes B^{\hat \otimes \Omega \smallsetminus S}.
\end{equation}
As in the proof of \cite[Proposition 3.9]{KN25}, a $G\wr_\Omega F$-equivariant c.p. section $I_{j}/I_{j+1} \to I_j$ is given by $\bigoplus_{S\subset \Omega, |S|=j} \mathrm{id}_{C_\phi}^{\hat \otimes S} \hat \otimes \psi^{\hat\otimes \Omega \smallsetminus S}$. These hold more generally for the Izumi filtration of $A^{\hat \otimes \Omega}$ associated with any short exact sequence of graded $G$-$C^*$-algebras $(0\to I \to A \to A/I\to 0)$ which admits a graded $G$-equivariant c.p. section. 

A more relevant fact we prove below is that the sequence \eqref{eq_Izumi_ref} admits a $G\wr_\Omega F$-equivariant c.c.p. section, at least, in the setting of the mapping cylinder/cone. To make the proof of Proposition \ref{prop_admissible_cylinder} below more understandable, it is helpful to illustrate the following key example:

\begin{example} Suppose $A=B=\bC$, $\phi=\mathrm{id}_\bC$, $G$ is the trivial group, and $F$ is the cyclic group $C_2$ of order two acting on $\Omega=C_2$ non-trivially. Then we have $M_\phi\cong C[0,1]$ and $M_\phi^{\otimes\Omega}=C([0,1]\times [0,1])$. The ideals $I_1$ and $I_2$ of the Izumi filtration are the ideals of $C([0,1]\times [0,1])$ vanishing at $(1,1)$ and on $[0, 1]\times \{1\} \cup \{1\}\times [0, 1]$, respectively. The quotient map $I_1\to I_1/I_2$ is identified as the restriction of $C_0$-functions on $[0,1]\times [0,1] \smallsetminus \{(1, 1)\} $ to $[0, 1)\times \{1\} \sqcup \{1\}\times [0, 1)$. The $C_2$-equivariant c.p. section $I_1/I_2\to I_1$ we mentioned above, is identified as the extension of functions on $[0, 1]\times \{1\} \cup \{1\}\times [0, 1]$ (vanishing at $(1,1)$) to the square $[0,1]\times [0,1]$ of the following form: a function $f=(f_1, f_2) \in C_0([0,1)\times \{1\}) \oplus C_0(\{1\} \times [0,1))$ is extended to $\tilde{f} \in C([0,1]\times [0,1])$ by $\tilde{f}(t_1, t_2)\coloneq t_2f_1(t_1, 1) + t_1f_2(1, t_2) =  t_2f(t_1, 1) +  t_1f(1,t_2) $. This extension is not contractive (it has norm $2$) as a map  $I_1/I_2\to I_1$ although it is quite canonical. On the other hand, we can define a $C_2$-equivariant c.c.p. extension by extending a function $f$ on  $[0, 1]\times \{1\} \cup \{1\}\times [0, 1]$ to $\tilde{f}$ on $[0,1]\times [0,1]$ as the $1$-homogeneous extension with respect to the $\ell^\infty$-metric on the cube:
\[
\tilde{f}(t_1, t_2) \coloneq 
  \begin{cases}  
\max{\{t_1, t_2\}}  f(\frac{t_1}{\max{\{t_1, t_2\}}}, \frac{t_2}{\max{\{t_1, t_2\}}})   &    (t_1, t_2) \neq (0, 0)  \\
    0 &   (t_1, t_2) =(0, 0). 
    \end{cases} 
  \]
  \end{example}

The proof of the following proposition is all about extending this construction to the general mapping cylinders, and it is not as complicated as it might look:

\begin{proposition}\label{prop_admissible_cylinder} Let $G, F$ be any locally compact groups, and let $\Omega$ be any finite $F$-set. Let $\phi\colon A\to B$ be any graded $G$-equivariant $\ast$-homomorphism for graded $G$-$C^*$-algebras $A, B$.  Let $(I_j)_{0\leq j \leq |\Omega|+1}$ be the Izumi filtration of $M_\phi^{\hat\otimes \Omega}$ associated with the mapping cone/cylinder exact sequence $(0\to C_\phi \to M_\phi \to B\to 0)$. For any $0\leq j \leq |\Omega| -1$, the short exact sequence
\[
0 \to I_{j+1} \to I_j \to I_j/I_{j+1}\to 0
 \]
 admits a graded, $G\wr_\Omega F$-equivariant, contractive completely positive (c.c.p.) section.
\end{proposition}
  \begin{proof}
 Note that $M_\phi^{\hat\otimes \Omega}$ is canonically a graded $G\wr_\Omega F$-subalgebra of $(A\oplus C([0,1], B))^{\hat\otimes\Omega}$. Using the canonical isomorphism (the binomial expansion)
 \[
 (A\oplus C([0,1], B))^{\hat\otimes\Omega} \cong \bigoplus_{S\subset \Omega} C([0, 1]^{S}\times \{0\}^{\Omega \smallsetminus S}, B^{\hat \otimes  S}\hat \otimes A^{\hat\otimes \Omega \smallsetminus S}),
 \]
 we observe that $M_\phi^{\hat \otimes \Omega}$ is canonically isomorphic to the following graded $C^*$-algebra
  \[
\mathcal{B}_\Omega \coloneq \{ (b_S)_{S\subset \Omega} \in  \bigoplus_{S\subset \Omega} C([0, 1]^{S}\times \{0\}^{\Omega \smallsetminus S}, B^{\hat \otimes  S}\hat \otimes A^{\hat\otimes \Omega \smallsetminus S}) \mid \text{ $(b_S)_{S\subset \Omega}$ is compatible}  \}
  \]
  equipped with the canonical $G\wr_\Omega F$-action. Here, we say $(b_S)_{S\subset \Omega}$ is compatible if for any $S\subset S'$, we have
  \[
  b_{S'}\mid_{[0,1]^S\times \{0\}^{\Omega \smallsetminus S} } = ( \mathrm{id_B}^{\hat\otimes S} \hat \otimes \phi^{\hat \otimes S'\smallsetminus S} \hat \otimes \mathrm{id_A}^{\hat\otimes \Omega \smallsetminus S'}) (b_S). 
  \]
  For example, for $\emptyset \subset \Omega$, this says $(b_\Omega)(0)= \phi^{\hat \otimes \Omega}(b_\emptyset) \in B^{\hat\otimes \Omega}$. We can see the isomorphism $M_\phi^{\hat\otimes\Omega}\cong \mathcal{B}_\Omega$ by iteratively using the fact that for any $C^*$-algebra $C$, $C\hat \otimes M_\phi$ is isomorphic to the subalgebra of $(C\hat \otimes A) \oplus  C([0,1], C\hat \otimes B)$ consisting of $(x, y)$ satisfying $y(0)=(\mathrm{id}_C \hat\otimes \phi)(x)$. The latter is just saying $C\hat \otimes M_\phi$ is canonically the mapping cylinder of $\mathrm{id}_C\hat\otimes \phi$.
  
   The algebra $M_\phi^{\hat \otimes \Omega}$ is best viewed as the algebra of sections of an upper semi-continuous field $\mathcal{A}$ of $C^*$-algebras over the cube $[0, 1]^\Omega$ whose fibers $\mathcal{A}_x$ are $B^{\hat \otimes S}\hat \otimes A^{\hat\otimes \Omega \smallsetminus S}$ on the half-open cube $(0, 1]^S\times \{0\}^{\Omega \smallsetminus S}$. Note that, we have 
  \[
  [0, 1]^\Omega = \bigsqcup_{S\subset \Omega} (0, 1]^S\times \{0\}^{\Omega \smallsetminus S}.
  \]
For convenience, we introduce some terminologies. Let \( Y \subset [0,1]^{\Omega} \) be a closed subset. A \emph{section defined on \(Y\)} is an (arbitrary) section of the field $\mathcal{A}$ (restricted on $Y$). A section \( f \) defined on \( Y \) is called \emph{legitimate} if for every \( S \subset \Omega \), the associated function
$
    f_S : Y \cap [0,1]^S \times \{0\}^{\Omega \smallsetminus S} \ \longrightarrow\ 
    B^{\hat\otimes S} \ \hat\otimes\  A^{\hat\otimes (\Omega \smallsetminus S)}
   $
    obtained by restricting \(f\) and applying \(\phi\) to the relevant factors pointwise is continuous. By using $M_\phi^{\hat\otimes\Omega}\cong \mathcal{B}_\Omega$, we see that the $M_\phi^{\hat\otimes\Omega}$ is canonically isomorphic to the algebra of legitimate sections defined on $[0,1]^\Omega$.

 For any $S\subset \Omega$, we define $|S|$-dimensional face $F_S^1 \coloneq [0 ,1]^S \times \{ 1 \}^{\Omega \smallsetminus S}$ of the cube $[0,1]^\Omega$. The Izumi filtration $(I_j)_{0\leq j \leq |\Omega|+1}$ of $M_\phi^{\hat \otimes \Omega}$ has the following description: $I_j\subset M_\phi^{\hat \otimes \Omega}$ consists of legitimate sections defined on the cube $[0,1]^\Omega$ that vanish on the faces $F_S^1$ for all $S$ such that $|S|\leq j-1$. For example, $I_{1}$ consists of legitimate sections on $[0,1]^\Omega$ that vanish at $\{1\}^{\Omega}$. 

Recall (see \eqref{eq_Izumi_decomp}) that $I_j/I_{j+1}$ is canonically isomorphic to $\bigoplus_{S\subset \Omega, |S|=j} C_\phi^{\hat \otimes S} \hat\otimes B^{\hat \otimes \Omega \smallsetminus S}$. Analogously to $M_\phi^{\hat\otimes\Omega}$, each summand $C_\phi^{\hat \otimes S}\hat\otimes B^{\hat \otimes \Omega \smallsetminus S}$ is canonically identified as the algebra of legitimate sections defined on $F^1_S$, vanishing on $F^1_{S'}$ for all $S'\subsetneq S$. It follows that $I_j/I_{j+1}$ is canonically identified as the algebra of legitimate sections defined on the union $\cup_{|S|=j}F_S^1$, vanishing on $\cup_{|S|=j-1}F_S^1$.

For any $0\leq j \leq |\Omega|-1$, let $f \in I_j/I_{j+1}$ be a legitimate section defined on the union $\cup_{|S|=j}F_S^1$, vanishing on $\cup_{|S|=j-1}F_S^1$, we define an extension $\tilde f$, a legitimate section defined on the union $\cup_{|S|=j+1}F_S^1$ as follows. For any $S\subset \Omega$ with $|S|=j+1$, we define $\tilde{f}$ on $F_S^1=[0,1]^S\times \{1\}^{\Omega \smallsetminus S}$ by the formula
  \[
  \tilde{f}((t_s)_{s\in S}, (1)_{s\in \Omega \smallsetminus S}) \coloneq 
  \begin{cases}  
   ( \max_{s\in S} t_s)   f( (\frac{t_s}{   ( \max_{s\in S} t_s)})_{s\in S}, (1)_{s\in \Omega \smallsetminus S} )   &    \max_{s\in S} t_s  >0,  \\
    0 &    \max_{s\in S} t_s =0.
    \end{cases} 
  \]
The first expression is well-defined since at least one of the coordinates of $(\frac{t_s}{   ( \max_{s\in S} t_s)})_{s\in S}$ is $1$ so $( (\frac{t_s}{   ( \max_{s\in S} t_s)})_{s\in S}, (1)_{s\in \Omega \smallsetminus S} ) \in F_{S'}^1$ for some $S'$ with $|S'|=j$. For any $S'\subset \Omega$, the function $\tilde{f}$ takes values in $B^{\hat \otimes S'}\hat \otimes A^{\hat\otimes \Omega \smallsetminus S'}$ on $F_S^1\cap (0, 1]^{S'}\times \{0\}^{\Omega \smallsetminus S'}$. We see that $\tilde{f}$ is a legitimate section defined on the union $\cup_{|S|=j+1}F_S^1$, since for any $S'\subset \Omega$ the above extension formula restricts compatibly to the (same) formula that extends continuous $B^{\hat\otimes S'}\hat\otimes A^{\hat\otimes \Omega \smallsetminus S'}$-valued functions on $\{0\}^{S'}\times [0,1]^{\Omega \smallsetminus S'} \cap (\cup_{|S|=j}F_S^1)$ to continuous  $B^{\hat\otimes S'}\hat\otimes A^{\hat\otimes \Omega \smallsetminus S'}$-valued functions on $\{0\}^{S'}\times [0,1]^{\Omega \smallsetminus S'} \cap (\cup_{|S|=j+1} F^1_S)$.

Inductively, this defines an extension of $f\in I_{j}/I_{j+1}$ to a legitimate section defined on the cube $[0,1]^{\Omega}$. It is not hard to see that the map $ f \mapsto \tilde f$ defines a graded, $G\wr_\Omega F$-equivariant, c.c.p. section $I_j/I_{j+1} \to I_j$ as desired.
  \end{proof}
  
We recall that $\langle \mathcal{CI}_G \rangle$ is the localizing subcategory of $KK_G$ generated by the compactly induced algebras. We remind again that $KK_G$ consists of the separable \emph{ungraded} $G$-$C^*$-algebras.

\begin{lemma}\label{lem_wr_CI} Let $G$ be a second countable, locally compact group. For any $A\in \langle \mathcal{CI}_G \rangle$, for any second countable, compact group $F$, and for any finite $F$-set $\Omega$, $A^{\otimes \Omega} \in \langle\mathcal{CI}_{G\wr_\Omega F} \rangle$.
\end{lemma}
\begin{proof}
Suppose $A\cong \mathrm{Ind}_L^G(B)$ for a compact subgroup $L$ of $G$ and for a separable $L$-$C^*$-algebra $B$.  Then, $A^{\otimes \Omega}$ is isomorphic to the induced algebra $\mathrm{Ind}_{L\wr_\Omega F}^{G\wr_\Omega F}(B^{\otimes \Omega})$. It follows $\hat\otimes_\Omega$ sends any compactly induced object in $KK_G$ to a compactly induced object in $KK_{G\wr_\Omega F}$.

Let $\mathcal{C}$ be the full subcategory of $KK_G$ that consists of all $A$ such that $A^{\otimes \Omega} \in \langle\mathcal{CI}_{G\wr_\Omega F} \rangle$ for \emph{any} second countable, compact group $F$ and for \emph{any} finite $F$-set $\Omega$. We just saw that $\mathcal{CI}_{G} \subset  \mathcal{C}$. Our assertion follows once we show that $\mathcal{C}$ is a localizing subcategory of $KK_G$. Note that since $\hat\otimes_\Omega$ is a functor for any $F$ and $\Omega$, if $A$ is $KK_G$-equivalent to $B$, then $A\in \mathcal{C}$ if and only if $B\in \mathcal{C}$. 

Suppose we have an exact triangle $\Sigma A_2\to A_0\to A_1\to A_2$ in $KK_G$. Suppose that $A_0, A_2$ belong to $\mathcal{C}$. We show that $A_1$ belongs to $\mathcal{C}$. Any exact triangle in $KK_G$ is equivalent to an extension triangle \cite[Definition 2.3]{MN06} for a mapping cone/cylinder exact sequence $C_\phi\to M_\phi\to B$ for some $G$-$\ast$-homomorphism $\phi\colon B_0\to B$. In this case, $C_\phi$ and $B$ belong to $\mathcal{C}$, and we only need to show that $M_\phi$ belongs to $\mathcal{C}$. Let $F$ be any second countable compact group, and let $\Omega$ be any finite $F$-set. Consider the Izumi filtration $(I_j)_{0\leq j\leq |\Omega|+1}$ of $M_\phi^{\otimes \Omega}$ associated with the mapping cone/cylinder exact sequence. By Proposition \ref{prop_admissible_cylinder}, each short exact sequence 
\begin{equation}
0\to I_{j+1} \to I_{j} \to I_j/I_{j+1} \to 0
\end{equation}
is admissible (\cite[Definition 2.3]{MN06}) in $KK_{G\wr_\Omega F}$. We have $I_{|\Omega|}\cong C_\phi^{\otimes \Omega}$, $I_0\cong M_\phi^{\otimes \Omega}$, and $I_{j}/I_{j+1} \cong \bigoplus_{S\subset \Omega, |S|=j} C_\phi^{\otimes S} \otimes B^{\otimes \Omega \smallsetminus S}$ for all $0\leq j \leq |\Omega|$. Note that $I_j/I_{j+1}$ is the direct sum of the induced algebras of the form $\mathrm{Ind}^{G\wr_\Omega F}_{G\wr_\Omega F_S}( C_\phi^{\otimes S}\otimes B^{\otimes \Omega \smallsetminus S})$ for $S\subset \Omega$ ($|S|=j$) where $F_S\leq F$ is the (set-wise) stabilizer of $S$ in $F$. Thus, to show  $I_j/I_{j+1}$ belongs to $\langle \mathcal{CI} _{G\wr_\Omega F} \rangle$, we only need to show that $C_\phi^{\otimes S}\otimes B^{\otimes \Omega \smallsetminus S}$ belongs to $\langle{\mathcal{CI}_{G\wr_\Omega F_S} \rangle}$. Since $C_\phi$ and $B$ are in $\mathcal{C}$, we have $C_\phi^{\otimes S} \in \langle \mathcal{CI}_{G\wr_S F_S} \rangle $ and $B^{\otimes \Omega \smallsetminus S} \in \langle \mathcal{CI}_{G\wr_{\Omega \smallsetminus S} F_S} \rangle$. Hence, as the external tensor product, $C_\phi^{\otimes S} \otimes B^{\otimes \Omega \smallsetminus S} $ belongs to $\langle \mathcal{CI}_{ (G\wr_S F_S) \times (G\wr_{\Omega \smallsetminus S} F_S)} \rangle$. Here, we used that for any groups $G_1, G_2$, the (external) minimal tensor product bi-functor $KK_{G_1}\times KK_{G_2} \to KK_{G_1\times G_2}$ sends $\langle \mathcal{CI}_{G_1}\rangle \times \langle \mathcal{CI}_{G_2}\rangle$ to $\langle \mathcal{CI}_{G_1\times G_2}\rangle$, which follows by a simpler bootstrap argument (in one variable at a time). We also note that for any closed subgroup $G_4$ of a group $G_3$, the restriction functor $KK_{G_3}\to KK_{G_4}$ maps $\langle \mathcal{CI}_{G_3}\rangle$ to  $\langle \mathcal{CI}_{G_4}\rangle$. Combing all these, we see that  $C_\phi^{\otimes S}\otimes B^{\otimes \Omega \smallsetminus S}$ belong to $\langle  \mathcal{CI}_{G\wr_\Omega F_S} \rangle$. Therefore, $I_{j}/I_{j+1} \in  \langle \mathcal{CI}_{G\wr_\Omega F} \rangle$. Using the admissible extensions $I_{j+1}\to I_j \to I_{j}/I_{j+1}$, starting from the bottom $j=|\Omega|$, it follows $I_j$ belongs to $\langle \mathcal{CI}_{G\wr_\Omega F} \rangle$ for all $j$. In particular, $I_0\cong  M_\phi^{\otimes \Omega}$ belongs to $\langle \mathcal{CI} _{G\wr_\Omega F} \rangle$ for any $F$ and $\Omega$, showing $M_\phi \in \mathcal{C}$ as desired.

If $A\in \mathcal{C}$, the suspension $\Sigma A=C_0(\bR)\otimes A \in \mathcal{C}$. This is because $(\Sigma A)^{\otimes \Omega} \cong C_0(\bR)^{\otimes \Omega} \otimes A^{\otimes \Omega}$, and $\langle \mathcal{CI}_{G\wr_\Omega F} \rangle$ is stable under taking the minimal (and maximal) tensor product with any object.

Finally, we show that if $A_n \in \mathcal{C}$ for $n\geq1$, the countable direct sum $A=\bigoplus_{n\geq 1} A_n \in \mathcal{C}$. We have 
\[
A^{\otimes \Omega} \cong \bigoplus_{c\colon \Omega \to \mathbb{N}} \bigotimes_{\omega \in \Omega }A_{c(\omega)} 
\]
equipped with the natural $G \wr_\Omega F$ action. This is the direct sum of the algebra induced from $\bigotimes_{\omega \in \Omega }A_{c(\omega)}$.
With the same argument that we used to show $I_{j}/I_{j+1} \in  \langle \mathcal{CI}_{G\wr_\Omega F} \rangle$ above, we see that $A^{\otimes \Omega}\in \langle \mathcal{CI}_{G\wr_\Omega F} \rangle$. Therefore, $A \in \mathcal{C}$. This completes the proof.
\end{proof}

\begin{lemma}\label{lem_wr_conjugate} For any topological groups $G, F$, for any finite $F$-set $\Omega$, and for any element $x \in G \wr_\Omega F$, let $x_F\in F$ be the image of $x$ under the quotient map $G\wr_\Omega F\to F$. Let $F_x$ be the closed subgroup of $F$ generated by $x_F$. Then, $x$ is conjugate to an element $y$ of the form $((g_\omega)_\omega, x_F)\in (\prod_\Omega G) \rtimes F = G \wr_\Omega F $ such that $g_\omega\neq 1$ for at most one $\omega$ in each $F_x$-orbit of $\Omega$. Such an element $y$ is contained in the closed subgroup $(\prod_{O\in \Omega/F_x} \prod_{O}L_O)\rtimes F_x\leq (\prod_\Omega G) \rtimes F$ where for each $F_x$-orbit $O$ of $\Omega$, $L_O$ is the closed subgroup of $G$ generated by  $g_\omega$ for $\omega \in O$ (note that at most one $g_\omega$ for $\omega \in O$ is nontrivial).
\end{lemma}
\begin{proof} Consider the special case when $F$ is the cyclic group $C_n$ of order $n$ and $\Omega=\{0, \dots, n-1\}$ with a non-trivial $F$-action with a single orbit. Let $x=((g_0, \dots , g_{n-1}), x_F)\in G\wr_\Omega F$ and suppose $x_F$ acts on $\Omega$ as the cycle $(0, 1, \dots, n-1)$. Then, the element $(g_0^{-1}, 1, \dots, 1)\in \prod_\Omega G$ conjugates $x$ to $x_1=((1, g_1g_0, g_2, \dots, g_{n-1}), x_F)$. The element $(1, (g_1g_0)^{-1}, 1, \dots, 1)$ conjugates $x_1$ to $x_2=((1, 1, g_2g_1g_0, g_3, \dots, g_{n-1}), x_F)$. Continuing recursively, we  conjugate $x_{n-2}=((1, \dots, 1, g_{n-2}\cdots g_0, g_{n-1}), x_F)$ by the element $(1, \dots , (g_{n-2}\cdots g_0)^{-1},1)$ to obtain
$$x_{n-1}=((1, \dots, 1, g_{n-1}\cdots g_0), x_F),$$
as desired.
If $y=((1,\dots, 1, g), x_F)$, then, $y$ is contained in $(\prod_\Omega L) \rtimes F_x$ where $L$ is the closed subgroup of $G$ generated by $g$, and $F_x$ is the closed subgroup of $F$ generated by $x_F$ (which is $F$ in this case). The general case follows easily from the above special case. 
\end{proof}

\begin{proof}[Proof of Theorem \ref{thm_gamma_wr_KK}] 
Let $D\in KK_G(P, \bC)$ be a Dirac morphism for $G$, and let $\eta \colon KK_G(\bC, P)$ be a dual-Dirac morphism for $G$ such that $D\otimes_\bC \eta=1_P$. Let $\gamma_G=\eta \otimes_P D$ be the gamma element for $G$. Then, we have $\hat\otimes_\Omega(D)\colon P^{\otimes \Omega}\to \bC$ with $P^{\otimes \Omega} \in \langle \mathcal{CI}_{G\wr_\Omega F} \rangle $ by Lemma \ref{lem_wr_CI}. We have $(\hat\otimes_\Omega(D)) \otimes_ \bC(\hat\otimes_\Omega(\eta)) = \hat\otimes_\Omega( D\otimes_\bC \eta ) = 1_{P^{\otimes \Omega}}$. 

Let $z = (\hat\otimes_\Omega(\eta )) \otimes_{P^{\otimes \Omega}} (\hat\otimes_\Omega(D)) =   \hat\otimes_\Omega(\eta \otimes_P D ) = \hat\otimes_\Omega(\gamma_G)$ in $KK_{G\wr_\Omega F}(\bC, \bC)$. We show that $\mathrm{Res}_{G\wr_\Omega F}^L(z)=1_L$ for all compact subgroups $L$ of $G\wr_\Omega F$. By the character theory (the Peter--Weyl theory \cite{PW}) for the finite-dimensional representations of the compact groups, $\mathrm{Res}_{G\wr_\Omega F}^L(z)=1_L$ if and only if $\mathrm{Res}_{G\wr_\Omega F}^{L_x}(z)=1_{L_x}$ for any closed subgroup $L_x\leq L$ generated by an element $x\in L$. Here, we used that the characters of finite-dimensional representations define an injection $KK_L(\bC, \bC)\cong R(L)\to C(L)^{\mathrm{Ad}L}$, extending $[\pi] \mapsto \chi_\pi = (l\mapsto \mathrm{Tr}(\pi(l)))$ for the isomorphism classes of (unitary) finite-dimensional representations $\pi$. It is injective since $\chi_\pi$ for ${[\pi]\in \widehat L}$ (the isomorphism classes of the irreducible unitary representations of $L$) are nonzero orthogonal vectors in $L^2(L)$. By Lemma \ref{lem_wr_conjugate}, any element $x\in G\wr_\Omega F$ is conjugate to an element $y$ contained in the closed subgroup $L=(\prod_{O\in \Omega/F_x} \prod_{O}L_O)\rtimes F_x\leq (\prod_\Omega G) \rtimes F$ where $F_x$, $y$ and $L_O$ are as in Lemma \ref{lem_wr_conjugate}, and in our case where $x$ generates a compact subgroup of $G\wr_\Omega F$, $L_O$ are compact subgroups of $G$, and $F_x$ is a compact subgroup of $F$. Overall, to show $\mathrm{Res}_{G\wr_\Omega F}^L(z)=1_L$ for all compact subgroups $L$ of $G\wr_\Omega F$, we only need to show  $\mathrm{Res}_{G\wr_\Omega F}^L(z)=1_L$ for compact subgroups $L$ of the form $L=(\prod_{O\in \Omega/F_x} \prod_{O}L_O)\rtimes F_x\leq (\prod_\Omega G) \rtimes F$. Recall that
\[
T_0\coloneq\mathrm{Res}_{G\wr_\Omega F}^L \circ \hat\otimes_\Omega \colon KK_G \to KK_{G\wr_\Omega F} \to KK_L
\]
is the unique functor sending $A$ to $A^{\otimes \Omega}$, and $\phi\colon A\to B$ to $\phi^{\otimes \Omega}$. On the other hand, we have the following composition of functors
\[
T_1\colon KK_G  \to  \prod_{O\in \Omega/F_x}KK_{L_O} \to \prod_{O \in\Omega/F_x}KK_{L_O\wr_O F_x} \to KK_{L},
\]
where the first functor is the product of $\mathrm{Res}_G^{L_O}$, the second is the product of the equivariant power functors $\hat\otimes_O\colon KK_{L_O}\to KK_{L_O\wr_O F_x}$, and the last is the external minimal tensor product followed by the restriction via the inclusion $L\to \prod_{O\in \Omega/F_x}(L_O\wr_O F_x)$. The functor $T_1$ sends  $A$ to $\bigotimes_{O\in \Omega/F_x}(A^{\otimes O})$ and $\phi \colon A\to B$ to $\phi^{\otimes \Omega}$. The two functors $T_0$ and $T_1$ are naturally isomorphic via the canonical isomorphism $A^{\otimes \Omega} \cong \bigotimes_{O\in \Omega/F_x}(A^{\otimes O})$. Since $\mathrm{Res}_G^{L_O}(\gamma_G)=1_{L_O}$ for any compact subgroup $L_O$ of $G$, it follows that $\mathrm{Res}_{G\wr_\Omega F}^L(z)=1_L$ as desired. 

Overall, we have shown that $\hat\otimes_\Omega (D)\colon P^{\otimes \Omega} \to \bC$ is a weak equivalence in $KK_{G\wr_\Omega F}$ with $P^{\otimes \Omega} \in \langle \mathcal{CI}_{G\wr_\Omega F}\rangle$, with a weak inverse $\hat\otimes_\Omega(\eta)$. The rest follows immediately from this.
\end{proof}

Our remaining task is to show that the map $\hat\otimes_\Omega$ can be extended to $KK_{G, \ell}(A, B)$ so that we have the following commutative diagram:
\begin{equation}\label{eq_commute_wr_functor}
\xymatrix{
\hat\otimes_\Omega \colon & KK_{G}(A, B) \ar[d] \ar[r] &  KK_{G\wr_\Omega F}(A^{\hat\otimes \Omega}, B^{\hat\otimes \Omega}) \ar[d] \\
\hat\otimes_\Omega \colon  & KK_{G, \ell}(A, B) \ar[r] & KK_{G\wr_\Omega F, \tilde\ell}(A^{\hat\otimes \Omega}, B^{\hat\otimes \Omega}).
}
\end{equation}
For Theorem \ref{thm_gamma_wr} (and thus for Theorem \ref{thm_hyp_BCC_wr}), we only need this map for $A=B=\bC$.

For this, we only need to extend Theorem \ref{thm_tech_thm} slightly. The only difference is that we take a cycle from $KK_{G, \ell}(A, B)$, instead of $KK_G(A, B)$.

\begin{theorem}\label{thm_tech_thm_Gl} For $1\leq i\leq N$, let $G_i$ be a second countable, locally compact group, let $\ell_i$ be any length function on $G_i$, let $A_i, B_i$ be separable, graded, $G_i$-$C^*$-algebras, and let $(E_i, \pi_i, T_i)$ be a Kasparov cycle for $KK_{G_i, \ell_i}(A_i, B_i)$. Let $E=\hat\bigotimes_{1\leq i \leq N}E_i$ be the exterior tensor product (over $\bC$), which is a graded $\hat\bigotimes_{1\leq i \leq N}A_i$-$\hat\bigotimes_{1\leq i \leq N}B_i$-module with the $G$-action $\hat\bigotimes_{1\leq i \leq N} \pi_i$ where $G=\prod_{1\leq i\leq N}G_i$. We regard $T_i\in \Linears(E)$ via $\Linears(E_i) \subset \Linears(E)$. Let $\Delta$ be any separable subspace of $\Linears(E)$ consisting of elements $x\in \Linears(E)$ such that $[x, \Compacts(E_i)]\subset \Compacts(E_i) + \Compacts(E)$ for all $i$ where we regard $\Compacts(E_i)\subset \Linears(E)$ (in practice, we include $\Compacts(E_j)$, $A_j$ and $T_j$ in $\Delta$ for all $j$).

Then, there is a family $(N_i)_{1\leq i\leq N}$ of elements in $\Linears(E)$ satisfying the following conditions:
\begin{enumerate}
\item $N_i$ are $G$-continuous, even degree, positive elements in $\mathcal{L}(E)$; 
\item $1 =\sum_{1\leq i \leq N}N_i$;
\item $[N_i, N_{j}] \in \Compacts(E)$ for any $1\leq i, j\leq N$;
\item $g(N_i)-N_i \in \Compacts(E)$ for any $g\in G$;
\item $N_i \Compacts(E_i) \subset \Compacts(E)$;
\item $[N_i, \Delta] \subset \Compacts(E)$.
\end{enumerate}

If in addition, $G_i$, $A_i$, $B_i$, $E_i$, $\pi_i$, $T_i$ are independent of $i$, then $(N_i)_{1\leq i\leq N}$ can be taken to be $S_N$-equivariant, for the symmetric group $S_N$, in the sense that 
\begin{enumerate}
\item[(7)] $w(N_{w^{-1}i})=N_{i}$ for any $w\in S_N$ and for any $1\leq i \leq N$,
\end{enumerate}
for the canonical $S_N$-action on $E$.
\end{theorem}
\begin{proof} Note that the only difference from the setting of Theorem \ref{thm_tech_thm} is that the $G$-actions on $\Compacts(E_i)$, $\Compacts(E)$, are not necessarily $\ast$-preserving, but continuous and uniformly bounded on compact subsets of $G$. For any approximate unit $(v_n)_{n\geq1}$ of any separable $C^*$-algebra $A$ equipped with such a (not necessarily $\ast$-preserving) continuous $G$-action, there is an approximate unit $(u_n)_{n\geq1}$ such that $u_n$ is in the convex hull of $(v_m)_{m\geq n}$, and that $\lim_{n\to \infty}\|{g(u_n)-u_n}\| = 0$ uniformly on compact subsets of $G$. This follows by the same argument as in the case when $A$ is a $G$-$C^*$-algebra (c.f. \cite[Lemma 1.4]{Kas88}): for any compact subset $Y$ of $G$, let
\[
a_n\coloneq(g\mapsto g(v_n)-v_n) \in C(Y,A).
\]
Then the sequence $(a_n)_{n\geq1}$ in $C(Y, A)$ converges strictly to zero, and the sequence $(a_n)$ is bounded in norm. By the Hahn--Banach separation argument, it follows that the norm closure of the convex hull of $(a_m)_{m\geq n}$ contains the zero element for any $n\geq1$. It follows that for any $\epsilon>0$, there is $u_n$ in the convex hull of $(v_n)_{m\geq n}$ such that $\sup_{g\in Y}\|g(u_n)-u_n\|<\epsilon$.

With this observation in mind, the proof of Theorem \ref{thm_tech_thm} works verbatim. Note that the action of $S_N$ on $E$ is unitary, hence the action by $w\in S_N$ preserves positive elements.
\end{proof}

\begin{theorem}\label{thm_wr_KK_Gl} Let $G$ and $F$ be second countable locally compact groups. Let $\ell$ be a length function on $G$. Let $\Omega$ be a finite $F$-set $\Omega$. Let $\tilde \ell$ be the length function on $G\wr_\Omega F$ defined by $\tilde \ell(g) \coloneq \sum_{\omega \in \Omega}\ell(g_\omega)$ for $(g_\omega)_{\omega \in \Omega}$ in $\prod_\Omega G \subset G \wr_\Omega F$ and $\tilde \ell (f) \coloneq 0$ for $f \in F$.

For any separable graded $G$-$C^*$-algebras $A, B$, there is a map
\[
\hat\otimes_\Omega \colon KK_{G, \ell}(A, B) \to KK_{G\wr_\Omega F, \tilde \ell}(A^{\hat\otimes \Omega}, B^{\hat\otimes \Omega})
\]
 such that  the diagram \eqref{eq_commute_wr_functor} commutative.
\end{theorem}
\begin{proof} Using Theorem \ref{thm_tech_thm_Gl}, the proof of the first part of Theorem \ref{thm_KK_wr_functor} works verbatim: for any cycle $(E, \pi, T)$ in $E_{G,\ell}(A,B)$, we define the triple
\[
(E^{\hat\otimes\Omega}, \pi\wr_\Omega F,  \sum_{\omega \in \Omega}N^{1/2}_\omega T_\omega)
\]
as in the proof of Theorem \ref{thm_KK_wr_functor}. By Theorem \ref{thm_tech_thm_Gl}, there are operators $(N_\omega)_{\omega \in \Omega}$ satisfying conditions (1)-(7) in the proof of Theorem \ref{thm_KK_wr_functor}. The triple is a cycle for $E_{G\wr_\Omega F,\tilde\ell}(A^{\hat\otimes \Omega}, B^{\hat\otimes \Omega})$. The same proof shows that the homotopy class $[E^{\hat\otimes\Omega}, \pi\wr_\Omega F,  \sum_{\omega \in \Omega}N^{1/2}_\omega T_\omega]$ does not depend on the choice of $(N_\omega)_\omega$ and that it respects the homotopy of cycles so this induces a map
\[
KK_{G, \ell}(A, B) \to KK_{G\wr_\Omega F, \tilde\ell}(A^{\hat\otimes \Omega}, B^{\hat\otimes \Omega}).
\]
The commutativity of the diagram \eqref{eq_commute_wr_functor} is evident from the construction.
\end{proof}

\begin{proof}[Proof of Theorem \ref{thm_gamma_wr}]
The first part is Theorem \ref{thm_gamma_wr_KK}. Thus, $G\wr_\Omega F$ has a gamma element $\gamma_{G\wr_\Omega F} = \hat\otimes_\Omega(\gamma_G)$. By Theorem \ref{thm_wr_KK_Gl}, in particular, by the commutative diagram \eqref{eq_commute_wr_functor}, $\gamma_{G\wr_\Omega F}=1_{G\wr_\Omega F}$ in $KK_{G\wr_\Omega F, \tilde\ell}(\bC, \bC)$ provided $\gamma_{G}=1_{G}$ in $KK_{G,\ell}(\bC, \bC)$. 
\end{proof}

\section{BCC for relatively hyperbolic groups} \label{sec_CL_BCC}

Here, we combine our results to verify the Baum--Connes conjecture with coefficients with finite wreath products for a broad class of relatively hyperbolic groups. 

\begin{theorem}\label{thm_CLP_BCC_wr} Let $(G,\{H_\lambda\}_{\lambda \in \Lambda}, \{N_\lambda\}_{\lambda \in \Lambda})$ be a Cohen--Lyndon triple, and let $\bh_\lambda = H_\lambda/N_\lambda$, $\bg=G/\ll \mathcal{N} \rr$ where $\mathcal{N}=\bigcup_{\lambda\in\Lambda}N_{\lambda}$. Suppose that $\bg$ and $N_\lambda$ for all $\lambda \in \Lambda$ satisfy the BCC with finite wreath products. Then, $G$ satisfies the BCC with finite wreath products.
\end{theorem} 
\begin{proof} Consider the extension
\[
0 \to \ll \mathcal{N}   \rr \to G \to \bg \to 0.
\]
We have
\[
 \ll \mathcal{N}   \rr  = \Asterisk_{ \lambda \in \Lambda} \Asterisk_{t\in T_\lambda}t N_\lambda t^{-1},
\]
and this group satisfies the BCC with finite wreath products by the assumption on $N_\lambda$ and the permanence property \eqref{item_wr_free} in Theorem \ref{thm_inh_wr}. The assertion follows by combing this with the assumption on $\bg$ and the permanence property \eqref{item_wr_ext} in Theorem \ref{thm_inh_wr}.
\end{proof}

\begin{theorem}\label{thm_BCC_rel} Suppose that a countable discrete group $G$ is relatively hyperbolic to a family $\{H_\lambda\}_{\lambda \in \Lambda}$ of groups $H_\lambda$, with $|\Lambda|<\infty$. Suppose for any $\lambda \in \Lambda$, $H_\lambda$ is residually finite and satisfies the BCC with finite wreath products. Then, $G$ satisfies the BCC with finite wreath products.
\end{theorem}
\begin{proof} By Theorem \ref{thm. simple Dehn filling} (\cite[Theorem 7.19]{dahmani2017hyperbolically}), for any deep enough normal subgroups $N_\lambda$ of $H_\lambda$, $\bg=G/\ll\N\rr$ is relatively hyperbolic to the family $\{\bhl\}_{\lambda \in \Lambda}$, 
and we have an extension
\[
0\to \Asterisk_{\lambda\in \Lambda} \Asterisk_{t \in T_\lambda} N_\lambda^t \to G\to \bg \to 0,
\]
for some subsets $T_\lambda$ of $G$. Since $H_\lambda$ is residually finite, there are deep enough $N_\lambda$ so that $\bhl$ are finite, in which case $\bg$ is hyperbolic.

All hyperbolic groups satisfy the BCC with finite wreath products by Theorem \ref{thm_hyp_BCC_wr}. By the assumption on $H_\lambda$, all $N_\lambda$ satisfy the BCC with finite wreath products. By the permanence properties \eqref{item_wr_ext}, \eqref{item_wr_free} in Theorem \ref{thm_inh_wr}, $G$ satisfies the BCC with finite wreath products.
\end{proof}

\begin{remark} Bartels demonstrated in \cite[Corollary]{Bartels17} that the Farrell--Jones conjecture holds for any relatively hyperbolic group $G$ if the peripheral subgroups satisfy the Farrell--Jones conjecture (without assuming that the peripheral subgroups are residually finite). An analogous result for the Baum--Connes conjecture (with coefficients) lies far beyond the scope of this article. 
\end{remark}

Recall that Farb \cite[Theorem 5.1]{Farb98} (see also \cite{Eberlein80}, \cite{Osin06}) showed that the fundamental group of any (connected) complete, noncompact, finite-volume Riemannian manifold with pinched negative sectional curvatures, is relatively hyperbolic to the set $\{H_1, \ldots, H_r\}$ of cusp subgroups. Here, $H_i$ are finitely generated and virtually nilpotent. Any finitely generated virtually nilpotent group is residually finite and amenable. Thus, we immediately obtain:

\begin{corollary}\label{cor_negative_curvature} Let $G$ be the fundamental group of a complete, noncompact, finite-volume Riemannian manifold with pinched negative sectional curvatures. Then, $G$ satisfies the BCC with finite wreath products.
\end{corollary}

\begin{corollary}\label{cor_lattice_BCC} Any lattice of a simple Lie group of real rank one with finitely many connected components satisfies the BCC with finite wreath products.
\end{corollary}
\begin{proof} Let $\Gamma$ be a lattice of a simple Lie group $G$ of real rank one with finitely many connected components. By permanence property \eqref{item_wr_over}, it is enough to consider $\Gamma_0=\Gamma \cap G_0$ where $G_0$ is the connected component of $G$. By permanence property \eqref{item_wr_ext}, it is enough to consider $\overline{\Gamma}=\Gamma_0/\Gamma_0 \cap Z(G_0)$ which is a lattice of $G_0/Z(G_0)$. The lattice $\overline{\Gamma}$ is either hyperbolic (uniform lattice) or contains a finite index torsion free subgroup satisfying the assumption of Corollary \ref{cor_negative_curvature}. Hence, $\overline{\Gamma}$ satisfies the BCC with finite wreath products by Theorem \ref{thm_hyp_BCC_wr}, Corollary \ref{cor_negative_curvature} and the permanence property \eqref{item_wr_over} in Theorem \ref{thm_inh_wr}. 
\end{proof}

\begin{corollary}\label{cor_quot_rel_hyp} Suppose that a countable discrete group $G$ is relatively hyperbolic to a family $\{H_\lambda\}_{\lambda \in \Lambda}$ of groups $H_\lambda$ with $|\Lambda|<\infty$. Suppose for any $\lambda \in \Lambda$, $H_\lambda$ is finitely generated and virtually nilpotent. Then, there are finite subsets $F_\lambda$ of $H_\lambda \smallsetminus \{1\}$ such that for any normal subgroup $N_\lambda$ of $H_\lambda$ such that $N_\lambda \cap F_\lambda =\emptyset$, $\bg= G/\ll \mathcal{N}\rr $ satisfies the BCC with finite wreath products.
\end{corollary}
\begin{proof} By  \Cref{thm. simple Dehn filling},
for deep enough quotients, $\bg$ is hyperbolic relative to $\overline{\h}$, and $\bhl$ are finitely generated and virtually nilpotent, so Theorem \ref{thm_BCC_rel} applies to $\bg$.
\end{proof} 

The above result can be applied to Dehn fillings $M_T$ of complete finite-volume manifolds $M$ with pinched negative curvature. A precise definition of $M_T$ is given in \cite[\S 6.1]{sun2019cohomologyii}. In the special case where $M$ has toral cusps, the construction of $M_T$ has a simpler description (see e.g.\,\cite[\S 7.3]{PS_L2}).

\begin{corollary}\label{cor_Dehn_filling_mfld} Let $\overline{M}$ be a compact oriented $n$-manifold with nilmanifold boundary components  such that the centre of the fundamental group of each boundary component is of rank at least $2$. Suppose the interior of $\overline{M}$ admits a Riemannian metric with a complete pinched negative sectional curvature and finite volume.   If $M_T$  is a sufficiently deep Dehn filling manifold of $\overline{M}$, then $M_T$ is a closed oriented aspherical manifold and $\pi_1(M_T)$ satisfies the BCC with finite wreath products.
\end{corollary}
\begin{proof} Let $G=\pi_1(\overline{M})$. Then $G$ is hyperbolic relative to a finite collection of finitely generated nilpotent subgroups $\{H_i\}_{i=1}^m$ which are the fundamental groups of the boundary components of $\overline{M}$. By assumptions, the centre $Z(H_i)$ of each $H_i$ has rank at least 2. For each $i$, we take sufficiently deep normal subgroups $N_i\lhd Z(H_i)$, i.e.~avoiding the finite subsets in the assumptions of both \Cref{cor_quot_rel_hyp} and \Cref{prop CL} \eqref{prop CL_4}, and such that $\overline{H_i}$ is torsion-free. Then, by  \cite[Cor.~6.6]{sun2019cohomologyii}, $M_T$ is a closed oriented aspherical manifold and by \Cref{cor_quot_rel_hyp}, $\pi_1(M_T) =G/\ll \mathcal{N}\rr $ satisfies the BCC with finite wreath products.
\end{proof}

Next, we present an application of \Cref{cor_quot_rel_hyp} to closed aspherical manifolds that admit an Einstein metric.

\begin{corollary}\label{cor_Einstein_mfld} Let $M$ be a complete, finite-volume hyperbolic manifold of dimension at least three, with toral cusps. Then, any
  $M_T$ obtained by a sufficiently large Dehn filling of the cusps of $M$ is an Anderson-type Einstein manifold with fundamental group satisfying the BCC with finite wreath products.
\end{corollary}
\begin{proof}  The fundamental group of $M$ is hyperbolic relative the fundamental groups $\{H_i\}_{i=1}^m$ of the boundary tori. Following \cite[Definition 7.8]{PS_L2}, the fundamental group of any Dehn filling manifold $M_T$ of $M$ is $\pi_1(M)/\llrr{\bigcup_{i=1}^m N_i}$ where each $N_i$ is an infinite cyclic subgroup of $H_i$.

By Theorem 1.1 of \cite{And06}, any Dehn filling manifold $M_T$ obtained by sufficiently large Dehn filling of the cusps of $M$ admits an Einstein metric. This means that there is family of finite subsets $\{\mathcal F'_i\subset H_i\smallsetminus \{1\}\}_{i=1}^m$ such that if the family of subgroups $\{N_i<H_i\}_{i=1}^m$ induces a manifold filling $M_T$ of $M$ (equivalently, $\forall i, \, H_i/{N_i}$ is torsion-free) and satisfies $\forall i, \, N_i\cap \mathcal F'_i=\emptyset$, then $M_T$ admits an Einstein metric. 

By \Cref{cor_quot_rel_hyp}, there is family of finite subsets $\{\mathcal F''_i\subset H_i\smallsetminus \{1\}\}_{i=1}^m$ such that if the family of subgroups $\{N_i<H_i\}_{i=1}^m$ satisfies $N_i\cap \mathcal F''_i=\emptyset, \, \forall i$, then $\pi_1(M_T)$ satisfies the BCC with finite wreath products.

We define the family of finite subsets $\{\mathcal F_i:=\mathcal F'_i\cup F''_i\}_{i=1}^m$. It is clear that if the family of subgroups $\{N_i<H_i\}_{i=1}^m$ induces a manifold filling $M_T$ of $M$ and satisfies $N_i\cap \mathcal F_i=\emptyset, \, \forall i$, then $M_T$ admits an Einstein metric and $\pi_1(M_T)$ satisfies the BCC with finite wreath products. 
\end{proof}

Lastly, we present an application of \Cref{thm_BCC_rel} to quotients of the mapping class group of the closed genus 2 surface by sufficiently high powers of all Dehn twists. These types of quotients of mapping class groups of surfaces have recently attracted significant attention due to their hierarchically hyperbolic structure \cite{BHMS24, DHS21, BHS17}. We will need the following result from \cite{BHMS24}.

\begin{proposition}[{\cite[Cor.~7.4]{BHMS24}}]\label{prop: MCG quotients}There exists an integer $K_0 \geq 1$ so that for all non-zero multiples $K$ of $K_0$, the quotient
$\mathrm{MCG}(\Sigma_2) / \mathrm{DT}_K$
is hyperbolic relative to an infinite index subgroup commensurable with the product of two $C'\big(1/6\big)$-groups, where $\mathrm{DT}_K$ denotes the normal subgroup generated by $K$-th powers of all the Dehn twists.  
\end{proposition}

\begin{corollary} \label{cor: MCG quotients} There is an integer $K_0 \geq 1$ such that for all non-zero multiples $K$ of $K_0$, the quotient $\mathrm{MCG}(\Sigma_2) / \mathrm{DT}_K$ satisfies the BCC with finite wreath products.
\end{corollary} 

\begin{proof} Let $K_0 \geq 1$ be as in \Cref{prop: MCG quotients}. Then $\mathrm{MCG}(\Sigma_2) / \mathrm{DT}_K$ is hyperbolic relative to a subgroup $H$ commensurable with the product of two $C'\big(1/6\big)$-groups. 
Any $C'(1/6)$-group is residually finite. Indeed, by \cite[Theorem~1.2]{wise06}, such a group acts properly and cocompactly on a CAT(0) cube complex. By \cite[Theorem~1.1]{agol2013virtual}, every hyperbolic cubulated group is virtually special. Finally, by \cite{HW08}, virtually special groups embed virtually into right-angled Artin groups, which are residually finite.

Since  $C'\big(1/6\big)$-groups are residually finite  and  $H$ commensurable with a product of two such groups, it is residually finite. Moreover, since all $C'\big(1/6\big)$-groups are hyperbolic, they satisfy the BCC with finite wreath products by Theorem \ref{thm_hyp_BCC_wr}. Applying the permanence properties \eqref{item_wr_commen} and \eqref{item_wr_prod} of \Cref{thm_inh_wr}, it follows that $H$ satisfies the BCC with finite wreath products. The result now follows from \Cref{thm_BCC_rel}.
\end{proof}
\begin{remark}\label{rem: a-T-amen} Alternatively to the above argument, since any $C'\big(1/6\big)$-group acts properly on a CAT(0) cube complex, the peripheral subgroup $H$ of \Cref{prop: MCG quotients} does as well, and is therefore a-T-amenable \cite{NR97}. Hence, $H$ satisfies the BCC with finite wreath products \cite{HK01} (see Example \ref{ex_HK}).
\end{remark}

\begin{corollary} \label{cor: MCG BCC equal} There is an integer $K_0 \geq 1$ such that   $\mathrm{MCG}(\Sigma_2)$ satisfies the BCC with finite wreath products if and only if $\mathrm{DT}_K$ satisfies BCC with finite wreath products for some non-zero multiple $K$ of $K_0$.
\end{corollary}
\begin{proof} Let $K_0 \geq 1$ be as in \Cref{prop: MCG quotients} and $K$ be a non-zero multiple of $K_0$. By \Cref{cor: MCG quotients}, $\mathrm{MCG}(\Sigma_2) / \mathrm{DT}_K$ satisfies the BCC with finite wreath products. The result now follows from properties \eqref{item_wr_sub} and \eqref{item_wr_ext} of \Cref{thm_inh_wr}.    
\end{proof}

\bibliographystyle{amsalpha}
\bibliography{bib1}

\end{document}